\pdfoutput=1
\documentclass[12pt]{amsart}
\usepackage{lmodern}
\usepackage{ifthen}
\usepackage{amsfonts}
\usepackage{amsxtra}
\usepackage{amssymb}
\usepackage{mathdots}
\usepackage{array}
\usepackage[margin=1.0in]{geometry}
\usepackage{xcolor}
\definecolor{cite}{rgb}{0.30,0.60,1.00}
\definecolor{url}{rgb}{0.00,0.00,0.80}
\definecolor{link}{rgb}{0.40,0.10,0.20}
\usepackage[colorlinks,linkcolor=link,urlcolor=url,citecolor=cite,pagebackref,breaklinks]{hyperref}
\usepackage{bbm}
\usepackage{mathtools}
\usepackage{mathrsfs}
\usepackage{appendix}
\usepackage[all]{xy}
\usepackage[lite,abbrev,msc-links,alphabetic]{amsrefs}

\usepackage{graphicx}
\usepackage{multirow}
\usepackage{pstricks}
\usepackage{pst-pdf}
\usepackage{enumitem}
\usepackage{pifont}
\usepackage{marvosym}
\usepackage{tensor}

\usepackage[OT2,T1]{fontenc}
\DeclareSymbolFont{cyrletters}{OT2}{wncyr}{m}{n}
\DeclareMathSymbol{\Sha}{\mathalpha}{cyrletters}{"58}

\usepackage{graphicx}

\makeatletter
\providecommand*{\Dashv}{%
  \mathrel{%
    \mathpalette\@Dashv\vDash
  }%
}
\newcommand*{\@Dashv}[2]{%
  \reflectbox{$\m@th#1#2$}%
}
\makeatother

\setitemize[0]{leftmargin=20pt}
\setenumerate[0]{leftmargin=30pt}

\numberwithin{equation}{section}

\theoremstyle{plain}
\newtheorem{proposition}{Proposition}[section]
\newtheorem{conjecture}[proposition]{Conjecture}
\newtheorem{corollary}[proposition]{Corollary}
\newtheorem{lem}[proposition]{Lemma}
\newtheorem{theorem}[proposition]{Theorem}

\newtheorem{hypothesis}[proposition]{Hypothesis}

\theoremstyle{definition}
\newtheorem{definition}[proposition]{Definition}

\newtheorem{notation}[proposition]{Notation}

\newtheorem{assumption}[proposition]{Assumption}

\theoremstyle{remark}
\newtheorem{remark}[proposition]{Remark}

\renewcommand{\b}[1]{\mathbf{#1}}
\renewcommand{\c}[1]{\mathcal{#1}}
\renewcommand{\d}[1]{\mathbb{#1}}
\newcommand{\f}[1]{\mathfrak{#1}}
\renewcommand{\r}[1]{\mathrm{#1}}
\newcommand{\s}[1]{\mathscr{#1}}

\renewcommand{\(}{\left(}
\renewcommand{\)}{\right)}
\newcommand{\res}{\mathbin{|}}
\newcommand{\ol}[1]{\overline{#1}{}}

\newcommand{\ul}{\underline}
\renewcommand{\leq}{\leqslant}
\renewcommand{\geq}{\geqslant}

\newcommand{\bT}{\b T}

\newcommand{\bu}{\b u}

\newcommand{\bw}{\b w}

\newcommand{\cA}{\c A}

\newcommand{\cC}{\c C}

\newcommand{\cH}{\c H}

\newcommand{\cM}{\c M}
\newcommand{\cN}{\c N}

\newcommand{\cV}{\c V}

\newcommand{\cX}{\c X}
\newcommand{\cY}{\c Y}
\newcommand{\cZ}{\c Z}

\newcommand{\dA}{\d A}

\newcommand{\dC}{\d C}

\newcommand{\dF}{\d F}

\newcommand{\dL}{\d L}

\newcommand{\dP}{\d P}
\newcommand{\dQ}{\d Q}
\newcommand{\dR}{\d R}
\newcommand{\dS}{\d S}
\newcommand{\dT}{\d T}

\newcommand{\dV}{\d V}

\newcommand{\dZ}{\d Z}

\newcommand{\fD}{\f D}
\newcommand{\fE}{\f E}
\newcommand{\fF}{\f F}

\newcommand{\fI}{\f I}

\newcommand{\fZ}{\f Z}

\newcommand{\fg}{\f g}

\newcommand{\fm}{\f m}

\newcommand{\fp}{\f p}

\newcommand{\rD}{\r D}
\newcommand{\rE}{\r E}
\newcommand{\rF}{\r F}

\newcommand{\rH}{\r H}
\newcommand{\rI}{\r I}
\newcommand{\rJ}{\r J}
\newcommand{\rK}{\r K}

\newcommand{\rR}{\r R}

\newcommand{\rT}{\r T}
\newcommand{\rU}{\r U}

\newcommand{\rZ}{\r Z}

\newcommand{\rd}{\,\r d}
\newcommand{\re}{\r e}

\newcommand{\rh}{\r h}

\newcommand{\rs}{\r s}
\newcommand{\rt}{\r t}

\newcommand{\sI}{\s I}

\newcommand{\sO}{\s O}

\newcommand{\sS}{\s S}

\newcommand{\tR}{\mathtt{R}}
\newcommand{\tS}{\mathtt{S}}
\newcommand{\tT}{\mathtt{T}}

\newcommand{\tV}{\mathtt{V}}

\newcommand{\tc}{\mathtt{c}}

\newcommand{\tg}{\mathtt{g}}

\newcommand{\tw}{\mathtt{w}}

\newcommand{\biota}{\boldsymbol{\iota}}
\newcommand{\bpi}{\boldsymbol{\pi}}
\newcommand{\bbE}{\boldsymbol{E}}
\newcommand{\bbX}{\boldsymbol{X}}

\newcommand{\pres}[2]{\prescript{#1}{}{{\hskip-0.1em\relax}#2}}

\newcommand{\ac}{\r{ac}}

\newcommand{\CF}{\mathbbm{1}}

\newcommand{\cl}{\r{cl}}

\newcommand{\cusp}{\r{cusp}}

\newcommand{\dr}{\r{dR}}
\newcommand{\dtm}{\r{det}\:}

\newcommand{\fin}{\r{fin}}

\newcommand{\Herm}{\r{Herm}}

\newcommand{\id}{\r{id}}

\newcommand{\inert}{\r{int}}

\newcommand{\ram}{\r{ram}}
\newcommand{\reg}{\r{reg}}

\newcommand{\spl}{\r{spl}}

\DeclareMathOperator{\BC}{BC}
\DeclareMathOperator{\CH}{CH}

\DeclareMathOperator{\Diff}{Diff}

\DeclareMathOperator{\Gal}{Gal}

\DeclareMathOperator{\GL}{GL}

\DeclareMathOperator{\Hom}{Hom}
\DeclareMathOperator{\IM}{im}

\DeclareMathOperator{\Ind}{Ind}

\DeclareMathOperator{\Ker}{ker}

\DeclareMathOperator{\Lie}{Lie}

\DeclareMathOperator{\Mat}{Mat}

\DeclareMathOperator{\Nm}{Nm}

\DeclareMathOperator{\ord}{ord}

\DeclareMathOperator{\rank}{rank}
\DeclareMathOperator{\Res}{Res}
\DeclareMathOperator{\SCH}{SCH}
\DeclareMathOperator{\Sch}{Sch}

\DeclareMathOperator{\Spec}{Spec}
\DeclareMathOperator{\Spf}{Spf}

\DeclareMathOperator{\supp}{supp}

\DeclareMathOperator{\tr}{tr}
\DeclareMathOperator{\Tr}{Tr}

\DeclareMathOperator{\vol}{vol}

\begin{document}

\title{Chow groups and $L$-derivatives of automorphic motives for unitary groups}

\author{Chao Li}
\address{Department of Mathematics, Columbia University, New York NY 10027, United States}
\email{chaoli@math.columbia.edu}

\author{Yifeng Liu}
\address{Department of Mathematics, Yale University, New Haven CT 06511, United States}
\email{yifeng.liu@yale.edu}

\date{\today}

\begin{abstract}
  In this article, we study the Chow group of the motive associated to a tempered global $L$-packet $\pi$ of unitary groups of even rank with respect to a CM extension, whose global root number is $-1$. We show that, under some restrictions on the ramification of $\pi$, if the central derivative $L'(1/2,\pi)$ is nonvanishing, then the $\pi$-nearly isotypic localization of the Chow group of a certain unitary Shimura variety over its reflex field does not vanish. This proves part of the Beilinson--Bloch conjecture for Chow groups and $L$-functions, which generalizes the Birch and Swinnerton-Dyer conjecture. Moreover, assuming the modularity of Kudla's generating functions of special cycles, we explicitly construct elements in a certain $\pi$-nearly isotypic subspace of the Chow group by arithmetic theta lifting, and compute their heights in terms of the central derivative $L'(1/2,\pi)$ and local doubling zeta integrals. This confirms the conjectural arithmetic inner product formula proposed by one of us, which generalizes the Gross--Zagier formula to higher dimensional motives.
\end{abstract}

\maketitle

\setcounter{tocdepth}{1}
\tableofcontents

\section{Introduction}
\label{ss:introduction}

In 1986, Gross and Zagier \cite{GZ86} proved a remarkable formula that relates the N\'{e}ron--Tate heights of Heegner points on a rational elliptic curve to the central derivative of the corresponding Rankin--Selberg $L$-function. A decade later, Kudla \cite{Kud97} revealed another striking relation between Gillet--Soul\'{e} heights of special cycles on Shimura curves and derivatives of Siegel Eisenstein series of genus two, suggesting an arithmetic version of theta lifting and the Siegel--Weil formula (see, for example, \cites{Kud02, Kud03}). This was later further developed in his joint work with Rapoport and Yang \cite{KRY06}. For the higher dimensional case, in a series of papers starting from the late 1990s, Kudla and Rapoport developed the theory of special cycles on integral models of Shimura varieties for GSpin groups in lower rank cases and for unitary groups of arbitrary ranks \cites{KR11,KR14}. They also studied special cycles on the relevant Rapoport--Zink spaces over non-archimedean local fields. In particular, they formulated a conjecture relating the arithmetic intersection number of special cycles on the unitary Rapoport--Zink space to the first derivative of local Whittaker functions \cite{KR11}*{Conjecture~1.3}.

In his thesis work \cites{Liu11,Liu12}, one of us studied special cycles as elements in the Chow group of the unitary Shimura variety over its reflex field (rather than in the arithmetic Chow group of a certain integral model) and the Beilinson--Bloch height of the arithmetic theta lifting (rather than the Gillet--Soul\'{e} height). In particular, in the setting of unitary groups, he proposed an explicit conjectural formula for the Beilinson--Bloch height in terms of the central $L$-derivative and local doubling zeta integrals. Such formula is completely parallel to the Rallis inner product formula \cite{Ral84} which computes the Petersson inner product of the global theta lifting, hence was named \emph{arithmetic inner product formula} in \cite{Liu11}, and can be regarded as a higher dimensional generalization of the Gross--Zagier formula.\footnote{By ``generalization of the Gross--Zagier formula'', we simply mean that they are both formulae relating Beilinson--Bloch heights of special cycles and central derivatives of $L$-functions. However, from a representation-theoretical point of view, the more accurate generalization of the Gross--Zagier formula should be the arithmetic Gan--Gross--Prasad conjecture.} In the case of $\rU(1,1)$ over an arbitrary CM extension, such conjectural formula was completely confirmed in \cite{Liu12}, while the case for $\rU(r,r)$ with $r\geq 2$ is significantly harder. Recently, the Kudla--Rapoport conjecture has been proved by W.~Zhang and one of us in \cite{LZ}; and it has become possible to attack the cases for higher rank groups. In what follows, we will explain our new results on Chow groups of automorphic motives for unitary groups and the arithmetic inner product formula.

\subsection*{Beilinson--Bloch conjecture}

Let $\bbE$ be a number field and $\bbX$ a projective smooth scheme over $\bbE$ of odd dimension $2r-1$. We have the $L$-function $L(s,\rH^{2r-1}(\bbX\otimes_{\bbE}\ol{\bbE},\dQ_\ell(r)))$ for the middle degree $\ell$-adic cohomology of $\bbX$ for every rational prime $\ell$, which is conjectured to be meromorphic, independent of $\ell$, and satisfy a functional equation with center $s=0$. Let $\CH^r(\bbX)^0$ be the group of codimension $r$ Chow cycles on $\bbX$ that are homologically trivial (on $\bbX\otimes_{\bbE}\ol{\bbE}$). Then the unrefined Beilinson--Bloch conjecture (\cite{Bei87}*{Conjecture~5.9} and \cite{Blo84}) predicts that
\[
\rank \CH^r(\bbX)^0 =\ord_{s=0}L(s,\rH^{2r-1}(\bbX\otimes_{\bbE}\ol{\bbE},\dQ_\ell(r)))
\]
holds for every $\ell$, hence in particular, $\CH^r(\bbX)^0$ has finite rank. Note that when $\bbX$ is an elliptic curve, this recovers the (unrefined) Birch and Swinnerton-Dyer conjecture.

In fact, this conjecture can also be formulated in terms of Chow motives. Based on this point of view, we have an equivariant version of the Beilinson--Bloch conjecture as follows. Suppose that $\bbX$ admits an action of an algebra $\bT$ via \'{e}tale correspondences. Then $\bT$ acts on both $\CH^r(\bbX)^0$ and $\rH^{2r-1}(\bbX\otimes_{\bbE}\ol{\bbE},\dQ_\ell(r))$. Let $\bpi$ be a nonzero irreducible finite-dimensional complex representation of $\bT$. Then for every $\ell$ and every embedding $\dQ_\ell\hookrightarrow\dC$, we have the $L$-function
\[
L(s,\Hom_\bT(\bpi,\rH^{2r-1}(\bbX\otimes_{\bbE}\ol{\bbE},\dQ_\ell(r))_\dC).
\]
Then it is expected that
\begin{align}\label{eq:bb}
\dim_\dC\Hom_\bT(\bpi,\CH^r(\bbX)^0_\dC)=\ord_{s=0}L(s,\Hom_\bT(\bpi,\rH^{2r-1}(\bbX\otimes_{\bbE}\ol{\bbE},\dQ_\ell(r))_\dC)
\end{align}
holds, which can be regarded as the Beilinson--Bloch conjecture for the (conjectural Chow) motive $\Hom_\bT(\bpi,\rh^{2r-1}(\bbX)(r)_\dC)$ where $\rh^{2r-1}(\bbX)$ is the (conjectural Chow) motive of $\bbX$ of degree $2r-1$.

Now we propose a more specific conjecture for unitary Shimura varieties, guided by the equivariant version of the Beilinson--Bloch conjecture above.

Let $E/F$ be a CM extension of number fields with the complex conjugation $\tc$. Take an even positive integer $n=2r$. We equip $W_r\coloneqq E^n$ with the skew-hermitian form (with respect to the involution $\tc$) given by the matrix $\(\begin{smallmatrix}&1_r\\ -1_r &\end{smallmatrix}\)$. Put $G_r\coloneqq\rU(W_r)$, the unitary group of $W_r$, which is a quasi-split reductive group over $F$. For every non-archimedean place $v$ of $F$, we denote by $K_{r,v}\subseteq G_r(F_v)$ the stabilizer of the lattice $O_{E_v}^n$, which is a special maximal subgroup.

We first recall the notation of unitary Shimura varieties. Consider a totally positive definite \emph{incoherent} hermitian space $V$ over $\dA_E$ of rank $n$ with respect to the involution $\tc$,\footnote{Recall that ``incoherent'' means that $V$ is not the base change of a hermitian space over $E$.} and put $H\coloneqq\rU(V)$ for its unitary group. For every embedding $\iota\colon E\hookrightarrow\dC$, we have a system $\{\pres{\iota}{X_L}\}$ of Shimura varieties\footnote{When $F=\dQ$, we have to replace $\pres{\iota}{X_L}$ by its canonical smooth toroidal compactification.} of dimension $n-1$ over $\iota(E)$ indexed by open compact subgroups $L\subseteq H(\dA_F^\infty)$ (see Section \ref{ss:special} for more details).

\begin{conjecture}\label{co:bb}
Let $\pi$ be a tempered cuspidal automorphic representation of $G_r(\dA_F)$. For every totally positive definite incoherent hermitian space $V$ over $\dA_E$ of rank $n$, every irreducible admissible representation $\tilde\pi^\infty$ of $H(\dA_F^\infty)$, and every embedding $\iota\colon E\hookrightarrow\dC$, satisfying
\begin{enumerate}[label=(\alph*)]
  \item $\tilde\pi^\infty_v\simeq\pi_v$ for all but finitely many non-archimedean places $v$ of $F$ for which $\tilde\pi^\infty_v$ is unramified;

  \item $\Hom_{H(\dA_F^\infty)}\(\tilde\pi^\infty,\varinjlim_L\rH^{n-1}_{\dr}(\pres{\iota}{X_L}/\dC)\)\neq\{0\}$,
\end{enumerate}
the identity
\[
\dim_\dC\Hom_{H(\dA_F^\infty)}\(\tilde\pi^\infty,\varinjlim_L\CH^r(\pres{\iota}{X_L})^0_\dC\)=\ord_{s=\frac{1}{2}}L(s,\Pi_{j(\tilde\pi^\infty)})
\]
holds. Here, $\Pi_{j(\tilde\pi^\infty)}$ is the cuspidal factor of the automorphic base change of $\pi$ determined by $\tilde\pi^\infty$ (see Lemma \ref{le:arthur}).
\end{conjecture}

In relation with \eqref{eq:bb}, we take $\bbE=\iota(E)$, $\bbX=\pres{\iota}{X_L}$ for $L$ such that $(\tilde\pi^\infty)^L\neq 0$, $\bT$ to be the Hecke algebra of $H(\dA_F^\infty)$ of level $L$, and $\bpi=(\tilde\pi^\infty)^L$. Moreover, in this case we know that the $L$-function on the right-hand side of \eqref{eq:bb} coincides with $L(s,\Pi_{j(\tilde\pi^\infty)})$ up to a shift by $\tfrac{1}{2}$. Thus, Conjecture \ref{co:bb} is a special case of \eqref{eq:bb} after taking limit of $L$.

In the case when $\{\pres{\iota}{X_L}\}$ is replaced by classical modular curves, Conjecture \ref{co:bb} in fact recovers the (unrefined) Birch and Swinnerton-Dyer conjecture for rational elliptic curves. See \cite{Gro04}*{Section~22} for more details from this point of view. Conjecture \ref{co:bb} was only known in the case of modular/Shimura curves when the analytic rank is at most $1$ \cites{GZ86,Kol90,Nek08,YZZ}, and partially known in the case of Shimura varieties for $\rU(2)\times\rU(3)$ when the analytic rank is exactly $1$ \cite{Xue19}.\footnote{Interestingly, the height formula in \cite{Xue19}, which is for the endoscopic case, is obtained by reducing it to the arithmetic inner product formula for $\rU(1,1)$.}

\begin{remark}\label{re:history}
The notion of incoherent spaces was first invented by Kudla (in the quadratic case) which he called an incoherent collection of quadratic spaces over local fields \cite{Kud97}*{Definition~2.1}. Around the similar time, Gross realized that a Shimura curve can be uniformized at its supersingular points in terms of a collection of quaternion algebras over the base number field (see \cite{Gro04} and also \cite{GGP12} for generalizations). In their work \cite{YZZ}, Yuan, S.~Zhang, and W.~Zhang put this infinite collection of quaternion algebras as a single quaternion algebra over the ad\`{e}les as a uniform description of the geometry of Shimura curves and the representation theory. This viewpoint was later adapted by W.~Zhang \cite{Zha12} and one of us \cites{Liu11,Liu12}. In \cite{Zha19} (which is based on his 2010 talk at Gross birthday conference), S.~Zhang summarizes how one can use the notion of incoherent quadratic/hermitian spaces to formulate various conjectures that are arithmetic counterparts of classical period formulae. In particular, there should exists a compatible system of varieties $\{X_L\}$ over $E$ such that for every embedding $\iota\colon E\hookrightarrow\dC$, we have an isomorphism between $\{X_L\otimes_E\iota(E)\}$ and $\{\pres{\iota}{X_L}\}$ -- this was explained in more details in \cite{Gro}. Based on this observation, Conjecture \ref{co:bb} can be formulated for the family $\{X_L\}$ without choosing $\iota$.
\end{remark}

\subsection*{Main results}

Our main results in this article prove part of Conjecture \ref{co:bb} under certain assumptions on $E/F$ and $\pi$. Denote by $\tV_F^{(\infty)}$ and $\tV_F^\fin$ the set of archimedean and non-archimedean places of $F$, respectively; and $\tV_F^\spl$, $\tV_F^\inert$, and $\tV_F^\ram$ the subsets of $\tV_F^\fin$ of those that are split, inert, and ramified in $E$, respectively. For every $v\in\tV_F^\fin$, we denote by $q_v$ the residue cardinality of $F_v$.

\begin{assumption}\label{st:main}
Suppose that $\tV_F^\ram=\emptyset$ and that $\tV_F^\spl$ contains all 2-adic places. In particular, $[F:\dQ]$ is even. We consider a cuspidal automorphic representation $\pi$ of $G_r(\dA_F)$ realized on a space $\cV_\pi$ of cusp forms, satisfying:
\begin{enumerate}
  \item For every $v\in\tV_F^{(\infty)}$, $\pi_v$ is the holomorphic discrete series representation of Harish-Chandra parameter $\{\tfrac{n-1}{2},\tfrac{n-3}{2},\dots,\tfrac{3-n}{2},\tfrac{1-n}{2}\}$.

  \item For every $v\in\tV_F^\spl$, $\pi_v$ is a principal series.

  \item For every $v\in\tV_F^\inert$, $\pi_v$ is either unramified or almost unramified (see Remark \ref{re:almost_unramified} below) with respect to $K_{r,v}$; moreover, if $\pi_v$ is almost unramified, then $v$ is unramified over $\dQ$.

  \item For every $v\in\tV_F^\fin$, $\pi_v$ is tempered.
\end{enumerate}
\end{assumption}

\begin{remark}\label{re:almost_unramified}
We have the following remarks concerning Assumption \ref{st:main}.
\begin{enumerate}
  \item In (1), by \cite{Sch75}*{Theorem~1.3}, the condition for $\pi_v$ is equivalent to that $\pi_v$ is a discrete series representation whose restriction to $K_{r,v}$ contains the character $\kappa_{r,v}^r$ (see Notation \ref{st:g}(G5,G6) for the notation). Moreover, one can also describe $\pi_v$ as the theta lifting of the trivial representation of the (positive) definite unitary group of rank $n$ (see, for example, \cite{KK07}).

  \item Part (2) will only be used in the proof of Lemma \ref{le:split_tempered} in order to quote a vanishing result from \cite{CS17}. However, in our second article on this subject \cite{LL}, we have successively removed this assumption by confirming the conjecture in Remark \ref{re:split_tempered} by proving a stronger vanishing property.

  \item In (3), the notion of almost unramified representations of $G_r(F_v)$ at $v\in\tV_F^\inert$ is defined in \cite{Liu20}*{Definition~5.3}. Roughly speaking, an irreducible admissible representation $\pi_v$ of $G_r(F_v)$ is almost unramified (with respect to $K_{r,v}$) if $\pi_v^{I_{r,v}}$ contains a particular character as a module over $\dC[I_{r,v}\backslash K_{r,v}/I_{r,v}]$, where $I_{r,v}$ is an Iwahori subgroup contained in $K_{r,v}$, and that the Satake parameter of $\pi_v$ contains the pair $\{q_v,q_v^{-1}\}$; it is \emph{not} unramified. By \cite{Liu20}*{Theorem~1.2}, when $q_v$ is odd, \emph{almost} unramified representations are exactly those representations whose local theta lifting to the non-quasi-split unitary group of the same rank $2r$ has nonzero invariants under the stabilizer of an \emph{almost} self-dual lattice.

  \item In fact, if one uses the endoscopic classification for automorphic representations of unitary groups \cites{Mok15,KMSW} and \cite{Car12}*{Theorem~1.2}, then condition (4) in Assumption \ref{st:main} will be implied by condition (1).
\end{enumerate}
\end{remark}

Suppose that we are in Assumption \ref{st:main}. Denote by
\begin{itemize}
  \item $L(s,\pi)$ the doubling $L$-function (see Definition \ref{de:lfunction} for the more precise definition),

  \item $\tR_\pi\subseteq\tV_F^\spl$ the (finite) subset for which $\pi_v$ is ramified,

  \item $\tS_\pi\subseteq\tV_F^\inert$ the (finite) subset for which $\pi_v$ is almost unramified.
\end{itemize}
Then we have $\varepsilon(\pi)=(-1)^{r[F:\dQ]+|\tS_\pi|}$ for the global (doubling) root number, so that the vanishing order of $L(s,\pi)$ at the center $s=\tfrac{1}{2}$ has the same parity as $|\tS_\pi|$ since $[F:\dQ]$ is even. The cuspidal automorphic representation $\pi$ determines a hermitian space $V_\pi$ over $\dA_E$ of rank $n$ via local theta dichotomy (such that the local theta lifting of $\pi_v$ to $\rU(V_\pi)(F_v)$ is nontrivial for every place $v$ of $F$), unique up to isomorphism, which is totally positive definite and satisfies that for every $v\in\tV_F^\fin$, the local Hasse invariant $\epsilon(V_\pi\otimes_{\dA_F}F_v)=1$ if and only if $v\not\in\tS_\pi$ (see Proposition \ref{pr:uniqueness}(2)).

Now suppose that $|\tS_\pi|$ is odd hence $\varepsilon(\pi)=-1$, which is equivalent to that $V_\pi$ is incoherent. In what follows, we take $V=V_\pi$ in the context of Conjecture \ref{co:bb}, hence $H=\rU(V_\pi)$. Let $\tR$ be a finite subset of $\tV_F^\fin$. We fix a special maximal subgroup $L^\tR$ of $H(\dA_F^{\infty,\tR})$ that is the stabilizer of a lattice $\Lambda^\tR$ in $V\otimes_{\dA_F}\dA_F^{\infty,\tR}$ (see Notation \ref{st:h}(H6) for more details). For a field $\dL$, we denote by $\dT^\tR_\dL$ the (abstract) Hecke algebra $\dL[L^\tR\backslash H(\dA_F^{\infty,\tR})/L^\tR]$, which is a commutative $\dL$-algebra. When $\tR$ contains $\tR_\pi$, the cuspidal automorphic representation $\pi$ gives rise to a character
\[
\chi_\pi^\tR\colon\dT^\tR_{\dQ^\ac}\to\dQ^\ac,
\]
where $\dQ^\ac$ denotes the subfield of $\dC$ of algebraic numbers; and we put $\fm_\pi^\tR\coloneqq\Ker\chi_\pi^\tR$, which is a maximal ideal of $\dT^\tR_{\dQ^\ac}$.

In what follows, we will fix an arbitrary embedding $\biota\colon E\hookrightarrow\dC$ and suppress it in the notation ${}^{\biota\!}X$ of unitary Shimura varieties. The following is the first main theorem of this article.

\begin{theorem}\label{th:main}
Let $(\pi,\cV_\pi)$ be as in Assumption \ref{st:main} with $|\tS_\pi|$ odd, for which we assume Hypothesis \ref{hy:galois}. If $L'(\tfrac{1}{2},\pi)\neq 0$, that is, $\ord_{s=\frac{1}{2}}L(s,\pi)=1$, then as long as $\tR$ satisfies $\tR_\pi\subseteq\tR$ and $|\tR\cap\tV_F^\spl|\geq 2$, the nonvanishing
\[
\varinjlim_{L_\tR}\(\CH^r(X_{L_\tR L^\tR})^0_{\dQ^\ac}\)_{\fm_\pi^\tR}\neq\{0\}
\]
holds, where the colimit is taken over all open compact subgroups $L_\tR$ of $H(F_\tR)$.
\end{theorem}

\begin{remark}\label{re:main}
We have the following remarks concerning Theorem \ref{th:main}.
\begin{enumerate}
  \item For every $v\in\tV_F^\fin$, the local doubling $L$-function $L(s,\pi_v)$ coincides with $L(s,\BC(\pi_v))$ where $\BC(\pi_v)$ denotes the standard base change of $\pi_v$ to $\GL_n(E_v)$ (see Remark \ref{re:doubling} for more details). In particular, combining with the local-global compatibility \cite{KMSW}*{Theorem~1.7.1}, we know that $L(s,\pi)$ coincides with the standard $L$-function of the automorphic base change of $\pi$.

  \item Since $|\tS_\pi|$ is odd, by (1) and Remark \ref{re:root}, Conjecture \ref{co:bb} predicts the nonvanishing
       \[
       \varinjlim_{L_\tR}\CH^r(X_{L_\tR L^\tR})^0_{\dQ^\ac}[\fm_\pi^\tR]\neq\{0\}
       \]
       when $\ord_{s=\frac{1}{2}}L(s,\pi)=1$ (by considering $\tilde\pi^\infty$ as the theta lifting of $\pi^\infty$), which further implies the nonvanishing in our statement. However, it is conjectured that $\CH^r(X_{L_\tR L^\tR})^0_{\dQ^\ac}$ is finite dimensional, which implies that the two types of nonvanishing are equivalent. Thus, our theorem provides evidence toward Conjecture \ref{co:bb}. See Theorem \ref{th:aipf}(2) below for a stronger result under an extra hypothesis.

  \item Hypothesis \ref{hy:galois} describes the Galois representation on the $\pi$-nearly isotypic subspace of the middle degree $\ell$-adic cohomology $\varinjlim_L\rH^{2r-1}(X_L\otimes_E\overline{E},\ol\dQ_\ell)$. See Remark \ref{re:galois} for the status of this hypothesis.

  \item In fact, the nonvanishing property we prove is that
      \[
      \varinjlim_{L_\tR}\(\SCH^r(X_{L_\tR L^\tR})^0_{\dQ^\ac}\)_{\fm_\pi^\tR}\neq\{0\},
      \]
      where $\SCH^r(X_{L_\tR L^\tR})^0$ denotes the subgroup of $\CH^r(X_{L_\tR L^\tR})^0$ generated by special cycles (recalled in Section \ref{ss:special}).

  \item It is clear that the field $\dQ^\ac$ in the statement of the theorem can be replaced by an arbitrary subfield over which $\pi^\infty$ (hence $\chi_\pi^\tR$) is defined.


  \item The main reason we assume $\tV_F^\ram=\emptyset$ is that the local ingredient \cite{LZ} only deals with places that are inert in $E$; and we hope to remove this assumption in the future.
\end{enumerate}
\end{remark}

Our remaining results rely on Hypothesis \ref{hy:modularity} on the modularity of Kudla's generating functions of special cycles, hence are conditional at this moment (see Remark \ref{re:modularity}).

\begin{theorem}\label{th:aipf}
Let $(\pi,\cV_\pi)$ be as in Assumption \ref{st:main} with $|\tS_\pi|$ odd, for which we assume Hypothesis \ref{hy:galois}. Assume Hypothesis \ref{hy:modularity} on the modularity of generating functions of codimension $r$.
\begin{enumerate}
  \item For every test vectors
     \begin{itemize}
       \item $\varphi_1=\otimes_v\varphi_{1v}\in\cV_{\pi}$ and $\varphi_2=\otimes_v\varphi_{2v}\in\cV_{\pi}$ such that for every $v\in\tV_F^{(\infty)}$, $\varphi_{1v}$ and $\varphi_{2v}$ have the lowest weight and satisfy $\langle\varphi_{1v}^\tc,\varphi_{2v}\rangle_{\pi_v}=1$,

       \item $\phi^\infty_1=\otimes_v\phi^\infty_{1v}\in\sS(V^r\otimes_{\dA_F}\dA_F^\infty)$ and $\phi^\infty_2=\otimes_v\phi^\infty_{2v}\in\sS(V^r\otimes_{\dA_F}\dA_F^\infty)$,
     \end{itemize}
     the identity
     \[
     \langle\Theta_{\phi^\infty_1}(\varphi_1),\Theta_{\phi^\infty_2}(\varphi_2)\rangle_{X,E}^\natural=
     \frac{L'(\tfrac{1}{2},\pi)}{b_{2r}(0)}\cdot C_r^{[F:\dQ]}
     \cdot\prod_{v\in\tV_F^\fin}\fZ^\natural_{\pi_v,V_v}(\varphi^\tc_{1v},\varphi_{2v},\phi_{1v}^\infty\otimes(\phi_{2v}^\infty)^\tc)
     \]
     holds. Here,
     \begin{itemize}
       \item $\Theta_{\phi^\infty_i}(\varphi_i)\in\varinjlim_L\CH^r(X_L)^0_\dC$ is the arithmetic theta lifting (Definition \ref{de:arithmetic_theta}), which is only well-defined under Hypothesis \ref{hy:modularity};

       \item $\langle\Theta_{\phi^\infty_1}(\varphi_1),\Theta_{\phi^\infty_2}(\varphi_2)\rangle_{X,E}^\natural$ is the normalized height pairing (Definition \ref{de:natural}),\footnote{Strictly speaking, $\langle\Theta_{\phi^\infty_1}(\varphi_1),\Theta_{\phi^\infty_2}(\varphi_2)\rangle_{X,E}^\natural$ relies on the choice of a rational prime $\ell$ and is a priori an element in $\dC\otimes_\dQ\dQ_\ell$. However, the above identity implicitly says that it belongs to $\dC$ and is independent of the choice of $\ell$.} which is constructed based on Beilinson's notion of height pairing;

       \item $b_{2r}(0)$ is defined in Notation \ref{st:f}(F4), which equals $L(M_r^\vee(1))$ where $M_r$ is the motive associated to $G_r$ by Gross \cite{Gr97}, and is in particular a positive real number;

       \item $C_r=(-1)^r2^{r(r-1)}\pi^{r^2}\frac{\Gamma(1)\cdots\Gamma(r)}{\Gamma(r+1)\cdots\Gamma(2r)}$, which is the exact value of a certain archimedean doubling zeta integral; and

       \item $\fZ^\natural_{\pi_v,V_v}(\varphi^\tc_{1v},\varphi_{2v},\phi_{1v}^\infty\otimes(\phi_{2v}^\infty)^\tc)$ is the normalized local doubling zeta integral (see Section \ref{ss:doubling}), which equals $1$ for all but finitely many $v$.
     \end{itemize}

  \item In the context of Conjecture \ref{co:bb}, take ($V=V_\pi$ and) $\tilde\pi^\infty$ to be the theta lifting of $\pi^\infty$ to $H(\dA_F^\infty)$. If $L'(\tfrac{1}{2},\pi)\neq 0$, that is, $\ord_{s=\frac{1}{2}}L(s,\pi)=1$, then
     \[
     \Hom_{H(\dA_F^\infty)}\(\tilde\pi^\infty,\varinjlim_{L}\CH^r(X_L)^0_\dC\)\neq\{0\}
     \]
     holds.
\end{enumerate}
\end{theorem}

\begin{remark}\label{re:aipf}
We have the following remarks concerning Theorem \ref{th:aipf}.
\begin{enumerate}
  \item Part (1) verifies the so-called \emph{arithmetic inner product formula}, a conjecture proposed by one of us \cite{Liu11}*{Conjecture~3.11}.

  \item The arithmetic inner product formula in part (1) is perfectly parallel to the classical Rallis inner product formula. In fact, suppose that $V$ is totally positive definite but \emph{coherent}. We have the classical theta lifting $\theta_{\phi^\infty}(\varphi)$ where we use standard Gaussian functions at archimedean places. Then the Rallis inner product formula in this case reads as
      \[
      \langle\theta_{\phi^\infty_1}(\varphi_1),\theta_{\phi^\infty_2}(\varphi_2)\rangle_H=
      \frac{L(\tfrac{1}{2},\pi)}{b_{2r}(0)}\cdot C_r^{[F:\dQ]}
      \cdot\prod_{v\in\tV_F^\fin}\fZ^\natural_{\pi_v,V_v}(\varphi^\tc_{1v},\varphi_{2v},\phi_{1v}^\infty\otimes(\phi_{2v}^\infty)^\tc),
      \]
      in which $\langle\;,\;\rangle_H$ denotes the Petersson inner product with respect to the \emph{Tamagawa measure} on $H(\dA_F)$.

  \item In part (2), the representation $\tilde\pi^\infty$ satisfies (a) of Conjecture \ref{co:bb}. By Remark \ref{re:main}(1) and Remark \ref{re:root}, if $\ord_{s=\frac{1}{2}}L(s,\pi)=1$, then $\tilde\pi^\infty$ satisfies (b) of Conjecture \ref{co:bb} as well, and $\Pi_{j(\tilde\pi^\infty)}$ is the unique cuspidal factor of the automorphic base change of $\pi$ such that $\ord_{s=\frac{1}{2}}L(s,\Pi_{j(\tilde\pi^\infty)})=1$. In particular, part (2) provides evidence toward Conjecture \ref{co:bb}, which is more direct than Theorem \ref{th:main} (but is conditional on the modularity of generating functions).
\end{enumerate}
\end{remark}

In the case where $\tR_\pi=\emptyset$, that is, $\pi_v$ is either unramified or almost unramified for every $v\in\tV_F^\fin$, we have a very explicit height formula for test vectors that are new everywhere.

\begin{corollary}\label{co:aipf}
Let $(\pi,\cV_\pi)$ be as in Assumption \ref{st:main} with $|\tS_\pi|$ odd, for which we assume Hypothesis \ref{hy:galois}. Assume Hypothesis \ref{hy:modularity} on the modularity of generating functions of codimension $r$. In the situation of Theorem \ref{th:aipf}(1), suppose further that
\begin{itemize}
  \item $\tR_\pi=\emptyset$;

  \item $\varphi_1=\varphi_2=\varphi\in\cV_\pi^{[r]\emptyset}$ (see Notation \ref{st:g}(G8) for the precise definition of the one-dimensional space $\cV_\pi^{[r]\emptyset}$ of holomorphic new forms) such that for every $v\in\tV_F$, $\langle\varphi_v^\tc,\varphi_v\rangle_{\pi_v}=1$; and

  \item $\phi^\infty_1=\phi^\infty_2=\phi^\infty$ such that for every $v\in\tV_F^\fin$, $\phi^\infty_v=\CF_{(\Lambda_v^\emptyset)^r}$.
\end{itemize}
Then the identity
\[
\langle\Theta_{\phi^\infty}(\varphi),\Theta_{\phi^\infty}(\varphi)\rangle_{X,E}^\natural=(-1)^r\cdot
\frac{L'(\tfrac{1}{2},\pi)}{b_{2r}(0)}\cdot C_r^{[F:\dQ]}
\cdot\prod_{v\in\tS_\pi}\frac{q_v^{r-1}(q_v+1)}{(q_v^{2r-1}+1)(q_v^{2r}-1)}
\]
holds.
\end{corollary}

\begin{remark}
Assuming the conjecture on the injectivity of the \'{e}tale Abel--Jacobi map, one can show that the cycle $\Theta_{\phi^\infty}(\varphi)$ is a primitive cycle of codimension $r$. By \cite{Bei87}*{Conjecture~5.5}, we expect that $(-1)^r\langle\Theta_{\phi^\infty}(\varphi),\Theta_{\phi^\infty}(\varphi)\rangle_{X,E}^\natural\geq 0$ holds, which, in the situation of Corollary \ref{co:aipf}, is equivalent to $L'(\tfrac{1}{2},\pi)\geq 0$.
\end{remark}

\subsection*{Strategy and structure}

The main strategy for the proofs of our main results is to adopt Beilinson's notion of height pairing together with various sophisticated uses of Hecke operators. In \cite{Bei87}, Beilinson constructed, under certain assumptions, a (hermitian) height pairing on $\CH^r(X_L)^0_\dC$ valued in $\dC$. Since those assumptions have not been resolved even today, we are not able to use the full notion of this height pairing. However, after choosing a sufficiently large prime $\ell$, Beilinson's construction gives an \emph{unconditional} height pairing on a subspace $\CH^r(X_L)^{\langle\ell\rangle}_\dC$ (a priori depending on $\ell$) of $\CH^r(X_L)^0_\dC$ valued in $\dC\otimes_\dQ\dQ_\ell$.

The candidates for those nonvanishing elements in Theorem \ref{th:main} are Kudla's special cycles $Z_T(\phi^\infty)$ (which will be recalled in Section \ref{ss:special}), which are in general elements in $\CH^r(X_L)_\dC$. We show that there exists an element $\rs\in\dT^\tR_{\dQ^\ac}\setminus\fm_\pi^\tR$ such that $\rs^*$ annihilates the quotient space $\CH^r(X_L)_\dC/\CH^r(X_L)^{\langle\ell\rangle}_\dC$. The existence of such element allows us to consider the modified cycles $\rs^*Z_T(\phi^\infty)$ without changing their (non)triviality in the localization of $\CH^r(X_L)_\dC$ at $\fm_\pi^\tR$, moreover at the same time to talk about their heights.

More precisely, we consider two such modified cycles $\rs_1^*Z_{T_1}(\phi^\infty_1)$ and $\rs_2^*Z_{T_2}(\phi^\infty_2)$. When $\phi^\infty_1\otimes\phi^\infty_2$ satisfies a certain regularity condition, the two cycles have disjoint support, hence their height pairing (in the sense of Beilinson) has a decomposition into so-called \emph{local indices} according to places $u$ of $E$. We mention especially that if $u$ is non-archimedean, then the local index at $u$ is defined via a winding number on the $\ell$-adic cohomology of $X_L\otimes_EE_u$, which a priori has nothing to do with intersection theory. When $X_L\otimes_EE_u$ has a smooth integral model, it is well-known that such winding number can be computed as the intersection number of integral extensions of the cycles. However, when $X_L\otimes_EE_u$ does not have smooth reduction, there is no general way to compute the local index. Nevertheless, we show that, under certain assumptions on the ramification and on the representation $\pi$, the local index between $\rs_1^*Z_{T_1}(\phi^\infty_1)$ and $\rs_2^*Z_{T_2}(\phi^\infty_2)$ can be computed in terms of the intersection number of some nice extensions of cycles on some nice regular model, after further suitable translations by elements in $\dT^\tR_{\dQ^\ac}\setminus\fm_\pi^\tR$. Eventually, all these local indices turn out to be (linear combinations of) Fourier coefficients of derivatives of Eisenstein series (and values of Eisenstein series for finitely many $u$).

The final ingredient is the Euler expansion of the doubling integral of cusp forms in $\pi$ against those derivatives of Eisenstein series (and Eisenstein series), which expresses the height pairing in terms of $L'(\tfrac{1}{2},\pi)$ and local doubling zeta integrals (in particular, it belongs to $\dC$ and is independent of $\ell$). An apparent technical challenge for this approach is to show that there exist test functions $(\phi^\infty_1,\phi^\infty_2)$ satisfying the regularity condition and yielding nonvanishing local doubling zeta integrals; this is solved in Proposition \ref{pr:regular}. The proofs for Theorem \ref{th:aipf} and Corollary \ref{co:aipf} follow from a similar strategy.

In Section \ref{ss:setup}, we collect setups and notation that are running through the entire article, organized in several groups so that readers can easily trace. In Section \ref{ss:doubling}, we recall the doubling method in the theory of theta lifting, and prove all necessary results from the representation-theoretical side. In Section \ref{ss:special}, we recall the notation of unitary Shimura varieties, their special cycles and generating functions. We introduce the important hypothesis on the modularity of generating functions, assuming which we define arithmetic theta lifting. In Section \ref{ss:height}, we introduce the notion of Beilinson's height, in a restricted but unconditional form, together with the decomposition into local indices. In Section \ref{ss:auxiliary}, we introduce a variant of unitary Shimura variety that admits moduli interpretation, which will only be used in computing local indices at various places. In Sections \ref{ss:split}, \ref{ss:inert1}, \ref{ss:inert2}, and \ref{ss:archimedean}, we compute local indices at split, inert with self-dual level, inert with almost self-dual level, and archimedean places, respectively. Finally, in Section \ref{ss:main}, we prove our main results. There are two appendices: Appendix \ref{ss:fourier} contains two lemmas in Fourier analysis that are only used in the proof of Proposition \ref{pr:regular}; and Appendix \ref{ss:beilinson} collects some new observations concerning Beilinson's local indices at non-archimedean places.

\subsection*{Notation and conventions}

\begin{itemize}
  \item When we have a function $f$ on a product set $A_1\times\cdots\times A_m$, we will write $f(a_1,\dots,a_m)$ instead of $f((a_1,\dots,a_m))$ for its value at an element $(a_1,\dots,a_m)\in A_1\times\cdots\times A_m$.

  \item For a set $S$, we denote by $\CF_S$ the characteristic function of $S$.

  \item All rings (but not algebras) are commutative and unital; and ring homomorphisms preserve units.

  \item If a base ring is not specified in the tensor operation $\otimes$, then it is $\dZ$.

  \item For an abelian group $A$ and a ring $R$, we put $A_R\coloneqq A\otimes R$.

  \item For an integer $m\geq 0$, we denote by $0_m$ and $1_m$ the null and identity matrices of rank $m$, respectively. We also denote by $\tw_m$ the matrix $\(\begin{smallmatrix}&1_m\\ -1_m &\end{smallmatrix}\)$.

  \item We denote by $\tc\colon\dC\to\dC$ the complex conjugation. For an element $x$ in a complex space with a default underlying real structure, we denote by $x^\tc$ its complex conjugation.

  \item For a field $K$, we denote by $\ol{K}$ the abstract algebraic closure of $K$. However, for aesthetic reason, we will write $\ol\dQ_p$ instead of $\ol{\dQ_p}$ and will denote by $\ol\dF_p$ its residue field. On the other hand, we denote by $\dQ^\ac$ the algebraic closure of $\dQ$ \emph{inside} $\dC$.

  \item For a number field $K$, we denote by $\psi_K\colon K\backslash \dA_K\to\dC^\times$ the standard additive character, namely, $\psi_K\coloneqq\psi_\dQ\circ\Tr_{K/\dQ}$ in which $\psi_\dQ\colon\dQ\backslash\dA\to\dC^\times$ is the unique character such that $\psi_{\dQ,\infty}(x)=\re^{2\pi ix}$.

  \item Throughout the entire article, all parabolic inductions are unitarily normalized.
\end{itemize}

\subsubsection*{Acknowledgements}

We would like to thank Wei~Zhang for helpful discussion and careful reading of early drafts with many valuable comments and suggestions for improvement. We also thank Miaofen~Chen, Wee~Teck~Gan, Benedict~Gross, and Shouwu~Zhang for helpful comments. Finally, we thank the anonymous referees for their careful reading and many useful suggestions and comments. The research of C.~L. is partially supported by the NSF grant DMS--1802269. The research of Y.~L. is partially supported by the NSF grant DMS--1702019, DMS--2000533, and a Sloan Research Fellowship.

\section{Running notation}
\label{ss:setup}

In this section, we collect several groups of more specific notation that will be used throughout the remaining sections except appendices.

\begin{notation}\label{st:f}
Let $E/F$ be a CM extension of number fields, so that $\tc$ is a well-defined element in $\Gal(E/F)$. We continue to fix an embedding $\biota\colon E\hookrightarrow\dC$. We denote by $\bu$ the (archimedean) place of $E$ induced by $\biota$ and regard $E$ as a subfield of $\dC$ via $\biota$.
\begin{enumerate}[label=(F\arabic*)]
  \item We denote by
      \begin{itemize}
        \item $\tV_F$ and $\tV_F^\fin$ the set of all places and non-archimedean places of $F$, respectively;

        \item $\tV_F^\spl$, $\tV_F^\inert$, and $\tV_F^\ram$ the subsets of $\tV_F^\fin$ of those that are split, inert, and ramified in $E$, respectively;

        \item $\tV_F^{(\diamond)}$ the subset of $\tV_F$ of places above $\diamond$ for every place $\diamond$ of $\dQ$; and

        \item $\tV_E^?$ the places of $E$ above $\tV_F^?$.
      \end{itemize}
      Moreover,
      \begin{itemize}
        \item for every place $u\in\tV_E$ of $E$, we denote by $\ul{u}\in\tV_F$ the underlying place of $F$;

        \item for every $v\in\tV_F^\fin$, we denote by $\fp_v$ the maximal ideal of $O_{F_v}$, and put $q_v\coloneqq|O_{F_v}/\fp_v|$;

        \item for every $v\in\tV_F$, we put $E_v\coloneqq E\otimes_FF_v$ and denote by $|\;|_{E_v}\colon E_v^\times\to\dC^\times$ the normalized norm character.
      \end{itemize}

  \item Let $m\geq 0$ be an integer.
      \begin{itemize}
        \item We denote by $\Herm_m$ the subscheme of $\Res_{E/F}\Mat_{m,m}$ of $m$-by-$m$ matrices $b$ satisfying $\pres{\rt}{b}^\tc=b$. Put $\Herm_m^\circ\coloneqq\Herm_m\cap\Res_{E/F}\GL_m$.

        \item For every ordered partition $m=m_1+\cdots+m_s$ with $m_i$ a positive integer, we denote by $\partial_{m_1,\dots,m_s}\colon\Herm_m\to\Herm_{m_1}\times\cdots\times\Herm_{m_s}$ the morphism that extracts the diagonal blocks with corresponding ranks.

        \item We denote by $\Herm_m(F)^+$ (resp.\ $\Herm^\circ_m(F)^+$) the subset of $\Herm_m(F)$ of elements that are totally semi-positive definite (resp.\ totally positive definite).
      \end{itemize}

  \item For every $u\in\tV_E^{(\infty)}$, we fix an embedding $\iota_u\colon E\hookrightarrow\dC$ inducing $u$ (with $\iota_\bu=\biota$), and identify $E_u$ with $\dC$ via $\iota_u$.

  \item Let $\eta\coloneqq\eta_{E/F}\colon \dA_F^\times\to\dC^\times$ be the quadratic character associated to $E/F$. For every $v\in\tV_F$ and every positive integer $m$, put
      \[
      b_{m,v}(s)\coloneqq\prod_{i=1}^m L(2s+i,\eta_v^{m-i}).
      \]
      Put $b_m(s)\coloneqq\prod_{v\in\tV_F}b_{m,v}(s)$.

  \item For every element $T\in\Herm_m(\dA_F)$, we have the character $\psi_T\colon\Herm_m(\dA_F)\to\dC^\times$ given by the formula $\psi_T(b)\coloneqq\psi_F(\tr bT)$.

  \item Let $R$ be a commutative $F$-algebra. A (skew-)hermitian space over $R\otimes_FE$ is a free $R\otimes_FE$-module $V$ of finite rank, equipped with a (skew-)hermitian form $(\;,\;)_V$ with respect to the involution $\tc$ that is nondegenerate.
\end{enumerate}
\end{notation}

\begin{notation}\label{st:h}
Throughout the article, we fix an even positive integer $n=2r$. Let $(V,(\;,\;)_V)$ be a hermitian space over $\dA_E$ of rank $n$ that is totally positive definite.
\begin{enumerate}[label=(H\arabic*)]
  \item For every commutative $\dA_F$-algebra $R$ and every integer $m\geq 0$, we denote by
      \[
      T(x)\coloneqq\((x_i,x_j)_V\)_{i,j}\in\Herm_m(R)
      \]
      the moment matrix of an element $x=(x_1,\dots,x_m)\in V^m\otimes_{\dA_F}R$.

  \item For every $v\in\tV_F$, we put $V_v\coloneqq V\otimes_{\dA_F}F_v$ which is a hermitian space over $E_v$, and define the local Hasse invariant of $V_v$ to be $\epsilon(V_v)\coloneqq\eta_v((-1)^r\dtm V_v)\in\{\pm 1\}$. In what follows, we will abbreviate $\epsilon(V_v)$ as $\epsilon_v$.

  \item Let $v$ be a place of $F$ and $m\geq 0$ an integer.
      \begin{itemize}
        \item For $T\in\Herm_m(F_v)$, we put $(V^m_v)_T\coloneqq\{x\in V^m_v\res T(x)=T\}$, and
            \[
            (V^m_v)_\reg\coloneqq\bigcup_{T\in\Herm_m^\circ(F_v)}(V^m_v)_T.
            \]

        \item We denote by $\sS(V_v^m)$ the space of (complex valued) Bruhat--Schwartz functions on $V_v^m$. When $v\in\tV_F^{(\infty)}$, we have the Gaussian function $\phi^0_v\in\sS(V_v^m)$ given by the formula $\phi^0_v(x)=\re^{-2\pi\tr T(x)}$.

        \item We have a Fourier transform map $\widehat{\phantom{a}}\colon \sS(V_v^m)\to \sS(V_v^m)$ sending $\phi$ to $\widehat\phi$ defined by the formula
            \[
            \widehat\phi(x)\coloneqq\int_{V_v^m}\phi(y)\psi_{E,v}\(\sum_{i=1}^m(x_i,y_i)_V\)\rd y,
            \]
            where $\r{d}y$ is the self-dual Haar measure on $V_v^m$ with respect to $\psi_{E,v}$.

        \item In what follows, we will always use this self-dual Haar measure on $V_v^m$.
      \end{itemize}

  \item Let $m\geq 0$ be an integer. For $T\in\Herm_m(F)$, we put
      \[
      \Diff(T,V)\coloneqq\{v\in\tV_F\res(V^m_v)_T=\emptyset\},
      \]
      which is a finite subset of $\tV_F\setminus\tV_F^\spl$.

  \item Take a nonempty finite subset $\tR\subseteq\tV_F^\fin$ that contains $\tV_F^\ram$. Let $\tS$ be the subset of $\tV_F^\fin\setminus\tR$ consisting of $v$ such that $\epsilon_v=-1$, which is contained in $\tV_F^\inert$.

  \item We fix a $\prod_{v\in\tV_F^\fin\setminus\tR}O_{E_v}$-lattice $\Lambda^\tR$ in $V\otimes_{\dA_F}\dA_F^{\infty,\tR}$ such that for every $v\in\tV_F^\fin\setminus\tR$, $\Lambda^\tR_v$ is a subgroup of $(\Lambda^\tR_v)^\vee$ of index $q_v^{1-\epsilon_v}$, where
      \[
      (\Lambda^\tR_v)^\vee\coloneqq\{x\in V_v\res\psi_{E,v}((x,y)_V)=1\text{ for every }y\in\Lambda^\tR_v\}
      \]
      is the $\psi_{E,v}$-dual lattice of $\Lambda^\tR_v$.

  \item Put $H\coloneqq\rU(V)$, which is a reductive group over $\dA_F$.

  \item Denote by $L^\tR\subseteq H(\dA_F^{\infty,\tR})$ the stabilizer of $\Lambda^\tR$, which is a special maximal subgroup.\footnote{When $r\geq 2$ (resp.\ $r=1$), the set of conjugacy classes of special maximal subgroups of $H(\dA_F^{\infty,\tR})$ is canonically a torsor over $\mu_2^{\oplus\tV_F^\inert\setminus\tR}$ (resp.\ $\mu_2^{\oplus\tV_F^\inert\setminus(\tR\cup\tS_\pi)}$).} We have the (abstract) Hecke algebra away from $\tR$
      \[
      \dT^\tR\coloneqq\dZ[L^\tR\backslash H(\dA_F^{\infty,\tR})/L^\tR],
      \]
      which is a ring with the unit $\CF_{L^\tR}$, and denote by $\dS^\tR$ the subring
      \[
      \varinjlim_{\substack{\tT\subseteq\tV_F^\spl\setminus\tR\\|\tT|<\infty}}
      \dZ[(L^\tR)_\tT\backslash H(F_\tT)/(L^\tR)_\tT]\otimes\CF_{(L^\tR)^\tT}
      \]
      of $\dT^\tR$.

  \item Suppose that $V$ is \emph{incoherent}, namely, $\prod_{v\in\tV_F}\epsilon_v=-1$. For every $u\in\tV_E\setminus\tV_E^\spl$, we fix a \emph{$u$-nearby space} $\pres{u}{V}$ of $V$, which is a hermitian space over $E$, and an isomorphism $\pres{u}{V}\otimes_F\dA_F^{\ul{u}}\simeq V\otimes_{\dA_F}\dA_F^{\ul{u}}$. More precisely,
      \begin{itemize}
        \item if $u\in\tV_E^{(\infty)}$, then $\pres{u}{V}$ is the hermitian space over $E$, unique up to isomorphism, that has signature $(n-1,1)$ at $u$ and satisfies $\pres{u}{V}\otimes_F\dA_F^{\ul{u}}\simeq V\otimes_{\dA_F}\dA_F^{\ul{u}}$;

        \item if $u\in\tV_E^\fin\setminus\tV_E^\spl$, then $\pres{u}{V}$ is the hermitian space over $E$, unique up to isomorphism, that satisfies $\pres{u}{V}\otimes_F\dA_F^{\ul{u}}\simeq V\otimes_{\dA_F}\dA_F^{\ul{u}}$.
      \end{itemize}
      Put $\pres{u}{H}\coloneqq\rU(\pres{u}{V})$, which is a reductive group over $F$. Then $\pres{u}{H}(\dA_F^{\ul{u}})$ and $H(\dA_F^{\ul{u}})$ are identified.
\end{enumerate}
\end{notation}

\begin{notation}\label{st:g}
Let $m\geq 0$ be an integer. We equip $W_m=E^{2m}$ and $\bar{W}_m=E^{2m}$ the skew-hermitian forms given by the matrices $\tw_m$ and $-\tw_m$, respectively.
\begin{enumerate}[label=(G\arabic*)]
  \item Let $G_m$ be the unitary group of both $W_m$ and $\bar{W}_m$. We write elements of $W_m$ and $\bar{W}_m$ in the row form, on which $G_m$ acts from the right.

  \item We denote by $\{e_1,\dots,e_{2m}\}$ and $\{\bar{e}_1,\dots,\bar{e}_{2m}\}$ the natural bases of $W_m$ and $\bar{W}_m$, respectively.

  \item Let $P_m\subseteq G_m$ be the parabolic subgroup stabilizing the subspace generated by $\{e_{r+1},\dots,e_{2m}\}$, and $N_m\subseteq P_m$ its unipotent radical.

  \item We have
     \begin{itemize}
       \item a homomorphism $m\colon\Res_{E/F}\GL_m\to P_m$ sending $a$ to
          \[
          m(a)\coloneqq
          \begin{pmatrix}
              a &  \\
               & \pres{\rt}{a}^{\tc,-1} \\
          \end{pmatrix}
          ,
          \]
          which identifies $\Res_{E/F}\GL_m$ as a Levi factor of $P_m$.

       \item a homomorphism $n\colon\Herm_m\to N_m$ sending $b$ to
          \[
          n(b)\coloneqq
          \begin{pmatrix}
              1_m & b \\
               & 1_m \\
          \end{pmatrix}
          ,
          \]
          which is an isomorphism.
     \end{itemize}

  \item We define a maximal compact subgroup $K_m=\prod_{v\in\tV_F}K_{m,v}$ of $G_m(\dA_F)$ in the following way:
    \begin{itemize}
      \item for $v\in\tV_F^\fin$, $K_{m,v}$ is the stabilizer of the lattice $O_{E_v}^{2m}$;

      \item for $v\in\tV_F^{(\infty)}$, $K_{m,v}$ is the subgroup of the form
         \[
         [k_1,k_2]\coloneqq\frac{1}{2}
         \begin{pmatrix}
           k_1+k_2   & -ik_1+ik_2 \\
           ik_1-ik_2   & k_1+k_2 \\
         \end{pmatrix}
         ,
         \]
         in which $k_i\in\GL_m(\dC)$ satisfying $k_i\pres{\rt}{k}_i^\tc=1_m$ for $i=1,2$. Here, we have identified $G_m(F_v)$ as a subgroup of $\GL_{2m}(\dC)$ via the embedding $\iota_u$ with $v=\ul{u}$ in Notation \ref{st:f}(F3).
    \end{itemize}

  \item For every $v\in\tV_F^{(\infty)}$, we have a character $\kappa_{m,v}\colon K_{m,v}\to\dC^\times$ that sends $[k_1,k_2]$ to $\dtm k_1/\dtm k_2$.\footnote{In fact, neither $K_{m,v}$ nor $\kappa_{m,v}$ depends on the choice of the embedding $\iota_v$ for $v\in\tV_F^{(\infty)}$.}

  \item For every $v\in\tV_F$, we define a Haar measure $\r{d}g_v$ on $G_m(F_v)$ as follows:
    \begin{itemize}
      \item for $v\in\tV_F^\fin$, $\r{d}g_v$ is the Haar measure under which $K_{m,v}$ has volume $1$;

      \item for $v\in\tV_F^{(\infty)}$, $\r{d}g_v$ is the product of the measure on $K_{m,v}$ of total volume $1$ and the standard hyperbolic measure on $G_m(F_v)/K_{m,v}$.
    \end{itemize}
    Put $\r{d}g=\prod_{v}\r{d}g_v$, which is a Haar measure on $G_m(\dA_F)$.

  \item We denote by $\cA(G_m(F)\backslash G_m(\dA_F))$ the space of both $\cZ(\fg_{m,\infty})$-finite and $K_{m,\infty}$-finite automorphic forms on $G_m(\dA_F)$, where $\cZ(\fg_{m,\infty})$ denotes the center of the complexified universal enveloping algebra of the Lie algebra $\fg_{m,\infty}$ of $G_m\otimes_FF_\infty$. We denote by
      \begin{itemize}
        \item $\cA^{[r]}(G_m(F)\backslash G_m(\dA_F))$ the maximal subspace of $\cA(G_m(F)\backslash G_m(\dA_F))$ on which for every $v\in\tV_F^{(\infty)}$, $K_{m,v}$ acts by the character $\kappa_{m,v}^r$,

        \item $\cA^{[r]\tR}(G_m(F)\backslash G_m(\dA_F))$ the maximal subspace of $\cA^{[r]}(G_m(F)\backslash G_m(\dA_F))$ on which
             \begin{itemize}
               \item for every $v\in\tV_F^\fin\setminus(\tR\cup\tS)$, $K_{m,v}$ acts trivially; and

               \item for every $v\in\tS$, the standard Iwahori subgroup $I_{m,v}$ acts trivially and $\dC[I_{m,v}\backslash K_{m,v}/I_{m,v}]$ acts by the character $\kappa_{m,v}^-$ (\cite{Liu20}*{Definition~2.1}),
             \end{itemize}

        \item $\cA_\cusp(G_m(F)\backslash G_m(\dA_F))$ the subspace of $\cA(G_m(F)\backslash G_m(\dA_F))$ of cusp forms, and by $\langle\;,\;\rangle_{G_m}$ the hermitian form on $\cA_\cusp(G_m(F)\backslash G_m(\dA_F))$ given by the Petersson inner product with respect to the Haar measure $\r{d}g$.
      \end{itemize}
      For a subspace $\cV$ of $\cA(G_m(F)\backslash G_m(\dA_F))$, we denote by
      \begin{itemize}
        \item $\cV^{[r]}$ the intersection of $\cV$ and $\cA^{[r]}(G_m(F)\backslash G_m(\dA_F))$,

        \item $\cV^{[r]\tR}$ the intersection of $\cV$ and $\cA^{[r]\tR}(G_m(F)\backslash G_m(\dA_F))$,

        \item $\cV^\tc$ the subspace $\{\varphi^\tc\res\varphi\in\cV\}$.
      \end{itemize}
\end{enumerate}
\end{notation}

\begin{notation}\label{st:w}
We review the Weil representation.
\begin{enumerate}[label=(W\arabic*)]
  \item For every $v\in\tV_F$, we have the Weil representation $\omega_{m,v}$ of $G_m(F_v)\times H(F_v)$, with respect to the additive character $\psi_{F,v}$ and the trivial splitting character, realized on the Schr\"{o}dinger model $\sS(V_v^m)$. For the readers' convenience, we review the formulas:
      \begin{itemize}
        \item for $a\in\GL_m(E_v)$ and $\phi\in \sS(V_v^m)$, we have
          \[
          \omega_{m,v}(m(a))\phi(x)=|\dtm a|_{E_v}^r\cdot \phi(x a);
          \]

        \item for $b\in\Herm_m(F_v)$ and $\phi\in \sS(V_v^m)$, we have
          \[
          \omega_{m,v}(n(b))\phi(x)=\psi_{T(x)}(b)\cdot \phi(x)
          \]
          (see Notation \ref{st:f}(F5) for $\psi_{T(x)}$);

        \item for $\phi\in \sS(V_v^m)$, we have
          \[
          \omega_{m,v}\(\tw_m\)\phi(x)=\gamma_{V_v,\psi_{F,v}}^m\cdot\widehat\phi(x),
          \]
          where $\gamma_{V_v,\psi_{F,v}}$ is certain Weil constant determined by $V_v$ and $\psi_{F,v}$;

        \item for $h\in H(F_v)$ and $\phi\in \sS(V_v^m)$, we have
          \[
          \omega_{m,v}(h)\phi(x)=\phi(h^{-1}x).
          \]
      \end{itemize}
      We put $\omega_m\coloneqq\otimes_v\omega_{m,v}$ as the ad\`{e}lic version, realized on $\sS(V^m)$.

  \item For every $v$ of $F$, we also realize the contragredient representation $\omega_{m,v}^\vee$ on the space $\sS(V_v^m)$ as well via the bilinear pairing
      \[
      \langle\;,\;\rangle_{\omega_{m,v}}\colon \sS(V_v^m)\times \sS(V_v^m)\to\dC
      \]
      defined by the formula
      \[
      \langle\phi^\vee,\phi\rangle_{\omega_{m,v}}\coloneqq\int_{V_v^m}\phi(x)\phi^\vee(x)\rd x
      \]
      for $\phi,\phi^\vee\in\sS(V_v^m)$.
\end{enumerate}
\end{notation}

\begin{notation}\label{st:c}
For a locally Noetherian scheme $X$ and an integer $m\geq 0$, we denote by $\rZ^m(X)$ the free abelian group generated by irreducible closed subschemes of codimension $m$ and $\CH^m(X)$ the quotient by rational equivalence. Suppose that $X$ is smooth over a field $K$ of characteristic zero. Let $\ell$ be a rational prime.
\begin{enumerate}[label=(C\arabic*)]
  \item We denote by $\rZ^m(X)^0$ the kernel of the de Rham cycle class map
     \[
     \cl_{X,\dr}\colon\rZ^m(X)\to\rH^{2m}_{\dr}(X/K)(m),
     \]
     and by $\CH^m(X)^0$ the image of $\rZ^m(X)^0$ in $\CH^m(X)$.

  \item When $K$ is a non-archimedean local field, we denote by $\rZ^m(X)^{\langle\ell\rangle}$ the kernels of the $\ell$-adic cycle class map
      \[
      \cl_{X,\ell}\colon\rZ^m(X)\to\rH^{2m}(X,\dQ_\ell(m)).
      \]

  \item When $K$ is a number field, we define $\rZ^m(X)^{\langle\ell\rangle}$ via the following Cartesian diagram
      \[
      \xymatrix{
      \rZ^m(X)^{\langle\ell\rangle} \ar[r]\ar[d] &  \prod_{v}\rZ^m(X_{K_v})^{\langle\ell\rangle} \ar[d] \\
      \rZ^m(X) \ar[r] & \prod_{v}\rZ^m(X_{K_v})
      }
      \]
      where the product is taken over all non-archimedean places of $K$. We denote by $\CH^m(X)^{\langle\ell\rangle}$ the image of $\rZ^m(X)^{\langle\ell\rangle}$ in $\CH^m(X)$, which is contained in $\CH^m(X)^0$ by the comparison theorem between de Rham and $\ell$-adic cohomology.
\end{enumerate}
\end{notation}

\section{Doubling method and analytic side}
\label{ss:doubling}

In this section, we review the doubling method and prove several statements on the analytic side of our desired height formula.

We have the doubling skew-hermitian space $W_r^\Box\coloneqq W_r\oplus\bar{W}_r$ (Notation \ref{st:g}(G1)). Let $G_r^\Box$ be the unitary group of $W_r^\Box$, which contains $G_r\times G_r$ canonically. We now take a basis $\{e^\Box_1,\dots,e^\Box_{4r}\}$ of $W_r^\Box$ by the formula
\[
e^\Box_i=e_i,\quad e^\Box_{r+i}=-\bar{e}_i,\quad e^\Box_{2r+i}=e_{r+i}, \quad e^\Box_{3r+i}=\bar{e}_{r+i}
\]
for $1\leq i\leq r$, under which we may identify $W_r^\Box$ with $W_{2r}$ and $G_r^\Box$ with $G_{2r}$. Put
\begin{align}\label{eq:doubling}
\tw_r^\Box\coloneqq\tw_{2r},\quad P_r^\Box\coloneqq P_{2r},\quad N_r^\Box\coloneqq N_{2r},\quad K_r^\Box\coloneqq K_{2r},
\quad \omega_r^\Box\coloneqq\omega_{2r}
\end{align}
(see Notation \ref{st:g} and Notation \ref{st:w}). We denote by
\[
\delta^\Box_r\colon P_r^\Box\to\Res_{E/F}\GL_1
\]
the composition of the Levi quotient map $P_r^\Box=P_{2r}\to M_{2r}$, the isomorphism $m^{-1}\colon M_{2r}\to\Res_{E/F}\GL_{2r}$, and the determinant $\Res_{E/F}\GL_{2r}\to\Res_{E/F}\GL_1$. Put
\[
\bw_r\coloneqq
\begin{pmatrix}
     &  & 1_r &  \\
     & 1_r &  &  \\
    -1_r & 1_r &  &  \\
     &  & 1_r & 1_r \\
\end{pmatrix}
\in G_r^\Box(F).
\]
Then $P_r^\Box\bw_r(G_r\times G_r)$ is Zariski open in $G_r^\Box$.

Let $v$ be a place of $F$. For $s\in\dC$, we have the degenerate principal series of $G_r^\Box(F_v)$, which is defined as the normalized induced representation
\[
\rI^\Box_{r,v}(s)\coloneqq\Ind_{P_r^\Box(F_v)}^{G_r^\Box(F_v)}(|\;|_{E_v}^s\circ\delta^\Box_{r,v})
\]
of $G_r^\Box(F_v)$. We denote by $\rI^\Box_r(s)$ the restricted tensor product of $\rI^\Box_{r,v}(s)$ for all places $v$ of $F$ with respect to unramified sections.

For every section $f\in\rI^\Box_r(0)$, let $f^{(s)}\in\rI^\Box_r(s)$ be the standard section induced by $f$. Then we have the Eisenstein series $E(g,f^{(s)})$ for $g\in G_r^\Box(\dA_F)$. We have a $G_r^\Box(\dA_F)$-intertwining map
\[
f_\bullet\colon \sS(V^{2r})\to\rI^\Box_r(0)
\]
sending $\Phi$ to $f_\Phi$ defined by the formula $f_\Phi(g)\coloneqq\omega_r^\Box(g)\Phi(0)$ (see \eqref{eq:doubling} for $\omega_r^\Box$). In particular, for $\Phi\in\sS(V^{2r})$, we have the Eisenstein series
\[
E(s,g,\Phi)=E(g,f_{\Phi}^{(s)})\coloneqq\sum_{\gamma\in P_r^\Box(F)\backslash G_r^\Box(F)}f_\Phi^{(s)}(\gamma g)
\]
for $g\in G_r^\Box(\dA_F)$. It is meromorphic in $s$ and holomorphic on the imaginary line.

\begin{assumption}\label{st:representation}
In what follows, we will consider an irreducible automorphic subrepresentation $(\pi,\cV_\pi)$ of $\cA_\cusp(G_r(F)\backslash G_r(\dA_F))$ satisfying that
\begin{enumerate}
  \item for every $v\in\tV_F^{(\infty)}$, $\pi_v$ is the (unique up to isomorphism) discrete series representation whose restriction to $K_{r,v}$ contains the character $\kappa_{r,v}^r$;

  \item for every $v\in\tV_F^\fin\setminus\tR$, $\pi_v$ is unramified (resp.\ almost unramified) with respect to $K_{r,v}$ if $\epsilon_v=1$ (resp.\ $\epsilon_v=-1$);

  \item for every $v\in\tV_F^\fin$, $\pi_v$ is tempered.
\end{enumerate}
We realize the contragredient representation $\pi^\vee$ on $\cV_\pi^\tc$ via the Petersson inner product $\langle\;,\;\rangle_{G_r}$ (Notation \ref{st:g}(G8)). By (1) and (2), we have $\cV_\pi^{[r]\tR}\neq\{0\}$, where $\cV_\pi^{[r]\tR}$ is defined in Notation \ref{st:g}(G8).
\end{assumption}

\begin{remark}\label{re:dichotomy}
By Proposition \ref{pr:uniqueness}(2) below, we know that when $\tR\subseteq\tV_F^\spl$, $V$ coincides with the hermitian space over $\dA_E$ of rank $n$ determined by $\pi$ via local theta dichotomy.
\end{remark}

\begin{definition}\label{de:lfunction}
We define the $L$-function for $\pi$ as the Euler product $L(s,\pi)\coloneqq\prod_{v}L(s,\pi_v)$ over all places of $F$, in which
\begin{enumerate}
  \item for $v\in\tV_F^\fin$, $L(s,\pi_v)$ is the doubling $L$-function defined in \cite{Yam14}*{Theorem~5.2};

  \item for $v\in\tV_F^{(\infty)}$, $L(s,\pi_v)$ is the $L$-function of the standard base change $\BC(\pi_v)$ of $\pi_v$. By Assumption \ref{st:representation}(1), $\BC(\pi_v)$ is the principal series representation of $\GL_n(\dC)$ that is the normalized induction of $\arg^{n-1}\boxtimes\arg^{n-3}\boxtimes\cdots\boxtimes\arg^{3-n}\boxtimes\arg^{1-n}$ where $\arg\colon\dC^\times\to\dC^\times$ is the argument character. In particular, we have
      \begin{align}\label{eq:lfunction}
      L(s+\tfrac{1}{2},\pi_v)=\(\prod_{i=1}^r2(2\pi)^{-(s+i)}\Gamma(s+i)\)^2.
      \end{align}
\end{enumerate}
\end{definition}

\begin{remark}\label{re:doubling}
Let $v$ be a place of $F$.
\begin{enumerate}
  \item For $v\in\tV_F^{(\infty)}$, doubling $L$-function is only well-defined up to an entire function without zeros. However, one can show that $L(s,\pi_v)$ satisfies the requirement for the doubling $L$-function in \cite{Yam14}*{Theorem~5.2}.

  \item For $v\in\tV_F^\spl$, the standard base change $\BC(\pi_v)$ is well-defined and we have $L(s,\pi_v)=L(s,\BC(\pi_v))$ by \cite{Yam14}*{Theorem~7.2}.

  \item For $v\in\tV_F^\inert\setminus\tR$, the standard base change $\BC(\pi_v)$ is well-defined and we have $L(s,\pi_v)=L(s,\BC(\pi_v))$ by \cite{Liu20}*{Remark~1.4}.
\end{enumerate}
In particular, when $\tR\subseteq\tV_F^\spl$, we have $L(s,\pi)=\prod_{v}L(s,\BC(\pi_v))$.
\end{remark}

Let $v$ be a place of $F$. We denote by $\langle\;,\;\rangle_{\pi_v}\colon\pi_v^\vee\times\pi_v\to\dC$ the tautological pairing. For $\varphi_v\in\pi_v$, $\varphi^\vee_v\in\pi_v^\vee$, and a good section $f^{(s)}\in\rI^\Box_{r,v}(s)$ (\cite{Yam14}*{Definition~3.1}), we have the local doubling zeta integral
\[
Z(\varphi^\vee_v\otimes\varphi_v,f^{(s)})\coloneqq\int_{G_r(F_v)}\langle\pi_v^\vee(g)\varphi^\vee_v,\varphi_v\rangle_{\pi_v}\cdot
f^{(s)}(\bw_r(g,1_{2r}))\rd g,
\]
and the normalized version
\[
Z^\natural(\varphi^\vee_v\otimes\varphi_v,f^{(s)})\coloneqq
\(\frac{L(s+\tfrac{1}{2},\pi_v)}{b_{2r,v}(s)}\)^{-1}\cdot Z(\varphi^\vee_v\otimes\varphi_v,f^{(s)}),
\]
which is holomorphic in $s$. In particular, taking $s=0$, we obtain a functional
\[
\fZ^\natural_{\pi_v,V_v}\colon\pi_v^\vee\otimes\pi_v\otimes\sS(V_v^{2r})\to\dC
\]
such that
\[
\fZ^\natural_{\pi_v,V_v}(\varphi^\vee_v,\varphi_v,\Phi_v)=Z^\natural(\varphi^\vee_v\otimes\varphi_v,f^{(0)}_{\Phi_v})
=Z^\natural(\varphi^\vee_v\otimes\varphi_v,f_{\Phi_v}).
\]

\begin{remark}\label{re:yamana}
By \cite{Yam14}*{Lemma~7.2}, we know that the integral defining $Z(\varphi^\vee_v\otimes\varphi_v,f^{(0)})$ is absolutely convergent, and that
\[
\frac{L(s+\tfrac{1}{2},\pi_v)}{b_{2r,v}(s)}
\]
is finite and invertible at $s=0$.
\end{remark}

\begin{proposition}\label{pr:uniqueness}
Let $(\pi,\cV_\pi)$ be as in Assumption \ref{st:representation}.
\begin{enumerate}
  \item For every $v\in\tV_F^\fin$, we have
     \begin{align*}
     \dim_\dC\Hom_{G_r(F_v)\times G_r(F_v)}(\rI^\Box_{r,v}(0),\pi_v\boxtimes\pi_v^\vee)=1.
     \end{align*}

  \item For every $v\in(\tV_F^\fin\setminus\tR)\cup\tV_F^\spl$, $V_v$ is the unique hermitian space over $E_v$ of rank $2r$, up to isomorphism, such that $\fZ^\natural_{\pi_v,V_v}\neq 0$.

  \item For every $v\in\tV_F^\fin$, $\Hom_{G_r(F_v)}(\sS(V_v^r),\pi_v)$ is irreducible as a representation of $H(F_v)$, and is nonzero if $v\in(\tV_F^\fin\setminus\tR)\cup\tV_F^\spl$.
\end{enumerate}
\end{proposition}

\begin{proof}
To ease notation, we will suppress the place $v$ throughout the proof. For a hermitian space $\tilde{V}$ over $E$ of rank $2r$, denote by $\rR(0,\tilde{V})\subseteq\rI^\Box_r(0)$ the subspace spanned by Siegel--Weil sections from $\tilde{V}$ and put $\Theta(\pi,\tilde{V})\coloneqq\Hom_{G_r(F)}(\sS(\tilde{V}^r),\pi)$. By the seesaw identity, we have
\begin{align*}
\Hom_{G_r(F)\times G_r(F)}(\rR(0,\tilde{V}),\pi\boxtimes\pi^\vee)\simeq
\Hom_{\tilde{H}(F)}(\Theta(\pi,\tilde{V})\otimes\Theta(\pi^\vee,\tilde{V}),\CF)
\end{align*}
where $\tilde{H}\coloneqq\rU(\tilde{V})$. Since $\pi$ is tempered, by (the same argument for) \cite{GI16}*{Theorem~4.1(v)}, $\Theta(\pi,\tilde{V})$ is a semisimple representation of $\tilde{H}(F)$. By \cite{GT16}*{Theorem~1.2}, we know that $\Theta(\pi,\tilde{V})$ is either zero or irreducible. By the local theta dichotomy \cite{GG11}*{Theorem~1.8} (see also \cite{HKS96}*{Corollary~4.4} and \cite{Har07}*{Theorem~2.1.7}), there exists exactly one choice $\tilde{V}$, up to isomorphism, such that $\Theta(\pi,\tilde{V})\neq 0$. Thus, we obtain (1) by \cite{KS97}*{Theorem~1.2 \& Theorem~1.3}.

For (2), there are two cases. If $v\in\tV_F^\spl$, then it follows from (1) and \cite{KS97}*{Theorem~1.3}. If $v\in\tV_F^\inert\setminus\tR$, then the uniqueness follows from (1) and \cite{KS97}*{Theorem~1.2}; and the nonvanishing of $\fZ^\natural_{\pi,V}$ follows from \cite{Liu20}*{Proposition~5.6 \& Lemma~6.1}.

For (3), the irreducibility of $\Theta(\pi,V)=\Hom_{G_r(F)}(\sS(V^r),\pi)$ has already been proved; and the nonvanishing follows from (2).
\end{proof}

\if false

\begin{proof}
For (1), to ease notation, we will suppress the place $v$ throughout the proof. Take an integer $0\leq t\leq r$. We have a homomorphism $m\colon\Res_{E/F}\GL_t\times G_{r-t}\to G_r$ given by the assignment
\[
\(a,
\begin{pmatrix}
g_1 & g_2 \\
g_3 & g_4
\end{pmatrix}
\)
\mapsto
\begin{pmatrix}
a & & & \\
& g_1 && g_2 \\
& & \pres{\rt}{a}^{\tc,-1} & \\
& g_3 && g_4
\end{pmatrix},
\]
whose image we denote by $M_r^t$. Let $P_r^t$ be the unique parabolic subgroup of $G_r$ containing $P_r$ and having $M_r^t$ as a Levi factor, and $\pres{t}P_r$ the opposite parabolic so that $P_r^t\cap \pres{t}P_r=M_r^t$. We denote by $\delta_r^t\colon P_r^t\to\Res_{E/F}\GL_1$ the composition of the Levi quotient map $P_r^t\to M_r^t$, the projection $M_r^t\to\Res_{E/F}\GL_t$, and the determinant $\Res_{E/F}\GL_t\to\Res_{E/F}\GL_1$. For $s\in\dC$, we put
\[
\rJ_r^t(s)\coloneqq\Ind^{G_r(F)\times G_r(F)}_{P_r^t(F)\times P_r^t(F)}
\(\((|\;|_E^{s+\frac{t}{2}}\circ\delta_r^t)\boxtimes(|\;|_E^{s+\frac{t}{2}}\circ\delta_r^t)\)\otimes\sS(G_{r-t}(F))\)
\]
as a representation of $G_r(F)\times G_r(F)$, where $P_r^t(F)\times P_r^t(F)$ acts on $\sS(G_{r-t}(F))$ via the two-side translation by its quotient $G_{r-t}(F)\times G_{r-t}(F)$.

By \cite{GT16}*{Lemma~3.2}, $\rI_r^\Box(s)$ admits a $G_r(F)\times G_r(F)$-stable increasing filtration
\[
0=\rI_r^\Box(s)^{-1}\subseteq \rI_r^\Box(s)^0\subseteq \rI_r^\Box(s)^1\subseteq\cdots\subseteq \rI_r^\Box(s)^r=\rI_r^\Box(s)
\]
with canonical $G_r(F)\times G_r(F)$-equivariant isomorphisms $\rJ_r^t(s)\simeq\rI_r^\Box(s)^t/\rI_r^\Box(s)^{t-1}$ for $0\leq t\leq r$. We claim that
\begin{align}\label{eq:uniqueness2}
\Hom_{G_r(F)\times G_r(F)}(\rJ_r^t(s),\pi\boxtimes\pi^\vee)=0,\qquad\text{for $t\geq 1$ and $s>-\tfrac{1}{2}$.}
\end{align}
Assuming \eqref{eq:uniqueness2}, then we have
\begin{multline*}
\dim_\dC\Hom_{G_r(F)\times G_r(F)}(\rI^\Box_r(0),\pi\boxtimes\pi^\vee)
\leq\dim_\dC\Hom_{G_r(F)\times G_r(F)}(\rJ_r^0(0),\pi\boxtimes\pi^\vee) \\
=\dim_\dC\Hom_{G_r(F)\times G_r(F)}(\sS(G_{r-t}(F)),\pi\boxtimes\pi^\vee)=1.
\end{multline*}
On the other hand, we do have a nonzero element in $\Hom_{G_r(F)\times G_r(F)}(\rI^\Box_r(0),\pi\boxtimes\pi^\vee)$ given by the zeta integral. Thus, (1) holds.

For \eqref{eq:uniqueness2}, by Bernstein's Second Adjointness Theorem, we have a canonical isomorphism
\begin{multline*}
\Hom_{G_r(F)\times G_r(F)}(\rJ_r^t(s),\pi\boxtimes\pi^\vee) \\
\simeq\Hom_{M_r^t(F)\times M_r^t(F)}
\(\((|\;|_E^{s+\frac{t}{2}}\circ\delta_r^t)\boxtimes(|\;|_E^{s+\frac{t}{2}}\circ\delta_r^t)\)\otimes\sS(G_{r-t}(F))
,\rR_{\pres{t}P_r}\pi\boxtimes\rR_{\pres{t}P_r}\pi^\vee\).
\end{multline*}
Here, $\rR_P$ denotes the normalized Jacquet functor with respect to a parabolic $P$. Thus, it suffices to show that for $t\geq 1$ and $s>-\tfrac{1}{2}$, we have
\begin{align}\label{eq:uniqueness3}
\Hom_{M_r^t(F)}\(|\;|_E^{s+\frac{t}{2}}\circ\delta_r^t,\rR_{\pres{t}P_r}\pi^\vee\)=0.
\end{align}
Applying Casselman's criterion for temperedness \cite{Wal03}*{Proposition~III.2.2}, we have
\[
\Hom_{M_r^t(F)}\(\rR_{P_r^t}\pi,|\;|_E^\nu\circ\delta_r^t\)=0
\]
as long as $\nu<0$. Taking dual, we obtain \eqref{eq:uniqueness3} hence \eqref{eq:uniqueness2}. Part (1) is proved.
\end{proof}

\fi

\begin{proposition}\label{pr:eisenstein}
Let $(\pi,\cV_\pi)$ be as in Assumption \ref{st:representation} such that $L(\tfrac{1}{2},\pi)=0$. Take
\begin{itemize}
  \item $\varphi_1=\otimes_v\varphi_{1v}\in\cV_\pi^{[r]\tR}$ and $\varphi_2=\otimes_v\varphi_{2v}\in\cV_\pi^{[r]\tR}$ such that $\langle\varphi_{1v}^\tc,\varphi_{2v}\rangle_{\pi_v}=1$ for $v\in\tV_F\setminus\tR$,\footnote{Strictly speaking, what we fixed is a decomposition $\varphi_1^\tc=\otimes_v(\varphi_1^\tc)_v$ and have abused notation by writing $\varphi^\tc_{1v}$ instead of $(\varphi_1^\tc)_v$.} and

  \item $\Phi=\otimes_v\Phi_v\in\sS(V^{2r})$ such that $\Phi_v$ is the Gaussian function (Notation \ref{st:h}(H3)) for $v\in\tV_F^{(\infty)}$, and $\Phi_v=\CF_{(\Lambda^\tR_v)^{2r}}$ for $v\in\tV_F^\fin\setminus\tR$.
\end{itemize}
Then we have
\begin{align*}
&\quad\int_{G_r(F)\backslash G_r(\dA_F)}\int_{G_r(F)\backslash G_r(\dA_F)}\varphi_2(g_2)\varphi_1^\tc(g_1)E'(0,(g_1,g_2),\Phi)\rd g_1\rd g_2 \\
&=\frac{L'(\tfrac{1}{2},\pi)}{b_{2r}(0)}\cdot C_r^{[F:\dQ]}
\cdot\prod_{v\in\tV_F^\fin}\fZ^\natural_{\pi_v,V_v}(\varphi^\tc_{1v},\varphi_{2v},\Phi_v) \\
&=\frac{L'(\tfrac{1}{2},\pi)}{b_{2r}(0)}\cdot C_r^{[F:\dQ]}
\cdot\prod_{v\in\tS}\frac{(-1)^rq_v^{r-1}(q_v+1)}{(q_v^{2r-1}+1)(q_v^{2r}-1)}
\cdot\prod_{v\in\tR}\fZ^\natural_{\pi_v,V_v}(\varphi^\tc_{1v},\varphi_{2v},\Phi_v),
\end{align*}
where
\[
C_r\coloneqq(-1)^r2^{r(r-1)}\pi^{r^2}\frac{\Gamma(1)\cdots\Gamma(r)}{\Gamma(r+1)\cdots\Gamma(2r)},
\]
and the measure on $G_r(\dA_F)$ is the one defined in Notation \ref{st:g}(G7).
\end{proposition}

\begin{proof}
By the formula derived in \cite{Liu11}*{Page~869}, we have
\[
\iint_{[G_r(F)\backslash G_r(\dA_F)]^2}\varphi_2(g_2)\varphi^\tc_1(g_1)E'(0,(g_1,g_2),\Phi)\rd g_1\rd g_2
=\frac{L'(\tfrac{1}{2},\pi)}{b_{2r}(0)}\prod_{v}\fZ^\natural_{\pi_v,V_v}(\varphi^\tc_{1v},\varphi_{2v},\Phi_v).
\]
For $v\in\tV_F^{(\infty)}$, it is clear that $\fZ^\natural_{\pi_v,V_v}(\varphi^\tc_{1v},\varphi_{2v},\Phi_v)$ depends only on $r$, which we denote by $C_r$. By \cite{HLS11}*{(4.2.7.4)} and \cite{Gar08}*{Theorem~3.1}, we have
\[
Z(0,\varphi^\tc_{1v}\otimes\varphi_{2v},\Phi_v)=(-1)^r\pi^{r^2}\frac{\Gamma(1)\cdots\Gamma(r)}{\Gamma(r+1)\cdots\Gamma(2r)}.
\]
Here, the extra factor $(-1)^r$ is a Weil constant that appears when we turn $f^{(s)}_{\Phi_v}(\bw_r(g,1_{2r}))$ into a matrix coefficient. By \eqref{eq:lfunction} and the formula
\[
b_{2r,v}(s)=\(\prod_{i=1}^r\pi^{-(s+i)}\Gamma(s+i)\)^2,
\]
we obtain our formula for $C_r$.

By \cite{Yam14}*{Proposition~7.1 \& (7.2)}, we have $\fZ^\natural_{\pi_v,V_v}(\varphi^\tc_{1v},\varphi_{2v},\Phi_v)=1$ for $v\in\tV_F^\fin\setminus(\tR\cup\tS)$. By \cite{Liu20}*{Proposition~5.6 \& Lemma~6.1}, we have
\[
\fZ^\natural_{\pi_v,V_v}(\varphi^\tc_{1v},\varphi_{2v},\Phi_v)=\frac{(-1)^rq_v^{r-1}(q_v+1)}{(q_v^{2r-1}+1)(q_v^{2r}-1)}
\]
for $v\in\tS$. The proposition is proved.
\end{proof}

Now we study the Eisenstein series $E(s,g,\Phi)$ via Whittaker functions. For every $v\in\tV_F$, $T^\Box\in\Herm_{2r}^\circ(F_v)$, and $\Phi_v\in\sS(V_v^{2r})$, we define the local Whittaker function on $G_r^\Box(F_v)$ with parameter $s\in\dC$ as
\begin{align}\label{eq:whittaker}
W_{T^\Box}(s,g,\Phi_v)\coloneqq\int_{\Herm_{2r}(F_v)}f^{(s)}_{\Phi_v}(\tw_r^\Box n(b)g)\psi_{T^\Box}(b)^{-1}\rd b
\end{align}
(see \eqref{eq:doubling} for $\tw_r^\Box$) by meromorphic continuation, where $\r{d}b$ is the self-dual measure on $\Herm_{2r}(F_v)$ with respect to $\psi_{F,v}$. By \cite{Liu11}*{Lemma~2.8(1)}, we know that $W_{T^\Box}(s,g,\Phi_v)$ is an entire function in the variable $s$.

\begin{definition}\label{de:measure}
By the definition of local Whittaker functions \eqref{eq:whittaker}, for every $v\in\tV_F$, there exists a unique Haar measure $\r{d}h_v$ on $H(F_v)$ such that for every $T^\Box\in\Herm_{2r}^\circ(F_v)$ and every $\Phi_v\in\sS(V_v^{2r})$, we have
\[
W_{T^\Box}(0,1_{4r},\Phi_v)=\frac{\gamma_{V_v,\psi_{F,v}}^{2r}}{b_{2r,v}(0)}\int_{H(F_v)}\Phi_v(h_v^{-1}x)\rd h_v,
\]
where $x$ is an arbitrary element in $(V^{2r}_v)_{T^\Box}$ (Notation \ref{st:h}(H3)). For every open compact subgroup $L_v$ of $H(F_v)$, we denote by $\vol(L_v)$ the volume of $L_v$ under the measure $\r{d}h_v$.

By \cite{Tan99}*{Proposition~3.2}, for all but finitely many $v\in\tV_F^\fin$, a hyperspecial maximal subgroup of $H(F_v)$ has volume $1$ under $\r{d}h_v$. In particular, we may define the \emph{normalized measure}
\[
\rd^\natural h\coloneqq\frac{1}{b_{2r}(0)}\prod_{v\in\tV_F}\rd h_v
\]
on $H(\dA_F)$. In what follows, for an open compact subgroup $L$ of $H(\dA_F^\infty)$, we will denote by $\vol^\natural(L)$ the volume of $H(F_\infty)L$ under the measure $\r{d}^\natural h$.
\end{definition}

\begin{remark}\label{re:volume}
Note that when $V$ is coherent, $\r{d}^\natural h$ coincides with the Tamagawa measure on $H(\dA_F)$. Later in Definition \ref{de:natural}, we will use the volume $\vol^\natural(L)$ to scale the normalized height pairing. In view of Remark \ref{re:aipf}(2), this is the most ``natural'' way.
\end{remark}

\begin{proposition}\label{pr:analytic_kernel}
Suppose that $V$ is incoherent.
\begin{enumerate}
  \item Take an element $u\in\tV_E\setminus\tV_E^\spl$, and $\pres{u}{\Phi}=\otimes_v\pres{u}{\Phi}_v\in\sS(\pres{u}{V}^{2r}\otimes_F\dA_F)$, where we recall from Notation \ref{st:h}(H9) that $\pres{u}{V}$ is the $u$-nearby hermitian space, such that $\supp(\pres{u}{\Phi}_v)\subseteq(\pres{u}{V}^{2r}_v)_\reg$ (Notation \ref{st:h}(H3)) for $v$ in a nonempty subset $\tR'\subseteq\tR$. Then for every $g\in P_r^\Box(F_{\tR'})G_r^\Box(\dA_F^{\tR'})$, we have
      \[
      E(0,g,\pres{u}{\Phi})=\sum_{T^\Box\in\Herm_{2r}^\circ(F)}\prod_{v\in\tV_F} W_{T^\Box}(0,g_v,\pres{u}{\Phi}_v).
      \]

  \item Take $\Phi=\otimes_v\Phi_v\in\sS(V^{2r})$ such that $\supp(\Phi_v)\subseteq(V^{2r}_v)_\reg$ for $v$ in a subset $\tR'\subseteq\tR$ of cardinality at least $2$. Then for every $g\in P_r^\Box(F_{\tR'})G_r^\Box(\dA_F^{\tR'})$, we have
      \[
      E'(0,g,\Phi)=\sum_{w\in\tV_F\setminus\tV_F^\spl}\fE(g,\Phi)_w,
      \]
      where
      \[
      \fE(g,\Phi)_w\coloneqq\sum_{\substack{T^\Box\in\Herm_{2r}^\circ(F)\\\Diff(T^\Box,V)=\{w\}}}W'_{T^\Box}(0,g_w,\Phi_w)
      \prod_{v\in\tV_F\setminus\{w\}} W_{T^\Box}(0,g_v,\Phi_v).
      \]
      Here, $\Diff(T^\Box,V)$ is defined in Notation \ref{st:h}(H4).
\end{enumerate}
\end{proposition}

\begin{proof}
This is proved in \cite{Liu12}*{Section~2B}.
\end{proof}

\begin{definition}\label{de:analytic_kernel}
Suppose that $V$ is incoherent. Take an element $u\in\tV_E\setminus\tV_E^\spl$, and a pair $(T_1,T_2)$ of elements in $\Herm_r(F)$.
\begin{enumerate}
  \item For $\pres{u}{\Phi}=\otimes_v\pres{u}{\Phi}_v\in\sS(\pres{u}{V}^{2r}\otimes_F\dA_F)$, we put
    \[
    E_{T_1,T_2}(g,\pres{u}{\Phi})\coloneqq\sum_{\substack{T^\Box\in\Herm_{2r}^\circ(F)\\\partial_{r,r}T^\Box=(T_1,T_2)}}
    \prod_{v\in\tV_F} W_{T^\Box}(0,g_v,\pres{u}{\Phi}_v).
    \]

  \item For $\Phi=\otimes_v\Phi_v\in\sS(V^{2r})$, we put
    \[
    \fE_{T_1,T_2}(g,\Phi)_u\coloneqq\sum_{\substack{T^\Box\in\Herm_{2r}^\circ(F)\\\Diff(T^\Box,V)=\{\ul{u}\}\\ \partial_{r,r}T^\Box=(T_1,T_2)}}
    W'_{T^\Box}(0,g_{\ul{u}},\Phi_{\ul{u}})\prod_{v\in\tV_F\setminus\{\ul{u}\}}W_{T^\Box}(0,g_v,\Phi_v).
    \]
\end{enumerate}
Here, $\partial_{r,r}\colon\Herm_{2r}\to\Herm_r\times\Herm_r$ is defined in Notation \ref{st:f}(F2).
\end{definition}

\begin{remark}\label{re:hermitian}
The image of $\Herm_{2r}^\circ(F)^+$ under $\partial_{r,r}$ is contained in $\Herm_r^\circ(F)^+\times \Herm_r^\circ(F)^+$.
\end{remark}

The following proposition ensures the sufficient supply of test functions with support in $(V^{2r}_v)_\reg$. As we have mentioned in Section \ref{ss:introduction}, it solves a key technical challenge for our approach.

\begin{proposition}\label{pr:regular}
Let $(\pi,\cV_\pi)$ be as in Assumption \ref{st:representation}. Take $v\in\tV_F^\fin$ and suppose that $\fZ^\natural_{\pi_v,V_v}\neq 0$. Then for every $\varphi_v\in\pi_v$ and $\varphi^\vee_v\in\pi^\vee_v$ that are both nonzero, we can find elements $\phi_v,\phi^\vee_v\in\sS(V^r_v)$ such that $\supp(\phi_v\otimes\phi^\vee_v)\in(V^{2r}_v)_\reg$ and $\fZ^\natural_{\pi_v,V_v}(\varphi^\vee_v,\varphi_v,\phi_v\otimes\phi^\vee_v)\neq 0$.
\end{proposition}

\begin{proof}
To ease notation, we will suppress the place $v$ throughout the proof. We identify the $F$-vector space $\Herm_{2r}(F)$ with its dual $\Herm_{2r}(F)^\vee$ via the bilinear form $(x,y)\mapsto\tr xy$. Take an element $\Psi\in\sS(\Herm_{2r}(F))$. Let $\widehat\Psi\in\sS(\Herm_{2r}(F))$ be the Fourier transform of $\Psi$ with respect to $\psi$. Let $f_\Psi$ be the unique section in $\rI_r^\Box(0)$ such that $f_\Psi(\tw_r^\Box n(b))=\widehat\Psi(b)$ and $f_\Psi=0$ outside $P_r^\Box(F)\tw_r^\Box N_r^\Box(F)$. Take $\varphi\in\pi$ and $\varphi^\vee\in\pi^\vee$ that are both nonzero. We claim that
\begin{itemize}
  \item[$(*)$] There exists an element $\Psi\in\sS(\Herm_{2r}^\circ(F))$ such that $Z(\varphi^\vee\otimes\varphi,f_\Psi)\neq 0$.
\end{itemize}
Assuming $(*)$, we continue the proof. Let $V'$ be the other hermitian space over $F$ of rank $2r$ that is not isomorphic to $V$ if $E$ is a field, or the zero space if $E=F\times F$. Let $\Herm_{2r}^\circ(F)_V$ and $\Herm_{2r}^\circ(F)_{V'}$ be the subset of $\Herm_{2r}^\circ(F)$ that is contained in the image of the moment maps from $V^{2r}$ and $V'^{2r}$, respectively. Then $\Herm_{2r}^\circ(F)_V\cup\Herm_{2r}^\circ(F)_{V'}$ is a disjoint open cover of $\Herm_{2r}^\circ(F)$. Choose $\Psi$ as in the claim and put $\Psi_V\coloneqq\Psi\cdot\CF_{\Herm_{2r}^\circ(F)_V}$ and $\Psi_{V'}\coloneqq\Psi\cdot\CF_{\Herm_{2r}^\circ(F)_{V'}}$. We may choose elements $\Phi_V\in\sS(V^{2r}_\reg)$ and $\Phi_{V'}\in\sS(V'^{2r}_\reg)$ such that $\Psi_V$ and $\Psi_{V'}$ are the pushforward of $\Phi_V$ and $\Phi_{V'}$ along the moment map $V^{2r}_\reg\to\Herm_{2r}^\circ(F)_V$ and $V'^{2r}_\reg\to\Herm_{2r}^\circ(F)_{V'}$, respectively. It is easy to see that $f_{\Phi_V}=f_{\Psi_V}$ and $f_{\Phi_{V'}}=f_{\Psi_{V'}}$. In particular, we have
\begin{align*}
Z(\varphi^\vee\otimes\varphi,f_\Psi)=Z(\varphi^\vee\otimes\varphi,f_{\Psi_V})+Z(\varphi^\vee\otimes\varphi,f_{\Psi_{V'}})
=Z(\varphi^\vee\otimes\varphi,f_{\Phi_V})+Z(\varphi^\vee\otimes\varphi,f_{\Phi_{V'}}).
\end{align*}
By Proposition \ref{pr:uniqueness} and Remark \ref{re:yamana}, we have $Z(\varphi^\vee\otimes\varphi,f_{\Phi_{V'}})=0$ if $\fZ^\natural_{\pi,V}\neq 0$. Thus, we have $Z(\varphi^\vee\otimes\varphi,f_{\Phi_V})\neq 0$ hence $\fZ^\natural_{\pi,V}(\varphi^\vee\otimes\varphi,\Phi_V)\neq 0$. The theorem follows as $\Phi_V$ can be written as a finite sum of elements of the form $\phi\otimes\phi^\vee$ satisfying $\supp(\phi\otimes\phi^\vee)\subseteq\supp(\Phi_V)\subseteq V^{2r}_\reg$.

It remains to show $(*)$. We identify $\Res_{E/F}\Mat_{r,r}\times\Herm_r\times\Herm_r$ with $\Herm_{2r}$ via the assignment
\[
(a,u,v)\mapsto
\begin{pmatrix}
u & \pres{\rt}{a}^\tc \\
a & v
\end{pmatrix}.
\]
Define a polynomial function $\Delta$ on $\Herm_{2r}$ sending $(a,u,v)$ to $\Nm_{E/F}\dtm a$. Let $\Omega$ be the complement of the Zariski closed subset of $\Herm_{2r}$ defined by the ideal $(\Delta)$. We define a morphism $\iota\colon\Omega\to G_r$ such that
\[
\iota(a,u,v)=
\begin{pmatrix}
    1_r & v \\
    0 & 1_r \\
\end{pmatrix}
\begin{pmatrix}
    -a & 0 \\
    0 & -\pres{\rt}{a}^{\tc,-1} \\
\end{pmatrix}
\tw_r
\begin{pmatrix}
    1_r & u \\
    0 & 1_r \\
\end{pmatrix},
\]
which is an isomorphism onto the Zariski open subset $P_r\tw_rN_r$ of $G_r$. By a direct computation, we have a unique morphism $p\colon\Omega\to P_r^\Box$ such that
\begin{align}\label{eq:regular1}
\bw_r\iota(a,u,v)=p(a,u,v)\cdot\tw_r^\Box\cdot
n\begin{pmatrix}
u & \pres{\rt}{a}^\tc \\
a & v
\end{pmatrix},
\end{align}
which satisfies
\begin{align}\label{eq:regular2}
\Nm_{E/F}\delta^\Box_r(p(a,u,v))=\Nm_{E/F}\dtm a=\Delta(a,u,v).
\end{align}
Define a locally constant function $\xi_{\varphi^\vee,\varphi}$ on $G_r(F)$ by $\xi_{\varphi^\vee,\varphi}(g)\coloneqq\langle\pi^\vee(g)\varphi^\vee,\varphi\rangle_\pi$. Then by \eqref{eq:regular1} and \eqref{eq:regular2}, we have
\begin{align*}
Z(\varphi^\vee\otimes\varphi,f_\Psi)&=\int_{G_r(F)}\xi_{\varphi^\vee,\varphi}(g)f_\Psi(\bw_r(g,1_{2r}))\rd g \\
&=\int_{P_r(F)\tw_rN_r(F)}\xi_{\varphi^\vee,\varphi}(\iota(a,u,v))|\Delta(a,u,v)|_F^r\widehat\Psi(a,u,v)\cdot\rd\iota(a,u,v).
\end{align*}
We define a locally constant function $\xi_{\varphi^\vee,\varphi}^\flat$ on $\Omega(F)$ by
\[
\xi_{\varphi^\vee,\varphi}^\flat(a,u,v)=|\Delta(a,u,v)|_F^{-r}\xi_{\varphi^\vee,\varphi}(\iota(a,u,v)).
\]
Note that there exists a unique Haar measure $\rd a\rd u\rd v$ on $\Herm_{2r}(F)$ such that
\[
\rd\iota(a,u,v)=|\Delta(a,u,v)|_F^{-2r}\rd a\rd u\rd v.
\]
Thus, we have
\begin{align}\label{eq:regular3}
Z(\varphi^\vee\otimes\varphi,f_\Psi)=
\int_{\Omega(F)}\xi_{\varphi^\vee,\varphi}^\flat(a,u,v)\widehat\Psi(a,u,v)\rd a\rd u\rd v.
\end{align}
As both $\varphi^\vee$ and $\varphi$ are nonzero, $\xi_{\varphi^\vee,\varphi}^\flat$ is nonzero, which is also locally integrable on $\Herm_{2r}(F)$ by Remark \ref{re:yamana}. The remaining discussion bifurcates.

When $E$ is a field, we have $\xi_{\varphi^\vee,\varphi}\in L^{2+\varepsilon}(G_r(F))$ for every $\varepsilon>0$, which is equivalent to
\[
\int_{\Omega(F)}|\xi_{\varphi^\vee,\varphi}^\flat(a,u,v)|^{2+\varepsilon}|\Delta(a,u,v)|_F^{r\varepsilon}\rd a\rd u\rd v<\infty.
\]
Applying Lemma \ref{le:fourier1} to $X=\Herm_{2r}(F)$, we obtain an element $\Psi\in\sS(V^{2r}_\reg)$ such that
\[
\int_{\Omega(F)}\xi_{\varphi^\vee,\varphi}^\flat(a,u,v)\widehat\Psi(a,u,v)\rd a\rd u\rd v\neq 0,
\]
which implies $Z(\varphi^\vee\otimes\varphi,f_\Psi)\neq 0$ by \eqref{eq:regular3}. Thus, $(*)$ is proved.

When $E=F\times F$, we have $\xi_{\varphi^\vee,\varphi}\in L^{2+\varepsilon}(Z_r(F)\backslash G_r(F))$ for every $\varepsilon>0$, where $Z_r$ denotes the center of $G_r$, which is equivalent to
\[
\int_{F^\times\backslash\Omega(F)}|\xi_{\varphi^\vee,\varphi}^\flat(a,u,v)|^{2+\varepsilon}|\Delta(a,u,v)|_F^{r\varepsilon}\rd a\rd u\rd v<\infty.
\]
Here, the action of $F^\times$ on $\Herm_{2r}(F)$ is given as follows: After identifying $\Mat_{r,r}(E)$ with $\Mat_{r,r}(F)\times\Mat_{r,r}(F)$ via the two factors of $F$ under which we write $a=(a_1,a_2)$, $\alpha\in F^\times$ sends $((a_1,a_2),u,v)$ to $((\alpha a_1,\alpha^{-1}a_2),u,v)$. Applying Lemma \ref{le:fourier2} to $X=\Herm_{2r}(F)$ with $X_1=\Mat_{r,r}(F)$, $X_2=\Mat_{r,r}(F)$, and $X_3=\Herm_r(F)\oplus\Herm_r(F)$, we obtain an element $\Psi\in\sS(V^{2r}_\reg)$ such that
\[
\int_{\Omega(F)}\xi_{\varphi^\vee,\varphi}^\flat(a,u,v)\widehat\Psi(a,u,v)\rd a\rd u\rd v\neq 0,
\]
which implies $Z(\varphi^\vee\otimes\varphi,f_\Psi)\neq 0$ by \eqref{eq:regular3}. Thus, $(*)$ is proved.
\end{proof}

To end this section, we recall some constructions concerning the tempered global $L$-packet given by $\pi$, which will be used in Section \ref{ss:height} and Section \ref{ss:inert2}.

\begin{notation}\label{co:galois}
Let $(\pi,\cV_\pi)$ be as in Assumption \ref{st:representation}.
\begin{enumerate}
  \item Let $\Pi$ be the automorphic base change of $\pi$, that is, the isobaric automorphic representation of $\GL_n(\dA_E)$ such that $\Pi_v$ is the standard base change of $\pi_v$ for all but finitely many $v\in\tV_F^\fin$ for which $\pi_v$ is unramified.\footnote{The existence of $\Pi$ follows from \cite{Shi} or more generally \cite{KMSW}, while the uniqueness of $\Pi$ up to isomorphism is ensured by the strong multiplicity one theorem.} By the local-global compatibility \cite{KMSW}*{Theorem~1.7.1}, for every $v\in\tV_F^{(\infty)}$, $\Pi_v$ is the normalized induction of $\arg^{n-1}\boxtimes\arg^{n-3}\boxtimes\cdots\boxtimes\arg^{3-n}\boxtimes\arg^{1-n}$ as in Definition \ref{de:lfunction}.

  \item Put $\fI\coloneqq\{n-1,n-3,\dots,3-n,1-n\}$. For each character $\chi\colon\mu_2^{\fI}\to\dC^\times$, we define the signature of $\chi$ to be the pair $(p,q)$ with $p+q=n$ such that $\chi$ takes value $1$ on $p$ $\mu_2$-generators hence $-1$ on $q$ $\mu_2$-generators. For such a character $\chi$ of signature $(p,q)$, we have a discrete series representation $\pi^\chi$ of $\rU(p,q)$. When $(p,q)=(n-1,1)$, we denote by $\pi^\chi_\infty$ the representation of $\pres{\bu}{H}(F_\infty)$ that is the inflation of $\pi^\chi$ along the quotient map $\pres{\bu}{H}(F_\infty)\to\pres{\bu}{H}(\dR)\simeq\rU(n-1,1)$.

  \item We may write $\Pi=\Pi_1\boxplus\cdots\boxplus\Pi_s$, in which $\Pi_j$ is a conjugate-selfdual cuspidal automorphic representation of $\GL_{n_j}(\dA_E)$, with $n_1+\cdots+n_s=n$. Then there is a unique partition $\fI=\fI_1\sqcup\cdots\sqcup\fI_s$ such that $\Pi_{j,w}$ is is the normalized induction of $\boxtimes_{i\in\fI_j}\arg^i$ for $1\leq j\leq s$.

  \item Let $\ell$ be a rational prime with an arbitrarily given isomorphism $\ol\dQ_\ell\simeq\dC$. For every $1\leq j\leq s$, we have a (semisimple) Galois representation
      \[
      \rho_{\Pi_j}\colon\Gal(\dQ^\ac/E)\to\GL_{n_j}(\ol\dQ_\ell)
      \]
      attached to $\Pi_j$ as described in \cite{Car12}*{Theorem~1.1}.
\end{enumerate}
Here, we recall from Notation \ref{st:f} that we have regarded $E$ as a subfield of $\dC$ via $\biota$.
\end{notation}

\begin{lem}\label{le:arthur}
For every irreducible admissible representation $\tilde\pi^\infty$ of $H(\dA_F^\infty)$ such that $\Pi_v$ is the standard base change of $\tilde\pi^\infty_v$ for all but finitely many $v\in\tV_F^\fin$ for which $\tilde\pi^\infty_v$ is unramified, exactly one of the following two cases happens:
\begin{enumerate}[label=(\alph*)]
  \item There does not exist a character $\chi\colon\mu_2^{\fI}\to\dC^\times$ of signature $(n-1,1)$ such that $\pi^\chi_\infty\otimes\tilde\pi^\infty$ is a cuspidal automorphic representation of $\pres{\bu}{H}(\dA_F)$.

  \item There is a unique integer $1\leq j(\tilde\pi^\infty)\leq s$, such that for every $\chi\colon\mu_2^{\fI}\to\dC^\times$ of signature $(n-1,1)$, $\pi^\chi_\infty\otimes\tilde\pi^\infty$ is a cuspidal automorphic representation of $\pres{\bu}{H}(\dA_F)$ if and only if the unique $\mu_2$-generator of $\mu_2^\fI$ on which $\chi$ takes value $-1$ is indexed by an element in $\fI_{j(\tilde\pi^\infty)}$.
\end{enumerate}
Moreover, we have
\[
\dim_\dC\Hom_{H(\dA_F^\infty)}\(\tilde\pi^\infty,\varinjlim_L\rH^{n-1}_{\dr}(X_L/\dC)\)=
\begin{dcases}
0, &\text{if $\tilde\pi^\infty$ fits in (a)} \\
n_{j(\tilde\pi^\infty)}>0, &\text{if $\tilde\pi^\infty$ fits in (b)}
\end{dcases}
\]
where $\{X_L\}$ is the unitary Shimura variety recalled in Section \ref{ss:special} below.
\end{lem}

\begin{proof}
The first part of the lemma is a consequence of Arthur's multiplicity formula for tempered global $L$-packets \cite{KMSW}*{Theorem~1.7.1}.

The second part of the lemma follows from Matsushima's formula and Arthur's multiplicity formula \cite{KMSW}*{Theorem~1.7.1}. In particular, the number of characters $\chi\colon\mu_2^{\fI}\to\dC^\times$ of signature $(n-1,1)$ such that $\pi^\chi_\infty\otimes\tilde\pi^\infty$ is a cuspidal automorphic representation of $\pres{\bu}{H}(\dA_F)$ equals $n_{j(\tilde\pi^\infty)}$. Here, we also use the well-known fact that the $(\fg,K)$-cohomology of a (cohomological) discrete series representation is one-dimensional in the middle degree and vanishes in all other degrees (see, for example, \cite{BW00}*{II.~Theorem~5.4}).
\end{proof}

\begin{remark}\label{re:root}
Assume that $\tilde\pi^\infty\coloneqq\Hom_{G_r(\dA_F^\infty)}(\sS(V^r\otimes_F\dA_F^\infty),\pi^\infty)$ is nonzero, which is then an irreducible admissible representation of $H(\dA_F^\infty)$ by Proposition \ref{pr:uniqueness}(3). By \cite{GI16}*{Theorem~4.1(ii)}, the global root number $\varepsilon(\Pi)$ equals $-1$. Moreover, using Arthur's multiplicity formula \cite{KMSW}*{Theorem~1.7.1} and Conjecture $\text{(P1)}_n$ in \cite{GI16}*{Section~4.4} (which is proved in that article), we see that $\tilde\pi^\infty$ fits in the situation (b) of Lemma \ref{le:arthur} if and only if there is exactly one element $j\in\{1,\dots,s\}$ such that $\varepsilon(\Pi_j)=-1$; and in this case we must have $j=j(\tilde\pi^\infty)$. When $\ord_{s=\frac{1}{2}}L(s,\Pi)=1$, there is exactly one element $j\in\{1,\dots,s\}$ such that $\varepsilon(\Pi_j)=-1$ as $L(s,\Pi)=\prod_{j=1}^sL(s,\Pi_j)$, so $\tilde\pi^\infty$ fits in the situation (b) of Lemma \ref{le:arthur} automatically.
\end{remark}

\section{Special cycles and generating functions}
\label{ss:special}

In this section, we review the construction of Kudla's special cycles and generating functions. We also introduce the hypothesis on the modularity of generating functions and derive some of its consequences. From now to the end of Section \ref{ss:main}, we assume $V$ incoherent.

Recall that we have fixed a $\bu$-nearby space $\pres{\bu}V$ and an isomorphism $\pres{\bu}{V}\otimes_F\dA_F^{\ul{\bu}}\simeq V\otimes_{\dA_F}\dA_F^{\ul{\bu}}$ from Notation \ref{st:h}(H9). For every open compact subgroup $L\subseteq H(\dA_F^\infty)$, we have the Shimura variety $X_L$ associated to $\Res_{F/\dQ}\pres{\bu}{H}$ of the level $L$, which is a smooth quasi-projective scheme over $E$ (which is regarded as a subfield of $\dC$ via $\biota$) of dimension $n-1$. We remind the readers its complex uniformization
\begin{align}\label{eq:complex_uniformization}
(X_L\otimes_E\dC)^{\r{an}}\simeq\pres{\bu}{H}(F)\backslash\fD\times H(\dA_F)/L,
\end{align}
where $\fD$ denotes the complex manifold of negative lines in $\pres{\bu}{V}\otimes_E\dC$ and the Deligne homomorphism is the one adopted in \cite{LTXZZ}*{Section~3.2}. In what follows, for a place $u\in\tV_E$, we put $X_{L,u}\coloneqq X_L\otimes_EE_u$ as a scheme over $E_u$.

Now we recall the construction of Kudla's special cycles and their generating functions. Take an integer $1\leq m\leq n-1$.

\begin{definition}\label{co:special_cycle}
For every element $x\in V^m\otimes_{\dA_F}\dA_F^\infty$, we have the \emph{special cycle} $Z(x)_L\in\CH^m(X_L)_\dQ$ defined as follows.
\begin{itemize}
  \item For $T(x)\not\in\Herm_m(F)^+$, we set $Z(x)_L=0$.

  \item For $T(x)\in\Herm^\circ_m(F)^+$, we may find elements $x'\in\pres{\bu}{V}^m$ and $h\in H(\dA_F^\infty)$ such that $hx=x'$ holds in $V^m\otimes_{\dA_F}\dA_F^\infty$. The components of $x'$ spans a totally positive definite hermitian subspace $V_{x'}$ of $\pres{\bu}{V}$ of rank $m$. Put $H^{x'}\coloneqq\rU(V_{x'}^\perp)$ which is naturally a subgroup of $\pres{\bu}{H}$, and let $\{X^{x'}_M\}_{M\subseteq H^{x'}(\dA_F^\infty)}$ be the associated system of Shimura varieties. Define $Z(x)_L$ to be the image cycle of the composite morphism
      \[
      X^{x'}_{hLh^{-1}\cap H^{x'}(\dA_F^\infty)}\to X_{hLh^{-1}}\xrightarrow{\cdot h}X_L.
      \]
      It is straightforward to check that $Z(x)_L$ does not depend on the choice of $x'$ and $h$. Moreover, $Z(x)_L$ is a well-defined element in $\rZ^m(X_L)$.

  \item For $T(x)\in\Herm_m(F)^+$ in general, we have an element $Z(x)_L\in\CH^m(X_L)_\dQ$ (not well-defined in $\rZ^m(X_L)_\dQ$). We refer the readers to \cite{Liu11}*{Section~3A} for more details as it is not important to us in this article.
\end{itemize}
For every $\phi^\infty\in\sS(V^m\otimes_{\dA_F}\dA_F^\infty)^L$ and $T\in\Herm_m(F)$, we put
\[
Z_T(\phi^\infty)_L\coloneqq\sum_{\substack{x\in L\backslash V^m\otimes_{\dA_F}\dA_F^\infty\\ T(x)=T}}\phi^\infty(x) Z(x)_L.
\]
As the above summation is finite, $Z_T(\phi^\infty)_L$ is a well-defined element in $\CH^m(X_L)_\dC$.
\end{definition}

\begin{remark}\label{re:special_cycle}
For $T\in\Herm_m^\circ(F)^+$, $Z_T(\phi^\infty)_L$ is even a well-defined element in $\rZ^m(X_L)_\dC$.
\end{remark}

Finally, for every $g\in G_m(\dA_F)$, Kudla's \emph{generating function} is defined to be
\[
Z_{\phi^\infty}(g)_L\coloneqq\sum_{T\in\Herm_m(F)^+}\omega_{m,\infty}(g_\infty)\phi^0_\infty(T)\cdot
Z_T(\omega_m^\infty(g^\infty)\phi^\infty)_L
\]
as a formal sum valued in $\CH^m(X_L)_\dC$, where
\[
\omega_{m,\infty}(g_\infty)\phi^0_\infty(T)\coloneqq\prod_{v\in\tV_F^{(\infty)}}\omega_{m,v}(g_v)\phi^0_v(T).
\]
Here, we note that for $v\in\tV_F^{(\infty)}$, the function $\omega_{m,v}(g_v)\phi^0_v$ factors through the moment map $V_v^m\to\Herm_m(F_v)$ (see Notation \ref{st:h}(H1)) hence $\omega_{m,v}(g_v)\phi^0_v(T)$ makes sense.

\begin{lem}\label{le:invariance}
In Definition \ref{co:special_cycle}, we have
\[
\omega_{m,\infty}(n(b)m(a))\phi^0_\infty(T)\cdot Z_T(\omega_m(n(b)m(a))\phi^\infty)_L
=\phi^0_\infty(\pres{\rt}{a}^\tc Ta)\cdot Z_{\pres{\rt}{a}^\tc Ta}(\phi^\infty)_L
\]
for every $a\in\GL_m(E)$ and every $b\in\Herm_m(F)$.
\end{lem}

\begin{proof}
This is proved in (the proof of) \cite{Liu11}*{Theorem~3.5}.
\end{proof}

\begin{lem}\label{re:translate}
In Definition \ref{co:special_cycle}, we have $\rt^*Z_T(\phi^\infty)_L=Z_T(\rt\phi^\infty)_L$ for every $\rt\in\dT^\tR_\dC$.
\end{lem}

\begin{proof}
By linearity, we may assume $\rt=\CF_{LhL}$ for some $h\in H(\dA_F^{\infty,\tR})$. Then it follows easily from Definition \ref{co:special_cycle} that
\[
\rt^*Z(x)_L=\sum_{h'\in h^{-1}Lh\cap L\backslash L}Z(hh'x)_L
\]
for every $x\in V^m\otimes_{\dA_F}\dA_F^\infty$. Thus,
\begin{align*}
\rt^*Z_T(\phi^\infty)_L
&=\sum_{\substack{x\in L\backslash V^m\otimes_{\dA_F}\dA_F^\infty\\ T(x)=T}}\phi^\infty(x)\rt^*Z(x)_L \\
&=\sum_{\substack{x\in L\backslash V^m\otimes_{\dA_F}\dA_F^\infty\\ T(x)=T}}\phi^\infty(x)\sum_{h'\in h^{-1}Lh\cap L\backslash L}Z(hh'x)_L \\
&=\sum_{\substack{x\in L\backslash V^m\otimes_{\dA_F}\dA_F^\infty\\ T(x)=T}}
\(\sum_{h'\in h^{-1}Lh\cap L\backslash L}\phi^\infty(h'^{-1}h^{-1}x)\)Z(x)_L \\
&=\sum_{\substack{x\in L\backslash V^m\otimes_{\dA_F}\dA_F^\infty\\ T(x)=T}}(\rt\phi^\infty)(x)Z(x)_L =Z_T(\rt\phi^\infty)_L.
\end{align*}
The lemma follows.
\end{proof}

\begin{hypothesis}[Modularity of generating functions of codimension $m$]\label{hy:modularity}
For every open compact subgroup $L\subseteq H(\dA_F^\infty)$, every $\phi^\infty\in\sS(V^m\otimes_{\dA_F}\dA_F^\infty)^L$, and every complex linear map $l\colon\CH^m(X_L)_\dC\to\dC$, the assignment
\[
g\mapsto l(Z_{\phi^\infty}(g)_L)
\]
is absolutely convergent, and gives an element in $\cA^{[r]}(G_m(F)\backslash G_m(\dA_F))$. In other words, the function $Z_{\phi^\infty}(-)_L$ defines an element in $\Hom_\dC(\CH^m(X_L)_\dC^\vee,\cA^{[r]}(G_m(F)\backslash G_m(\dA_F)))$.
\end{hypothesis}

\begin{remark}\label{re:modularity}
Hypothesis \ref{hy:modularity} is believed to hold. In fact, in the case of symplectic groups over $\dQ$, the analogous statement has been confirmed in \cite{BW15}. In our situation, Hypothesis \ref{hy:modularity} is proved in \cite{Liu11}*{Theorem~3.5} for $m=1$; for $m\geq 2$, we know that $l(Z_{\phi^\infty}(g)_L)$ is formally modular by \cite{Liu11}*{Theorem~3.5}.
\end{remark}

Note that the natural inclusion
\[
\cA^{[r]}(G_m(F)\backslash G_m(\dA_F))\otimes_\dC\CH^m(X_L)_\dC\subseteq
\Hom_\dC(\CH^m(X_L)_\dC^\vee,\cA^{[r]}(G_m(F)\backslash G_m(\dA_F)))
\]
might be proper, since we do not know whether $\CH^m(X_L)_\dC$ is finite dimensional. However, we have the following result.

\begin{proposition}\label{pr:modularity}
Assume Hypothesis \ref{hy:modularity} on the modularity of generating functions of codimension $m$.
\begin{enumerate}
  \item For every open compact subgroup $L\subseteq H(\dA_F^\infty)$ and every $\phi^\infty\in\sS(V^m\otimes_{\dA_F}\dA_F^\infty)^L$, $Z_{\phi^\infty}(-)_L$ belongs to $\cA^{[r]}(G_m(F)\backslash G_m(\dA_F))\otimes_\dC\CH^m(X_L)_\dC$.

  \item The map
      \[
      \sS(V^m\otimes_{\dA_F}\dA_F^\infty)^L\to\cA^{[r]}(G_m(F)\backslash G_m(\dA_F))\otimes_\dC\CH^m(X_L)_\dC
      \]
      sending $\phi^\infty$ to $Z_{\phi^\infty}(-)_L$ is $G_m(\dA_F^\infty)\times \dT^\tR_\dC$-equivariant, where $G_m(\dA_F^\infty)$ acts on the source via the Weil representation and on the target via the right translation on the first factor; and $\dT^\tR_\dC$ acts on the source via the Weil representation and on the target via the Hecke correspondences.
\end{enumerate}
\end{proposition}

\begin{proof}
For (1), fix an open compact subgroup $K\subseteq G_m(\dA_F^\infty)$ that fixes $\phi^\infty$, and a set of representatives $\{g^{(1)},\dots,g^{(s)}\}$ of the finite double coset $G_m(F)\backslash G_m(\dA_F^\infty)/K$. For every $1\leq i\leq s$, the restriction of $Z_{\phi^\infty}(-)_L$ to $G_m(F_\infty)\times\{g^{(i)}\}$ is given by the hermitian $q$-expansion
\[
f^{(i)}(q)\coloneqq\sum_{T\in\Herm_m(F)^+}Z_T(\phi^{\infty(i)})_L \cdot q^T,
\]
where $\phi^{\infty(i)}\coloneqq\omega_m^\infty(g^{(i)})\phi^\infty$. By Hypothesis \ref{hy:modularity}, for every $l\in\CH^m(X_L)_\dC^\vee$, the $q$-expansion
\[
l(f^{(i)})(q)\coloneqq\sum_{T\in\Herm_m(F)^+}l(Z_T(\phi^{\infty(i)})_L) \cdot q^T
\]
belongs to $\cM^{[r]}_m(\Gamma^{(i)})$, the space of holomorphic hermitian Siegel modular form of $G_m$ of weight $(\kappa_{m,w}^r)_{w\in\tV_F^{(\infty)}}$ (Notation \ref{st:g}(G6)) and level $\Gamma^{(i)}\coloneqq G_m(F)\cap g^{(i)}K(g^{(i)})^{-1}$. Let $\cM^{(i)}$ be the subspace of $\CH^m(X_L)_\dC$ spanned by $Z_T(\phi^{\infty(i)})_L$ for all $T\in\Herm_m(F)^+$. We claim that
\begin{align}\label{eq:modularity}
\dim_\dC \cM^{(i)}\leq \dim_\dC \cM^{[r]}_m(\Gamma^{(i)}) < \infty.
\end{align}
Take arbitrary elements $l_1,\dots,l_d$ of $\CH^m(X_L)_\dC^\vee$ with $d>\dim_\dC \cM^{[r]}_m(\Gamma^{(i)})$. Then there exist $c_1,\dots,c_d\in\dC$ not all zero, such that $\sum_{j=1}^dc_jl_j(f^{(i)})(q)=0$; in other words,
\[
\sum_{j=1}^dc_jl_j(Z_T(\phi^{\infty(i)})_L)=0,\qquad\forall T\in\Herm_m(F)^+.
\]
Thus, we have $\sum_{j=1}^dc_jl_j\res_{\cM^{(i)}}=0$, which implies \eqref{eq:modularity}. However, \eqref{eq:modularity} implies that the subspace of $\CH^m(X_L)_\dC$ generated by $Z_T(\omega_m^\infty(g^\infty)\phi^\infty)_L$ for all $T\in\Herm_m(F)^+$ and $g^\infty\in G_m(\dA_F^\infty)$ is finite dimensional. Thus, (1) follows.

Part (2) is follows from Lemma \ref{re:translate} and the construction.

The proposition follows.
\end{proof}

\begin{definition}\label{de:arithmetic_theta}
Let $(\pi,\cV_\pi)$ be as in Assumption \ref{st:representation}. Assume Hypothesis \ref{hy:modularity} on the modularity of generating functions of codimension $r$. For every $\varphi\in\cV_\pi^{[r]}$, every open compact subgroup $L\subseteq H(\dA_F^\infty)$, and every $\phi^\infty\in\sS(V^r\otimes_{\dA_F}\dA_F^\infty)^L$, we put
\[
\Theta_{\phi^\infty}(\varphi)_L\coloneqq\int_{G_r(F)\backslash G_r(\dA_F)}
\varphi^\tc(g)Z_{\phi^\infty}(g)_L\rd g,
\]
which is an element in $\CH^r(X_L)_\dC$ by Proposition \ref{pr:modularity}. It is clear that the image of $\Theta_{\phi^\infty}(\varphi)_L$ in
\[
\CH^r(X)_\dC\coloneqq\varinjlim_L\CH^r(X_L)_\dC
\]
depends only on $\varphi$ and $\phi^\infty$, which we denote by $\Theta_{\phi^\infty}(\varphi)$. Finally, we define the \emph{arithmetic theta lifting} of $(\pi,\cV_\pi)$ to $V$ (with respect to $\biota$) to be the complex subspace $\Theta(\pi,V)$ of $\CH^r(X)_\dC$ spanned by $\Theta_{\phi^\infty}(\varphi)$ for all $\varphi\in\cV_\pi^{[r]}$ and $\phi^\infty\in\sS(V^r\otimes_{\dA_F}\dA_F^\infty)$.
\end{definition}

\section{Auxiliary Shimura variety}
\label{ss:auxiliary}

In this section, we introduce an auxiliary Shimura variety that will \emph{only} be used in the computation of local indices $I_{T_1,T_2}(\phi^\infty_1,\phi^\infty_2,\rs_1,\rs_2,g_1,g_2)_{L,u}$ to be introduced in the next sectoin. We continue the discussion from Section \ref{ss:special}.

\begin{notation}\label{co:auxiliary}
We denote by $\rT_0$ the torus over $\dQ$ such that for every commutative $\dQ$-algebra $R$, we have
$\rT_0(R)=\{a\in E\otimes_\dQ R\res \Nm_{E/F}a \in R^\times\}$.

We choose a CM type $\Phi$ of $E$ containing $\biota$ and denote by $E'$ the subfield of $\dC$ generated by $E$ and the reflex field of $\Phi$. We also choose a skew hermitian space $W$ over $E$ of rank $1$, whose group of rational similitude is canonically $\rT_0$. For a (sufficiently small) open compact subgroup $L_0$ of $\rT_0(\dA^\infty)$, we have the PEL type moduli scheme $Y$ of CM abelian varieties with CM type $\Phi$ and level $L_0$, which is a smooth projective scheme over $E'$ of dimension $0$ (see, \cite{Kot92}, for example). In what follows, when we invoke this construction, the data $\Phi$, $W$, and $L_0$ will be fixed, hence will not be carried into the notation $E'$ and $Y$. For every open compact subgroup $L\subseteq H(\dA_F^\infty)$, we put
\[
X'_L\coloneqq X_L\otimes_EY
\]
as a scheme over $E'$.
\end{notation}

Unlike $X_L$, the scheme $X'_L$ has a moduli interpretation as first observed in \cite{RSZ}.

\begin{lem}\label{le:moduli_auxiliary}
The $E'$-scheme $X'_L$ represents the functor that assigns to every locally Noetherian scheme $S$ over $E'$ the set of equivalence classes of sextuples $(A_0,\lambda_0,\eta_0;A,\lambda,\eta)$ where
\begin{itemize}
  \item $(A_0,\lambda_0,\eta_0)$ is an element in $Y(S)$;

  \item $(A,\lambda)$ is a unitary $O_E$-abelian scheme of signature type $n\Phi-\biota+\biota^\tc$ over $S$ (see \cite{LTXZZ}*{Definition~3.4.2 \& Definition~3.4.3});

  \item $\eta$ is an $L$-level structure, that is, for a chosen geometric point $s$ on every connected component of $S$, a $\pi_1(S,s)$-invariant $L$-orbit of isomorphisms
      \[
      \eta\colon V\otimes_{\dA_F}\dA_F^\infty\to
      \Hom_{E\otimes_\dQ\dA^\infty}^{\lambda_0,\lambda}(\rH_1(A_{0s},\dA^\infty),\rH_1(A_s,\dA^\infty))
      \]
      of hermitian spaces over $E\otimes_\dQ\dA^\infty=E\otimes_F\dA_F^\infty$ (see \cite{LTXZZ}*{Construction~3.4.4} for the hermitian form on the target of $\eta$).
\end{itemize}
Two sextuples $(A_0,\lambda_0,\eta_0;A,\lambda,\eta)$ and $(A'_0,\lambda'_0,\eta'_0;A',\lambda',\eta')$ are equivalent if there are $O_F$-linear quasi-isogenies $\varphi_0\colon A_0\to A'_0$ and $\varphi\colon A\to A'$ such that
\begin{itemize}
  \item $\varphi_0$ carries $\eta_0$ to $\eta'_0$;

  \item there exists $c\in\dQ^\times$ such that $\varphi_0^\vee\circ\lambda'_0\circ\varphi_0=c\lambda_0$ and $\varphi^\vee\circ\lambda'\circ\varphi=c\lambda$;

  \item the $L$-orbit of maps $v\mapsto\varphi_*\circ\eta(v)\circ(\varphi_{0*})^{-1}$ for $v\in V\otimes_{\dA_F}\dA_F^\infty$ coincides with $\eta'$.
\end{itemize}
\end{lem}

\begin{proof}
This is shown in \cite{RSZ}*{Section~3.2}. See also \cite{LTXZZ}*{Section~4.1}.
\end{proof}

\begin{definition}\label{de:moduli_cycle}
For every $x\in V^r\otimes_{\dA_F}\dA_F^\infty$ with $T(x)\in\Herm_r^\circ(F)^+$, we define a moduli functor $Z'(x)_L$ over $E'$ as follows: for every locally Noetherian scheme $S$ over $E'$, $Z'(x)_L(S)$ is the set of equivalence classes of septuples $(A_0,\lambda_0,\eta_0;A,\lambda,\eta;\tilde{x})$ where
\begin{itemize}
  \item $(A_0,\lambda_0,\eta_0;A,\lambda,\eta)$ belongs to $X'_L(S)$;

  \item $\tilde{x}$ is an element in $\Hom_{O_E}(A_0^r,A)_\dQ$ satisfying $\tilde{x}_*\in\eta(Lx)$.
\end{itemize}
By Lemma \ref{le:moduli_cycle}(1) below, the image of $Z'(x)_L$ defines an element in $\rZ^r(X'_L)$ which we denote by $Z(x)'_L$.
\end{definition}

\begin{lem}\label{le:moduli_cycle}
For every $x\in V^r\otimes_{\dA_F}\dA_F^\infty$ with $T(x)\in\Herm_r^\circ(F)^+$, we have
\begin{enumerate}
  \item the forgetful morphism $Z'(x)_L\to X'_L$ is finite and unramified;

  \item the restriction of the algebraic cycle $Z(x)_L$ to $X'_L$ coincides with $Z(x)'_L$, as elements in $\rZ^r(X'_L)$.
\end{enumerate}
\end{lem}

\begin{proof}
In the proof below, we will frequently use notations from Definition \ref{co:special_cycle}. Take $x'\in\pres{\bu}{V}^r$ and $h\in H(\dA_F^\infty)$ such that $hx=x'$ holds in $V^m\otimes_{\dA_F}\dA_F^\infty$. For both statements, it suffices to show that there is an isomorphism $Z'(x)_L\xrightarrow{\sim}X^{x'}_{hLh^{-1}\cap H^{x'}(\dA_F^\infty)}\otimes_EY$, rendering the following diagram
\begin{align}\label{eq:moduli_cycle}
\xymatrix{
Z'(x)_L \ar[rr]^-{\sim}\ar[dr] && X^{x'}_{hLh^{-1}\cap H^{x'}(\dA_F^\infty)}\otimes_EY \ar[dl]^-{\cdot h} \\
& X'_L
}
\end{align}
commute. Let $S$ be a locally Noetherian scheme over $E'$. We will construct a functorial bijection between $Z'(x)_L(S)$ and $(X^{x'}_{hLh^{-1}\cap H^{x'}(\dA_F^\infty)}\otimes_EY)(S)$.

Take an element $(A_0,\lambda_0,\eta_0;A,\lambda,\eta;\tilde{x})\in Z'(x)_L(S)$. We may find an $O_E$-abelian scheme $A_1$ of signature type $r\Phi-\biota+\biota^\tc$ over $S$, and an element $\tilde{x}_1\in\Hom_{O_E}(A_1,A)_\dQ$, such that $\tilde{x}_1\oplus\tilde{x}$ is an isomorphism in $\Hom_{O_E}(A_1\times A_0^r,A)_\dQ$, and that the composition $\tilde{x}^\vee\circ\lambda\circ\tilde{x}_1\in\Hom_{O_E}(A_1,(A_0^r)^\vee)_\dQ$ equals zero. Put $\lambda_1\coloneqq\tilde{x}_1^\vee\circ\lambda\circ\tilde{x}_1$. As $\tilde{x}_*\in\eta(Lx)$, we may replace $h$ by an element in $hL$ such that the restriction of $\eta\circ h^{-1}$ to $V_{x'}^\perp\otimes_\dQ\dA^\infty$, which we denote by $\eta_1$, is contained in the submodule $\Hom_{E\otimes_\dQ\dA^\infty}^{\lambda_0,\lambda_1}(\rH_1(A_{0s},\dA^\infty),\rH_1(A_{1s},\dA^\infty))$. Thus, we obtain an element
\[
(A_0,\lambda_0,\eta_0;A_1,\lambda_1,\eta_1)\in(X^{x'}_{hLh^{-1}\cap H^{x'}(\dA_F^\infty)}\otimes_EY)(S).
\]
By construction, it maps to $(A_0,\lambda_0,\eta_0;A,\lambda,\eta)\in X'_L(S)$ in \eqref{eq:moduli_cycle}.

For the reverse direction, take an element
\[
(A_0,\lambda_0,\eta_0;A_1,\lambda_1,\eta_1)\in(X^{x'}_{hLh^{-1}\cap H^{x'}(\dA_F^\infty)}\otimes_EY)(S).
\]
Put $A_2\coloneqq A_0^r$ and let $\lambda_2$ be the polarization such that we have an isomorphism
\[
\eta_2\colon V_{x'}\xrightarrow{\sim}\Hom_{O_E}^{\lambda_0,\lambda_2}(A_0,A_2)_\dQ
\]
of hermitian spaces over $E$. Put $A\coloneqq A_1\times A_2$, $\lambda\coloneqq\lambda_1\times\lambda_2$, and $\eta\coloneqq(\eta_1\oplus\eta_2\otimes_\dQ\dA^\infty)\circ h$. Then $(A_0,\lambda_0,\eta_0;A,\lambda,\eta)$ is the image of $(A_0,\lambda_0,\eta_0;A_1,\lambda_1,\eta_1)$ in $X'_L(S)$ in \eqref{eq:moduli_cycle}. Let $\tilde{x}$ be the isomorphism in $\Hom_{O_E}(A_0^r,A_2)_\dQ$ that corresponds to $\eta_2(x')$, which we regard as an element in $\Hom_{O_E}(A_0^r,A)_\dQ$. Then we obtain an element $(A_0,\lambda_0,\eta_0;A,\lambda,\eta;\tilde{x})\in Z'(x)_L(S)$ lying above $(A_0,\lambda_0,\eta_0;A,\lambda,\eta)$.

It is straightforward to check that the above two assignments are inverse to each other. The lemma follows.
\end{proof}

The following lemma will only be used in Section \ref{ss:archimedean}.

\begin{lem}\label{le:moduli_archimedean}
For every $u\in\tV_E^{(\infty)}$, there exists an isomorphism
\[
\{X_L\otimes_{E,\iota_u}\dC\}\simeq\{\pres{u}{X_L}\otimes_{\iota_u(E)}\dC\}
\]
of systems of complex schemes under which $Z(x)_L\otimes_{E,\iota_u}\dC$ coincides with $\pres{u}Z(x)_L\otimes_{\iota_u(E)}\dC$ for every $x\in V^r\otimes_{\dA_F}\dA_F^\infty$ with $T(x)\in\Herm_r^\circ(F)^+$. Here, $\pres{u}{X_L}$ and $\pres{u}Z(x)_L$ are defined similarly as $X_L$ and $Z(x)_L$ with $\biota$ replaced by $\iota_u$, hence are schemes and cycles over $\iota_u(E)$.
\end{lem}

\begin{proof}
We choose an isomorphism $\sigma\colon\dC\xrightarrow{\sim}\dC$ satisfying $\iota_u=\sigma\circ\biota$.

Choose an element $(A_0,\lambda_0,\eta_0)\in Y(\dC)$. Then by Lemma \ref{le:moduli_auxiliary}, $X_L\otimes_E\dC$ has the following moduli interpretation: For every locally Noetherian complex scheme $S$, $(X_L\otimes_E\dC)(S)$ is the set of equivalence classes of triples $(A,\lambda,\eta)$ as in Lemma \ref{le:moduli_auxiliary}. In particular, $(A,\lambda)$ is a unitary $O_E$-abelian scheme of signature type $n\Phi-\biota+\biota^\tc$ over $S$. Since $X_L\otimes_E\dC$ does not depend on the choice of $\Phi$, such moduli interpretation also holds if we replace $\Phi$ by any other CM type that contains $\biota$. In particular, we take $\Phi$ such that it contains both $\iota_u$ and $\biota$. Then we have such moduli interpretation for both $\Phi$ and $\sigma^{-1}\Phi$. Using both moduli interpretations, we obtain an isomorphism $\{X_L\otimes_{E,\iota_u}\dC\}\simeq\{\pres{u}{X_L}\otimes_{\iota_u(E)}\dC\}$. By Lemma \ref{le:moduli_cycle}, it follows easily that under such isomorphism, $Z(x)_L\otimes_{E,\iota_u}\dC$ coincides with $\pres{u}Z(x)_L\otimes_{\iota_u(E)}\dC$ for every $x\in V^r\otimes_{\dA_F}\dA_F^\infty$ with $T(x)\in\Herm_r^\circ(F)^+$. The lemma is proved.
\end{proof}

\begin{notation}\label{st:auxiliary}
In Sections \ref{ss:split}, \ref{ss:inert1}, and \ref{ss:inert2}, we will consider a place $u\in\tV_E^\fin\setminus\tV_E^\ram$. Let $p$ be the underlying rational prime of $u$. We will fix an isomorphism $\dC\xrightarrow\sim\ol\dQ_p$ under which $\biota$ induces the place $u$. In particular, we may identify $\Phi$ as a subset of $\Hom(E,\ol\dQ_p)$.

We further require that $\Phi$ in Notation \ref{co:auxiliary} is \emph{admissible} in the following sense: if $\Phi_v\subseteq\Phi$ denotes the subset inducing the place $v$ for every $v\in\tV_F^{(p)}$, then it satisfies
\begin{enumerate}
  \item when $v\in\tV_F^{(p)}\cap\tV_F^\spl$, $\Phi_v$ induces the same place of $E$ above $v$, which we denote by $v_c$ and by $v_e$ its conjugate;

  \item when $v\in\tV_F^{(p)}\cap\tV_F^\inert$, $\Phi_v$ is the pullback of a CM type of the maximal subfield of $E_v$ unramified over $\dQ_p$.
\end{enumerate}
To release the burden of notation, we denote by $K$ the subfield of $\ol\dQ_p$ generated by $E_u$ and the reflex field of $\Phi$, by $k$ its residue field, and by $\breve{K}$ the completion of the maximal unramified extension of $K$ in $\ol\dQ_p$ with the residue field $\ol\dF_p$. It is clear that admissible CM type always exists, and that when $\tV_F^{(p)}\cap\tV_F^\ram=\emptyset$, the field $K$ is unramified over $E_u$.

We also choose a (sufficiently small) open compact subgroup $L_0$ of $\rT_0(\dA^\infty)$ such that $L_{0,p}$ is maximal compact. We denote by $\cY$ the integral model of $Y$ over $O_K$ such that for every $S\in\Sch'_{/O_K}$, $\cY(S)$ is the set of equivalence classes of triples $(A_0,\lambda_0,\eta_0^p)$ where
\begin{itemize}
  \item $(A_0,\lambda_0)$ is a unitary $O_E$-abelian scheme over $S$ of signature type $\Phi$ such that $\lambda_0$ is a $p$-principal polarization;

  \item $\eta_0^p$ is an $L_0^p$-level structure (see \cite{LTXZZ}*{Definition~4.1.2} for more details).
\end{itemize}
By \cite{How12}*{Proposition~3.1.2}, $\cY$ is finite and \'{e}tale over $O_K$.
\end{notation}

\section{Height pairing and geometric side}
\label{ss:height}

In this section, we introduce the notion of a height pairing after Beilinson and initiate the study of the geometric side of our desired height formula. We continue the discussion from Section \ref{ss:special}. From this moment, we will further assume $F\neq\dQ$, which implies that $X_L$ is projective.

We apply Beilinson's construction of the height pairing in \cite{Bei87}*{Section~4} to obtain a map
\begin{align*}
\langle\;,\;\rangle_{X_L,E}^\ell\colon\CH^r(X_L)^{\langle\ell\rangle}_\dC\times\CH^r(X_L)^{\langle\ell\rangle}_\dC\to\dC\otimes_\dQ\dQ_\ell
\end{align*}
(see Notation \ref{st:c}(C3) for the notation) that is complex linear in the first variable, and conjugate symmetric. Here, $\ell$ is a rational prime such that $X_{L,u}$ has smooth projective reduction for every $u\in\tV_E^{(\ell)}$. For a pair $(c_1,c_2)$ of elements in $\rZ^r(X_L)^{\langle\ell\rangle}_\dC\times\rZ^r(X_L)^{\langle\ell\rangle}_\dC$ with disjoint supports, we have
\begin{align}\label{eq:height1}
\langle c_1,c_2\rangle_{X_L,E}^\ell=\sum_{u\in\tV_E^{(\infty)}}2\langle c_1,c_2\rangle_{X_{L,u},E_u}+
\sum_{u\in\tV_E^\fin}\log q_u\cdot \langle c_1,c_2\rangle_{X_{L,u},E_u}^\ell,
\end{align}
in which
\begin{itemize}
  \item $q_u$ is the residue cardinality of $E_u$ for $u\in\tV_E^\fin$;

  \item $\langle c_1,c_2\rangle_{X_{L,u},E_u}^\ell\in\dC\otimes_\dQ\dQ_\ell$ is the non-archimedean local index \eqref{eq:index} recalled in Appendix \ref{ss:beilinson} for $u\in\tV_E^\fin$ (see Remark \ref{re:index_smooth} when $u$ is above $\ell$), which equals zero for all but finitely many $u$;

  \item $\langle c_1,c_2\rangle_{X_{L,u},E_u}\in\dC$ is the archimedean local index for $u\in\tV_E^{(\infty)}$, which will be recalled when we compute it in Section \ref{ss:archimedean}.
\end{itemize}

\begin{definition}\label{de:good}
We say that a rational prime $\ell$ is \emph{$\tR$-good} if $\ell$ is unramified in $E$ and satisfies $\tV_F^{(\ell)}\subseteq\tV_F^\fin\setminus(\tR\cup\tS)$.
\end{definition}

\begin{definition}\label{de:tempered_generic}
For every open compact subgroup $L_\tR$ of $H(F_\tR)$ and every subfield $\dL$ of $\dC$, we define
\begin{enumerate}
  \item $(\dS^\tR_\dL)^0_{L_\tR}$ to be the ideal of $\dS^\tR_\dL$ (Notation \ref{st:h}(H8)) of elements that annihilate
      \[
      \bigoplus_{i\neq 2r-1}\rH^i_{\dr}(X_{L_\tR L^\tR}/E)\otimes_\dQ\dL,
      \]

  \item for every rational prime $\ell$, $(\dS^\tR_\dL)^{\langle\ell\rangle}_{L_\tR}$ to be the ideal of $\dS^\tR_\dL$ of elements that annihilate
      \[
      \bigoplus_{u\in\tV_E^\fin\setminus\tV_E^{(\ell)}}\rH^{2r}(X_{L_\tR L^\tR,u},\dQ_\ell(r))\otimes_\dQ\dL.
      \]
\end{enumerate}
Here, $L^\tR$ is defined in Notation \ref{st:h}(H8).
\end{definition}

\begin{definition}\label{de:kernel_geometric}
Consider a nonempty subset $\tR'\subseteq\tR$, an $\tR$-good rational prime $\ell$, and an open compact subgroup $L$ of $H(\dA_F^\infty)$ of the form $L_\tR L^\tR$ where $L^\tR$ is defined in Notation \ref{st:h}(H8). An \emph{$(\tR,\tR',\ell,L)$-admissible sextuple} is a sextuple $(\phi^\infty_1,\phi^\infty_2,\rs_1,\rs_2,g_1,g_2)$ in which
\begin{itemize}
  \item for $i=1,2$, $\phi^\infty_i=\otimes_v\phi^\infty_{iv}\in\sS(V^r\otimes_{\dA_F}\dA_F^\infty)^L$ in which $\phi^\infty_{iv}=\CF_{(\Lambda^\tR_v)^r}$ for $v\in\tV_F^\fin\setminus\tR$, satisfying that $\supp(\phi^\infty_{1v}\otimes(\phi^\infty_{2v})^\tc)\subseteq(V^{2r}_v)_\reg$ for $v\in\tR'$;

  \item for $i=1,2$, $\rs_i$ is a product of two elements in $(\dS^\tR_{\dQ^\ac})^{\langle\ell\rangle}_{L_\tR}$;

  \item for $i=1,2$, $g_i$ is an element in $G_r(\dA_F^{\tR'})$.
\end{itemize}

For an $(\tR,\tR',\ell,L)$-admissible sextuple $(\phi^\infty_1,\phi^\infty_2,\rs_1,\rs_2,g_1,g_2)$ and every pair $(T_1,T_2)$ of elements in $\Herm_r^\circ(F)^+$, we define
\begin{enumerate}
  \item the global index $I_{T_1,T_2}(\phi^\infty_1,\phi^\infty_2,\rs_1,\rs_2,g_1,g_2)^\ell_L$ to be
     \begin{align*}
     \langle \omega_{r,\infty}(g_{1\infty})\phi^0_\infty(T_1)\cdot\rs_1^*Z_{T_1}(\omega_r^\infty(g_1^\infty)\phi^\infty_1)_L,
     \omega_{r,\infty}(g_{2\infty})\phi^0_\infty(T_2)\cdot\rs_2^*Z_{T_2}(\omega_r^\infty(g_2^\infty)\phi^\infty_2)_L
     \rangle_{X_L,E}^\ell
     \end{align*}
     as an element in $\dC\otimes_\dQ\dQ_\ell$, where we note that for $i=1,2$, $\rs_i^*Z_{T_i}(\omega_r^\infty(g_i^\infty)\phi^\infty_i)_L$ belongs to $\CH^r(X_L)^{\langle\ell\rangle}_\dC$ by Definition \ref{de:tempered_generic}(2);

  \item for every $u\in\tV_E^\fin$, the local index $I_{T_1,T_2}(\phi^\infty_1,\phi^\infty_2,\rs_1,\rs_2,g_1,g_2)^\ell_{L,u}$ to be
     \begin{align*}
     \langle \omega_{r,\infty}(g_{1\infty})\phi^0_\infty(T_1)\cdot\rs_1^*Z_{T_1}(\omega_r^\infty(g_1^\infty)\phi^\infty_1)_L,
     \omega_{r,\infty}(g_{2\infty})\phi^0_\infty(T_2)\cdot\rs_2^*Z_{T_2}(\omega_r^\infty(g_2^\infty)\phi^\infty_2)_L
     \rangle_{X_{L,u},E_u}^\ell
     \end{align*}
     as an element in $\dC\otimes_\dQ\dQ_\ell$, in view of Remark \ref{re:special_cycle} and Lemma \ref{le:disjoint}(2) below;

  \item for every $u\in\tV_E^{(\infty)}$, the local index $I_{T_1,T_2}(\phi^\infty_1,\phi^\infty_2,\rs_1,\rs_2,g_1,g_2)_{L,u}$ to be
     \begin{align*}
     \langle \omega_{r,\infty}(g_{1\infty})\phi^0_\infty(T_1)\cdot\rs_1^*Z_{T_1}(\omega_r^\infty(g_1^\infty)\phi^\infty_1)_L,
     \omega_{r,\infty}(g_{2\infty})\phi^0_\infty(T_2)\cdot\rs_2^*Z_{T_2}(\omega_r^\infty(g_2^\infty)\phi^\infty_2)_L
     \rangle_{X_{L,u},E_u}
     \end{align*}
     as an element in $\dC$, in view of Remark \ref{re:special_cycle} and Lemma \ref{le:disjoint}(2) below.
\end{enumerate}
\end{definition}

\begin{lem}\label{le:disjoint}
Let $\tR$, $\tR'$, $\ell$, and $L$ be as in Definition \ref{de:kernel_geometric}. Let $(T_1,T_2)$ be a pair of elements in $\Herm_r^\circ(F)^+$.
\begin{enumerate}
  \item For $x_1,x_2\in V\otimes_{\dA_F}\dA_F^\infty$ satisfying $T(x_1)=T_1$, $T(x_2)=T_2$, and $(L_vx_{1v},L_vx_{2v})\subseteq(V^{2r}_v)_\reg$ for some $v\in\tR'$, the algebraic cycles $Z(x_1)_L$ and $Z(x_2)_L$ in $\rZ^r(X_L)_\dC$ have disjoint supports.

  \item For every $(\tR,\tR',\ell,L)$-admissible sextuple $(\phi^\infty_1,\phi^\infty_2,\rs_1,\rs_2,g_1,g_2)$, the algebraic cycles $\rs_1^*Z_{T_1}(\omega_r^\infty(g_1^\infty)\phi^\infty_1)_L$ and $\rs_2^*Z_{T_2}(\omega_r^\infty(g_2^\infty)\phi^\infty_2)_L$ in $\rZ^r(X_L)_\dC$ have disjoint supports.
\end{enumerate}
\end{lem}

\begin{proof}
It is clear that (2) follows from (1).

For (1), it suffices to check that they are disjoint under complex uniformization \eqref{eq:complex_uniformization}. By definition, for $i=1,2$, the support of $Z(x_i)_L$ consists of points $(z_i,h'_ih_i)$ in the double coset \eqref{eq:complex_uniformization}, where $h_ix_i=x'_i$ with $x'_i\in\pres{\bu}{V}^r$; $z_i$ is perpendicular to $V_{x'_i}$; and $h'_i$ acts trivially on $V_{x'_i}$. Suppose that the supports of $Z(x_1)_L$ and $Z(x_2)_L$ are not disjoint, then we may find $\gamma\in\pres{\bu}{H}(F)$ such that $z_1=\gamma z_2$ and $h'_1h_1L=\gamma h'_2h_2L$. In particular, $V_{x'_1}\cap\gamma V_{x'_2}\neq\{0\}$, which implies that the subspace of $\pres{\bu}{V}^{2r}$ generated by $(x'_1,\gamma x'_2)$ is a proper subspace. Thus, $(h_1x_1,\gamma h_2x_2)\not\in(V^{2r}_v)_\reg$ for every $v\in\tR'$. On the other hand, we have $(h_1x_1,\gamma h_2x_2)=(h'_1h_1x_1,\gamma h'_2h_2x_2)$, which implies that $(L_vx_{1v},L_vx_{2v})$ is not contained in $(V^{2r}_v)_\reg$, which is a contradiction. Thus, (1) follows.
\end{proof}

The following definition will be used in the future.

\begin{definition}\label{de:basic}
Let $p$ be a rational prime. We say that an element $\phi^\infty\in\sS(V^m\otimes_{\dA_F}\dA_F^\infty)$ for some integer $m\geq 1$ is \emph{$p$-basic} if it is of the form $\phi^\infty=\otimes_v\phi^\infty_v$ in which $\phi^\infty_v=\CF_{(\Lambda^\tR_v)^m}$ for every $v\in\tV_F^{(p)}\setminus(\tR\cup\tV_F^\spl)$.
\end{definition}

Recall from Notation \ref{co:galois}(1) that $\Pi$ is the automorphic base change of $\pi$.

\begin{hypothesis}\label{hy:galois}
Let $\ell$ be a rational prime with an arbitrarily given isomorphism $\ol\dQ_\ell\simeq\dC$. For every irreducible admissible representation $\tilde\pi^\infty$ of $H(\dA_F^\infty)$ such that $\Pi_v$ is the standard base change of $\tilde\pi^\infty_v$ for all but finitely many $v\in\tV_F^\fin$ for which $\tilde\pi^\infty_v$ is unramified, if we are in the situation (b) of Lemma \ref{le:arthur}, then the semisimplification of the representation
\[
\rho[\pi^\infty]\coloneqq\Hom_{H(\dA_F^\infty)}\(\tilde\pi^\infty,\varinjlim_L\rH^{2r-1}(X_L\otimes_{E}\dQ^\ac,\ol\dQ_\ell)\)
\]
of $\Gal(\dQ^\ac/E)$ is isomorphic to $\rho^\tc_{\Pi_{j(\tilde\pi^\infty)}}$, where $\rho_{\Pi_j}$ is introduced in Notation \ref{co:galois}(4).
\end{hypothesis}

\begin{remark}\label{re:galois}
Concerning Hypothesis \ref{hy:galois}, we have
\begin{enumerate}
  \item When $n=2$, it has been confirmed in \cite{Liu19}*{Theorem~D.6}.

  \item When $\Pi$ is cuspidal (that is, $s=1$ in Notation \ref{co:galois}(3)), it will be confirmed in \cite{KSZ} (under the help of \cites{Mok15,KMSW}).

  \item In general, it will follow from \cite{KSZ} as long as the full endoscopic classification for unitary groups is obtained.
\end{enumerate}
\end{remark}

\begin{definition}\label{de:hecke}
Let $(\pi,\cV_\pi)$ be as in Assumption \ref{st:representation}. We define a character
\[
\chi^\tR_\pi\colon\dT^\tR_{\dQ^\ac}\to\dQ^\ac,
\]
as follow. Let $\cH^\tR_{W_r}$ be the restricted tensor product of commutative complex algebras $\cH^\pm_{W_{r,v}}$ if $\epsilon_v=\pm1$ over $v\in\tV_F^\fin\setminus\tR$, where $\cH^\pm_{W_{r,v}}$ is defined in \cite{Liu20}*{Definition~2.5} for $v\in\tV_F^\inert$ and is simply the spherical Hecke algebra for $v\in\tV_F^\spl$. Using the construction in \cite{Liu20}*{Definition~2.8}, we have a canonical surjective homomorphism $\theta^\tR\colon\cH^\tR_{W_r}\to\dT^\tR_\dC$ of complex commutative algebras. Since $\pi_v$ is unramified (resp.\ almost unramified) when $\epsilon_v=1$ (resp.\ $\epsilon_v=-1$) by Assumption \ref{st:representation}(2), the algebra $\cH^\tR_{W_r}$ acts on $\cV_\pi^{[r]\tR}$ by a character $\chi_{\pi,W_r}^\tR\colon\cH^\tR_{W_r}\to\dC$ which factors through $\theta^\tR$ by \cite{Liu20}*{Definition~5.3 \& Theorem~1.1(1)}. Since $\pi$ is cohomological by Assumption \ref{st:representation}(1) hence has algebraic Satake parameters, there exists a unique character $\chi^\tR_\pi\colon\dT^\tR_{\dQ^\ac}\to\dQ^\ac$ such that $\chi_{\pi,W_r}^\tR=(\chi^\tR_\pi\otimes_{\dQ^\ac}\dC)\circ\theta^\tR$.

We put $\fm_\pi^\tR\coloneqq\Ker\chi^\tR_\pi$, which is a maximal ideal of $\dT^\tR_{\dQ^\ac}$.
\end{definition}

\begin{proposition}\label{pr:tempered_generic}
Let $(\pi,\cV_\pi)$ be as in Assumption \ref{st:representation}. For every open compact subgroup $L_\tR$ of $H(F_\tR)$, we have
\begin{enumerate}
  \item $(\dS^\tR_{\dQ^\ac})^0_{L_\tR}\setminus \fm_\pi^\tR$ is nonempty;

  \item under Hypothesis \ref{hy:galois}, $(\dS^\tR_{\dQ^\ac})^{\langle\ell\rangle}_{L_\tR}\setminus \fm_\pi^\tR$ is nonempty.
\end{enumerate}
\end{proposition}

\begin{proof}
For (1), by Matsushima's formula, we know that the localization of the $\dS^\tR_\dC$-module $\rH^i_{\dr}(X_{L_\tR L^\tR}/E)\otimes_\dQ\dC$ at $\fm_\pi^\tR$ is isomorphic to the direct sum of $\rH^i(\tilde\pi_\infty)\otimes\tilde\pi^\infty$ for all cuspidal automorphic representations  $\tilde\pi$ of $\pres{\bu}{H}(\dA_F)$ such that the standard base change of $\tilde\pi_v$ is isomorphic to $\Pi_v$ for all but finitely many $v\in\tV_F^\spl$, where $\rH^i(\tilde\pi_\infty)$ denotes the $(\fg,K)$-cohomology of $\tilde\pi_\infty$. By \cite{Ram}*{Theorem~A}, we know that $\Pi$ must be the automorphic base change of $\tilde\pi$. By \cite{KMSW}*{Theorem~1.7.1}, we know that $\tilde\pi_\infty$ is tempered, hence $\rH^i(\tilde\pi_\infty)$ vanishes for $i\neq 2r-1$. Therefore, (1) follows as $\rH^i_{\dr}(X_{L_\tR L^\tR}/E)$ is of finite dimension.

For (2), note that for all but finitely many $u\in\tV_E^\fin\setminus\tV_E^{(\ell)}$, the natural map
\[
\rH^{2r}(X_{L_\tR L^\tR,u},\dQ_\ell(r))\to\rH^{2r}(X_{L_\tR L^\tR,u}\otimes_{E_u}\ol{E_u},\dQ_\ell(r))
\]
is injective by the Hochschild--Serre spectral sequence and the Weil conjecture. As an $\dS^\tR_{\dQ^\ac}$-module, we have
\[
\rH^{2r}(X_{L_\tR L^\tR,u}\otimes_{E_u}\ol{E_u},\dQ_\ell(r))\otimes_\dQ(E\otimes_\dQ\dQ^\ac)
\simeq\rH^{2r}_{\dr}(X_{L_\tR L^\tR}/E)\otimes_\dQ(\dQ_\ell\otimes_\dQ\dQ^\ac).
\]
By (1), we know that there exist elements in $\dS^\tR_{\dQ^\ac}\setminus\fm_\pi^\tR$ that annihilate $\rH^{2r}(X_{L_\tR L^\tR,u},\dQ_\ell(r))$ for all but finitely many $u\in\tV_E^\fin\setminus\tV_E^{(\ell)}$. Thus, it remains to show that for every given $u\in\tV_E^\fin\setminus\tV_E^{(\ell)}$ and every embedding $\dQ^\ac\hookrightarrow\ol\dQ_\ell$, the localization of $\rH^{2r}(X_{L_\tR L^\tR,u},\ol\dQ_\ell(r))$ at $\fm_\pi^\tR$ vanishes. By (1) and the Hochschild--Serre spectral sequence, it suffices to show that
\[
\rH^1(E_u,\rH^{2r-1}(X_{L_\tR L^\tR}\otimes_{E}\dQ^\ac,\ol\dQ_\ell(r))_{\fm_\pi^\tR})=0.
\]
By Hypothesis \ref{hy:galois}, it suffices to show that $\rH^1(E_u,\rho^\tc_{\Pi_j}(r))=0$ for every $j$. As shown in the proof of \cite{Car12}*{Theorem~7.4}, the associated Weil--Deligne representation of $\rho_{\Pi_j}(r)$ at $u$ is pure (of weight not zero), which implies $\rH^1(E_u,\rho^\tc_{\Pi_j}(r))=0$ by \cite{Nek07}*{Proposition~4.2.2(1)}.

The proposition is proved.
\end{proof}

Till the end of this section, let $(\pi,\cV_\pi)$ be as in Assumption \ref{st:representation}, and assume Hypothesis \ref{hy:modularity} on the modularity of generating functions of codimension $r$.

\begin{proposition}\label{pr:arithmetic_theta}
In the situation of Definition \ref{de:arithmetic_theta} (and suppose that $F\neq\dQ$), suppose that $L$ has the form $L_\tR L^\tR$ where $L^\tR$ is defined in Notation \ref{st:h}(H8). For every elements $\varphi\in\cV_\pi^{[r]\tR}$ and $\phi^\infty\in\sS(V^r\otimes_{\dA_F}\dA_F^\infty)^L$, we have
\begin{enumerate}
  \item $\rs^*\Theta_{\phi^\infty}(\varphi)_L=\chi^\tR_\pi(\rs)^\tc\cdot\Theta_{\phi^\infty}(\varphi)_L$ for every $\rs\in\dS^\tR_{\dQ^\ac}$;

  \item $\Theta_{\phi^\infty}(\varphi)_L\in\CH^r(X_L)_\dC^0$;

  \item under Hypothesis \ref{hy:galois}, $\Theta_{\phi^\infty}(\varphi)_L\in\CH^r(X_L)_\dC^{\langle\ell\rangle}$ for every $\tR$-good rational prime $\ell$.
\end{enumerate}
\end{proposition}

\begin{proof}
For (1), by Lemma \ref{re:translate}, we have
\[
\rs^*\Theta_{\phi^\infty}(\varphi)_L=\Theta_{\rs\phi^\infty}(\varphi)_L
=\int_{G_r(F)\backslash G_r(\dA_F)}\varphi^\tc(g)Z_{\rs\phi^\infty}(g)_L\rd g,
\]
which, by \cite{Liu11}*{Proposition~A.5} for split places (see also \cite{Ral82}*{Page~511}), equals
\[
\int_{G_r(F)\backslash G_r(\dA_F)}(\pi^\vee(\rs)\varphi^\tc)(g)Z_{\phi^\infty}(g)_L\rd g,
\]
which equals
\[
\int_{G_r(F)\backslash G_r(\dA_F)}(\chi^\tR_\pi(\rs)^\tc\cdot\varphi^\tc)(g)Z_{\phi^\infty}(g)_L\rd g
=\chi^\tR_\pi(\rs)^\tc\cdot\Theta_{\phi^\infty}(\varphi)_L.
\]

Part (2) is a consequence of (1) and Proposition \ref{pr:tempered_generic}(1).

Part (3) is a consequence of (1) and Proposition \ref{pr:tempered_generic}(2).
\end{proof}

\if false

\begin{lem}\label{le:natural}
Let $L$ be an open compact subgroup of $H(\dA_F^\infty)$ of the form $L_\tR L^\tR$ where $L^\tR$ is defined in Notation \ref{st:h}(H8). Let $\ell$ be an $\tR$-good rational prime. Consider $\varphi_1,\varphi_2\in\cV_\pi^{[r]\tR}$, and $\phi^\infty_1,\phi^\infty_2\in\sS(V^r\otimes_{\dA_F}\dA_F^\infty)$ in which $\phi^\infty_{iv}=\CF_{(\Lambda^\tR_v)^r}$ for $v\in\tV_F^\fin\setminus\tR$ for $i=1,2$. Then the value
\begin{align}\label{eq:natural}
\langle\rs^*\Theta_{\phi^\infty_1}(\varphi_1)_L,\rs^*\Theta_{\phi^\infty_2}(\varphi_2)_L\rangle_{X_L,E}^\ell
\end{align}
does not depend on the choice of $\rs\in(\dS^\tR_{\dQ^\ac})^{\langle\ell\rangle}_{L_\tR}$ satisfying $\chi^\tR_\pi(\rs)=1$.
\end{lem}

\begin{proof}
We first note that \eqref{eq:natural} is well-defined as long as $\rs\in(\dS^\tR_{\dQ^\ac})^{\langle\ell\rangle}_{L_\tR}$. Let $\rs,\rs'$ be two such elements as in the statement of the lemma. By the identity $\chi^\tR_\pi(\rs')=1$ and \cite{Liu11}*{Proposition~A.5} for split places (see also \cite{Ral82}*{Page~511}), we have
\[
\langle\rs^*\Theta_{\phi^\infty_1}(\varphi_1)_L,\rs^*\Theta_{\phi^\infty_2}(\varphi_2)_L\rangle_{X_L,E}^\ell
=\langle\rs^*\Theta_{\rs'\phi^\infty_1}(\varphi_1)_L,
\rs^*\Theta_{\rs'\phi^\infty_2}(\varphi_2)_L\rangle_{X_L,E}^\ell,
\]
which, by Lemma \ref{re:translate}, equals
\[
\langle\rs^*\rs'^*\Theta_{\phi^\infty_1}(\varphi_1)_L,
\rs^*\rs'^*\Theta_{\phi^\infty_2}(\varphi_2)_L\rangle_{X_L,E}^\ell.
\]
By the same reason, we have
\[
\langle\rs'^*\Theta_{\phi^\infty_1}(\varphi_1)_L,
\rs'^*\Theta_{\phi^\infty_2}(\varphi_2)_L\rangle_{X_L,E}^\ell
=\langle\rs'^*\rs^*\Theta_{\phi^\infty_1}(\varphi_1)_L,
\rs'^*\rs^*\Theta_{\phi^\infty_2}(\varphi_2)_L\rangle_{X_L,E}^\ell.
\]
Since $\rs^*$ and $\rs'^*$ commute, \eqref{eq:natural} is independent of $\rs$. The lemma follows.
\end{proof}

\fi

We now define the \emph{normalized height pairing} between the cycles $\Theta_{\phi^\infty}(\varphi)$ in Definition \ref{de:arithmetic_theta}, under Hypothesis \ref{hy:galois}.

\begin{definition}\label{de:natural}
Under Hypothesis \ref{hy:galois}, for every elements $\varphi_1,\varphi_2\in\cV_\pi^{[r]}$ and $\phi^\infty_1,\phi^\infty_2\in\sS(V^r\otimes_{\dA_F}\dA_F^\infty)$, we define the \emph{normalized height pairing}
\[
\langle\Theta_{\phi^\infty_1}(\varphi_1),\Theta_{\phi^\infty_2}(\varphi_2)\rangle_{X,E}^\natural
\in\dC\otimes_\dQ\dQ_\ell
\]
to be the unique element\footnote{Readers may notice that we have dropped $\ell$ in the notation $\langle\Theta_{\phi^\infty_1}(\varphi_1),\Theta_{\phi^\infty_2}(\varphi_2)\rangle_{X,E}^\natural$. This is because for those normalized height pairings we are able to compute in this article, the value will turn out to be in $\dC$ and is independent of the choice of $\ell$.} such that for every $L=L_\tR L^\tR$ as in Proposition \ref{pr:arithmetic_theta} (with $\tR$ possibly enlarged) satisfying $\varphi_1,\varphi_2\in\cV_\pi^{[r]\tR}$, $\phi^\infty_1,\phi^\infty_2\in\sS(V^r\otimes_{\dA_F}\dA_F^\infty)^L$, and that $\ell$ is $\tR$-good, we have
\[
\langle\Theta_{\phi^\infty_1}(\varphi_1),\Theta_{\phi^\infty_2}(\varphi_2)\rangle_{X,E}^\natural=
\vol^\natural(L)\cdot
\langle\Theta_{\phi^\infty_1}(\varphi_1)_L,\Theta_{\phi^\infty_2}(\varphi_2)_L\rangle_{X_L,E}^\ell,
\]
where $\vol^\natural(L)$ is introduced in Definition \ref{de:measure},\footnote{In fact, it is a good exercise to show that the total degree of the Hodge line bundle on $X_L$ is equal to $2\vol^\natural(L)^{-1}$.} and $\langle\Theta_{\phi^\infty_1}(\varphi_1)_L,\Theta_{\phi^\infty_2}(\varphi_2)_L\rangle_{X_L,E}^\ell$ is well-defined by Proposition \ref{pr:arithmetic_theta}(3). Note that by the projection formula, the right-hand side of the above formula is independent of $L$.
\end{definition}

\section{Local indices at split places}
\label{ss:split}

In this section, we compute local indices at almost all places in $\tV_E^\spl$. Our goal is to prove the following proposition.

\begin{proposition}\label{pr:index_split}
Let $\tR$, $\tR'$, $\ell$, and $L$ be as in Definition \ref{de:kernel_geometric} such that the cardinality of $\tR'$ is at least $2$. Let $(\pi,\cV_\pi)$ be as in Assumption \ref{st:representation}, for which we assume Hypothesis \ref{hy:galois}. For every $u\in\tV_E^\spl$ such that
\begin{enumerate}[label=(\alph*)]
  \item the representation $\pi_{\ul{u}}$ is a (tempered) principal series;

  \item $\tV_F^{(p)}\cap\tR\subseteq\tV_F^\spl$ where $p$ is the underlying rational prime of $u$,
\end{enumerate}
there exist elements $\rs_1^u,\rs_2^u\in\dS_{\dQ^\ac}^\tR\setminus\fm_\pi^\tR$ such that
\[
I_{T_1,T_2}(\phi^\infty_1,\phi^\infty_2,\rs_1^u\rs_1,\rs_2^u\rs_2,g_1,g_2)^\ell_{L,u}=0
\]
for every $(\tR,\tR',\ell,L)$-admissible sextuple $(\phi^\infty_1,\phi^\infty_2,\rs_1,\rs_2,g_1,g_2)$ and every pair $(T_1,T_2)$ in $\Herm_r^\circ(F)^+$. Moreover, we may take $\rs_1^u=\rs_2^u=1$ if $\ul{u}\not\in\tR$.
\end{proposition}

Since $\ul{u}$ splits in $E$, we may fix an isomorphism $H\otimes_{\dA_F}F_{\ul{u}}\simeq\GL_{n,F_{\ul{u}}}$ such that $L_{\ul{u}}$ is contained in $\GL_n(O_{F_{\ul{u}}})$, and moreover equal if $\ul{u}\not\in\tR$. For every integer $m\geq 0$, denote by $L_{\ul{u},m}\subseteq\GL_n(O_{F_{\ul{u}}})$ the principal congruence subgroup of level $m$.

From now to the end of this section, we assume $\tV_F^{(p)}\cap\tR\subseteq\tV_F^\spl$. We invoke Notation \ref{co:auxiliary} together with Notation \ref{st:auxiliary}, which is possible since $\tV_F^{(p)}\cap\tV_F^\ram=\emptyset$. To ease notation, we put $X_m\coloneqq X'_{L_{\ul{u},m}L^{\ul{u}}}\otimes_{E'}K$ for $m\geq 0$. The isomorphism $\dC\xrightarrow\sim\ol\dQ_p$ in Notation \ref{st:auxiliary} identifies $\Hom(E,\dC)$ with $\Hom(E,\ol\dQ_p)$. For every $v\in\tV_F^{(p)}\cap\tV_F^\spl\setminus\{\ul{u}\}$, let $\{v_c,v_e\}$ be the two places of $E$ above $v$ from Notation \ref{st:auxiliary}; and identify $H\otimes_{\dA_F}F_v$ with $\GL(V\otimes_{\dA_F}E_{v_e})$.

Let $S$ be a locally Noetherian scheme over $O_K$ and $(A,\lambda)$ a unitary $O_E$-abelian scheme of signature type $n\Phi-\iota_w+\iota_w^\tc$ over $S$. Then the $p$-divisible group $A[p^\infty]$ admits a decomposition $A[p^\infty]=\prod_{v\in\tV_F^{(p)}}A[v^\infty]$.

For every integer $m\geq 0$, we define a moduli functor $\cX_m$ over $O_K$ as follows: For every locally Noetherian scheme $S$ over $O_K$, $\cX_m(S)$ is the set of equivalence classes of tuples $(A_0,\lambda_0,\eta_0^p;A,\lambda,\eta^p,\{\eta_v\}_{v\in\tV_F^{(p)}\cap\tV_F^\spl\setminus\{\ul{u}\}},\eta_{u,m})$ where
\begin{itemize}
  \item $(A_0,\lambda_0,\eta_0^p)$ is an element in $\cY(S)$;

  \item $(A,\lambda)$ is a unitary $O_E$-abelian scheme of signature type $n\Phi-\iota_w+\iota_w^\tc$ over $S$, such that
      \begin{itemize}
        \item for every $v\in\tV_F^{(p)}$, $\lambda[v^\infty]$ is an isogeny (rather than a quasi-isogeny) whose kernel has order $q_v^{1-\epsilon_v}$;

        \item $\Lie(A[u^{\tc,\infty}])$ is of rank $1$ on which the action of $O_E$ is given by the embedding $\iota_w^\tc$;\footnote{Since $\Phi$ is admissible (Notation \ref{st:auxiliary}), the Eisenstein condition at $v\neq\ul{u}$ is implied by the Kottwitz condition, and at $\ul{u}$ is implied by the Kottwitz condition and that $\Lie(A[u^{\tc,\infty}])$ is of rank $1$ on which the action of $O_E$ is given by the embedding $\iota_w^\tc$.}
      \end{itemize}

  \item $\eta^p$ is an $L^p$-level structure, analogous to the one in Lemma \ref{le:moduli_auxiliary};

  \item for every $v\in\tV_F^{(p)}\cap\tV_F^\spl\setminus\{\ul{u}\}$, $\eta_v$ is an $L_v$-level structure, that is, an $L_v$-orbit of $E_{v_e}$-linear isomorphisms
      \[
      \eta_v\colon V\otimes_{\dA_F}E_{v_e}\xrightarrow\sim \Hom_{O_{E_{v_e}}}(A_0[v_e^\infty],A[v_e^\infty])\otimes_{O_{E_{v_e}}}E_{v_e}
      \]
      of $E_{v_e}$-sheaves over $S$;

  \item $\eta_{u,m}\colon(\fp_{\ul{u}}^{-m}/O_{F_{\ul{u}}})^n\to
      \ul\Hom_{O_{F_{\ul{u}}}}(A_0[u^{\tc,\infty}][\fp_{\ul{u}}^m],A[u^{\tc,\infty}][\fp_{\ul{u}}^m])$ is a Drinfeld level-$m$ structure (see \cite{RSZ}*{Section~4.3} for more details).
\end{itemize}
By \cite{RSZ}*{Theorem~4.5}, for every $m\geq 0$, $\cX_m$ is a regular scheme, flat (smooth, if $m=0$) and projective over $O_K$, and admits a canonical isomorphism $\cX_m\otimes_{O_K}K\simeq X_m$ of schemes over $K$.\footnote{Here, we have to use the fact that $K$ is unramified over $E_u$ to conclude that $\cX_m$ is regular when $m>0$.} Note that for every integer $m\geq0$, $\dS^{\tR\cup\tV_F^{(p)}}$ naturally gives a ring of \'{e}tale correspondences of $\cX_m$.

We first prove the follow lemma which addresses the easy part of Proposition \ref{pr:index_split} as a warm-up.

\begin{lem}\label{le:index_split}
Let the situation be as in Proposition \ref{pr:index_split}. Suppose that $\ul{u}\not\in\tR$. Then we have
\[
I_{T_1,T_2}(\phi^\infty_1,\phi^\infty_2,\rs_1,\rs_2,g_1,g_2)^\ell_{L,u}=0
\]
for every $(\tR,\tR',\ell,L)$-admissible sextuple $(\phi^\infty_1,\phi^\infty_2,\rs_1,\rs_2,g_1,g_2)$ and every pair $(T_1,T_2)$ in $\Herm_r^\circ(F)^+$.
\end{lem}

\begin{proof}
It suffices to show that for every $x_1,x_2\in V^r\otimes_{\dA_F}\dA_F^\infty$ satisfying $T(x_1),T(x_2)\in\Herm_r^\circ(F)^+$ and $(L_vx_{1v},L_vx_{2v})\subseteq(V^{2r}_v)_\reg$ for some $v\in\tR'$, we have
\begin{align*}
\langle Z(x_1)_L,Z(x_2)_L\rangle^\ell_{X_{L,u},E_u}=0.
\end{align*}
Since $\ul{u}\not\in\tR$, by Lemma \ref{le:base_change} and Lemma \ref{le:moduli_cycle}, it suffices to show that
\begin{align}\label{eq:index_split}
\langle Z(x_1)'_L,Z(x_2)'_L\rangle^\ell_{X_0,K}=0.
\end{align}

We use the integral model $\cX_0$ just constructed above, which is smooth and projective over $O_K$ of relative dimension $n-1$. For $i=1,2$, let $\cZ(x_i)'_L$ be the Zariski closure of $Z(x_i)'_L$ in $\cX_0$. We claim that $\cZ(x_1)'_L$ and $\cZ(x_2)'_L$ have empty intersection. By Proposition \ref{pr:index_smooth}, we obtain \eqref{eq:index_split}.

For the claim, we assume the converse. Then we can find a point
\[
(A_0,\lambda_0,\eta_0^p;A,\lambda,\eta^p,\{\eta_v\}_{v\in\tV_F^{(p)}\cap\tV_F^\spl\setminus\{\ul{u}\}})\in\cX_0(\ol\dF_p)
\]
that is in the supports of both $\cZ(x_1)'_L$ and $\cZ(x_2)'_L$. In particular, for $i=1,2$, we can find an element $\tilde{x}_i\in\Hom_{O_E}(A_0^r,A)_\dQ$ satisfying $\tilde{x}_{i,*}\in\eta^p(L^p x_i^p)$. As $(L_vx_{1v},L_vx_{2v})\subseteq(V^{2r}_v)_\reg$ for some $v\in\tR'$, we know that $A$ is quasi-isogenous to $A_0^{2r}$, which is impossible by \cite{RSZ}*{Lemma~8.7}. It follows that the supports of $Z(x_1)'_L$ and $Z(x_2)'_L$ have nonempty intersection, which is a contradiction to Lemma \ref{le:disjoint}(1). Thus, the claim hence the lemma are proved.
\end{proof}

To study the general case, we need the following vanishing result.

\begin{lem}\label{le:split_tempered}
Let the situation be as in Proposition \ref{pr:index_split} with $p\neq\ell$. Then for every integer $m\geq 0$, we have $(\rH^{2r}(\cX_m,\dQ_\ell(r))\otimes_\dQ\dQ^\ac)_\fm=0$ where $\fm\coloneqq\fm_\pi^\tR\cap\dS^{\tR\cup\tV_F^{(p)}}_{\dQ^\ac}$.
\end{lem}

\begin{proof}
For every integer $m\geq 0$, put $Y_m\coloneqq\cX_m\otimes_{O_K}k$, $Y_{m,0}\coloneqq Y_m^{\r{red}}$, and for $0\leq j\leq n-1$, denote by $Y_{m,j}$ the Zariski closed subset of $Y_m$ on which the formal part of $A[u^{\tc,\infty}]$ has height at least $j+1$. By the similar argument of \cite{HT01}*{Corollary~III.4.4}, we know that $Y_{m,j}^\circ\coloneqq Y_{m,j}\setminus Y_{m,j+1}$ is smooth over $k$ of pure dimension $n-1-j$.\footnote{In the notation of \cite{HT01}*{Section~III.4}, our $Y_{m,j}^\circ$ is parallel to $\ol{X}_{U^p,m}^{(n-1-j)}$.} Applying Corollary \ref{co:tempered}(2) to $\dS=\dS^{\tR\cup\tV_F^{(p)}}$, $\dL=\dQ^\ac$, and $\fm=\fm_\pi^\tR\cap\dS^{\tR\cup\tV_F^{(p)}}_{\dQ^\ac}$, it suffices to show that for every $m\geq 0$
\begin{enumerate}
  \item $(\rH^{2r}(X_m,\dQ_\ell(r))\otimes_\dQ\dQ^\ac)_\fm=0$; and

  \item $(\rH^i(Y_{m,j}^\circ\otimes_k\ol\dF_p,\dQ_\ell)\otimes_\dQ\dQ^\ac)_\fm=0$ for every $i\leq 2r-2(j+1)$ and every $0\leq j\leq n-1$.
\end{enumerate}
Part (1) has already been proved in Proposition \ref{pr:tempered_generic}(2) (as we have assumed Hypothesis \ref{hy:galois}).

Part (2) follows from the following stronger statement:
\begin{enumerate}
  \setcounter{enumi}{2}
  \item For an arbitrary embedding $\dQ^\ac\hookrightarrow\ol\dQ_\ell$, $\rH^i(Y_{m,j}^\circ\otimes_k\ol\dF_p,\ol\dQ_\ell)_\fm=0$ for every $m\geq0$ unless $j=0$ and $i=2r-1$.
\end{enumerate}
The argument for (3) is similar to the proof of \cite{CS17}*{Theorem~6.3.1}. For $m\geq 0$ and $0\leq j\leq n-1$, let $I_{m,j}$ be the Igusa variety (of the first kind) so that $Y_{m,j}^\circ$ is the disjoint union of finitely many $I_{m,j}$ (see \cite{HT01}*{Section~IV.1}). For each $j$, we obtain a projective system $\{I_{m,j}\res m\geq 0\}$ with finite \'{e}tale transition morphisms. If $\rH^i(Y_{m,j}^\circ\otimes_k\ol\dF_p,\ol\dQ_\ell)_\fm=0$ for all $m\geq 0$, $i$, and $j$, then we are done. Otherwise, let $j$ be the maximal integer such that $\rH^i(Y_{m,j}^\circ\otimes_k\ol\dF_p,\ol\dQ_\ell)_\fm\neq 0$ for some $m$ and $i$. Then $\varinjlim_{m}\rH^i(I_{m,j}\otimes_k\ol\dF_p,\ol\dQ_\ell)_\fm\neq 0$. Now we would like to apply \cite{CS17}*{Corollary~6.1.4}, where in our case, the set $B(G,\mu^{-1})$ is identified with $\{0,\dots,n-1\}$ under which $d=n-j=2r-1-j$; and $\r{Ig}^j$ is the perfection of $\sI^j_{\r{Mant}}\coloneqq\varprojlim_m\sI^j_{\r{Mant},m}$ \cite{CS17}*{Proposition~4.3.8} in which $\sI^j_{\r{Mant},m}$ is a finite Galois cover of $I_{m,j}$.\footnote{The Galois cover comes from the fact that in the definition of $\sI^j_{\r{Mant},m}$, there is also a level structure on the formal part of $A[u^{\tc,\infty}]$.} Then we have $\rH^i(\r{Ig}^j\otimes_k\ol\dF_p,\ol\dQ_\ell)_\fm\neq 0$ which, by \cite{CS17}*{Corollary~6.1.4}\footnote{Strictly speaking, the authors assumed that the level at $p$ is hyperspecial maximal. In our case, we only require that $L_{\ul{u}}$ is hyperspecial. However, by our special signature condition, the argument of \cite{CS17} works in our case verbatim.} (for the coefficients $\ol\dQ_\ell$), implies that $i$ can only be $2r-1-j$. In particular, combining with the Poincar\'{e} duality, we have $[\rH_c(\sI^j_{\r{Mant}}\otimes_k\ol\dF_p,\ol\dQ_\ell)]_\fm\neq 0$ (where we adopted the notation from \cite{CS17}*{Theorem~5.5.7}). By the local-global compatibility at split places (\cite{Shi}*{Theorem~1.1} or more generally \cite{KMSW}*{Theorem~1.7.1}), we have $\Pi_u\simeq\pi_{\ul{u}}$, where we recall that $\Pi$ is the automorphic base change of $\pi$ in Notation \ref{co:galois}(1). In particular, $\Pi_u$ is a tempered principal series by (a). Then by \cite{CS17}*{Theorem~5.5.7} (together with the modification in the proof of \cite{LTXZZ}*{Theorem~D.1.3}) and the very strong multiplicity one property \cite{Ram}*{Theorem~A}, we must have $j=0$ hence $i=2r-1$. Thus, (3) follows.

The lemma is proved.
\end{proof}

\begin{remark}\label{re:split_tempered}
In fact, we conjecture that Lemma \ref{le:split_tempered} remains true without condition (a) in Proposition \ref{pr:index_split}. If this is confirmed, then we may remove condition (2) in Assumption \ref{st:main}.
\end{remark}

\begin{proof}[Proof of Proposition \ref{pr:index_split}]
The last part of the proposition has been confirmed in Lemma \ref{le:index_split}. We prove the first part. We may assume $p\neq\ell$ since otherwise it has been covered in Lemma \ref{le:index_split}. Fix an integer $m\geq 0$ such that $L_{\ul{u}}$ contains $L_{\ul{u},m}$.

It suffices to show that there exists $\rs\in\dS^{\tR\cup\tV_F^{(p)}}\setminus\fm_\pi^\tR$ such that for every $x_1,x_2\in V^r\otimes_{\dA_F}\dA_F^\infty$ satisfying $T(x_1),T(x_2)\in\Herm_r^\circ(F)^+$ and $(L_vx_{1v},L_vx_{2v})\subseteq(V^{2r}_v)_\reg$ for some $v\in\tR'\setminus\{\ul{u}\}$ (which is nonempty as we assume $|\tR'|\geq 2$), we have
\begin{align*}
\langle\rs^*Z(x_1)_L,\rs^*Z(x_2)_L\rangle^\ell_{X_{L,u},E_u}=0.
\end{align*}
By Lemma \ref{le:base_change} and Lemma \ref{le:moduli_cycle}, it suffices to have
\begin{align}\label{eq:index_split_1}
\langle\rs^*Z(x_1)'_L,\rs^*Z(x_2)'_L\rangle^\ell_{X_m,K}=0.
\end{align}
To compute the local index on $X_m$, we use the model $\cX_m$ constructed above. Take $\rs\in\dS^{\tR\cup\tV_F^{(p)}}_{\dQ^\ac}$ that is an $\ell$-tempered $\dQ^\ac$-\'{e}tale correspondence of $\cX_m$, which exists by Lemma \ref{le:split_tempered} and Corollary \ref{co:tempered}(1). Then by Proposition \ref{pr:tempered_integral}, we have
\[
\langle\rs^*Z(x_1)'_L,\rs^*Z(x_2)'_L\rangle^\ell_{X_m,K}=[\rs^*\cZ(x_1)'_L].[\rs^*\cZ(x_2)'_L],
\]
where $\cZ(x_i)'_L$ is the Zariski closure of $Z(x_i)'_L$ in $\cX_m$ for $i=1,2$. By the similar argument used in the proof of Lemma \ref{le:index_split}, $\rs^*\cZ(x_1)'_L$ and $\rs^*\cZ(x_2)'_L$ have disjoint supports, which implies $\rs^*\cZ(x_1)'_L.\rs^*\cZ(x_2)'_L=0$. Thus, \eqref{eq:index_split_1} hence the proposition hold with $\rs_1^u=\rs_2^u=\rs$.
\end{proof}

\section{Local indices at inert places: unramified case}
\label{ss:inert1}

In this section, we compute local indices at places in $\tV_E^\inert$ that are not above $\tR\cup\tS$. Our goal is to prove the following proposition.

\begin{proposition}\label{pr:index_inert}
Let $\tR$, $\tR'$, $\ell$, and $L$ be as in Definition \ref{de:kernel_geometric}. Take an element $u\in\tV_E^\inert$ such that $\ul{u}\not\in\tS$ and whose underlying rational prime $p$ is odd and satisfies $\tV_F^{(p)}\cap\tR\subseteq\tV_F^\spl$. Then we have
\[
\log q_u\cdot \vol^\natural(L)\cdot I_{T_1,T_2}(\phi^\infty_1,\phi^\infty_2,\rs_1,\rs_2,g_1,g_2)^\ell_{L,u}=
\fE_{T_1,T_2}((g_1,g_2),\Phi_\infty^0\otimes(\rs_1\phi^\infty_1\otimes(\rs_2\phi^\infty_2)^\tc))_u
\]
for every $(\tR,\tR',\ell,L)$-admissible sextuple $(\phi^\infty_1,\phi^\infty_2,\rs_1,\rs_2,g_1,g_2)$ and every pair $(T_1,T_2)$ in $\Herm_r^\circ(F)^+$, where the right-hand side is defined in Definition \ref{de:analytic_kernel} with the Gaussian function $\Phi_\infty^0\in\sS(V^{2r}\otimes_{\dA_F}F_\infty)$ (Notation \ref{st:h}(H3)), and $\vol^\natural(L)$ is defined in Definition \ref{de:measure}.
\end{proposition}

To prove Proposition \ref{pr:index_inert}, we may rescale the hermitian form on $V$ hence assume that $\psi_{F,v}$ is unramified and that $\Lambda^\tR_v$ is either a self-dual or an almost self-dual lattice of $V_v$ for every $v\in\tV_F^{(p)}\setminus\tV_F^\spl$.

\begin{lem}\label{le:index_inert}
Let the situation be as in Proposition \ref{pr:index_inert}. If the weaker version of Proposition \ref{pr:index_inert} where we only consider $(\tR,\tR',\ell,L)$-admissible sextuples $(\phi^\infty_1,\phi^\infty_2,\rs_1,\rs_2,g_1,g_2)$ in which $g_{1v}=g_{2v}=1_{2r}$ for every $v\in\tV_F^{(\infty)}\cup\tV_F^{(p)}$ holds, then the original Proposition \ref{pr:index_inert} holds.
\end{lem}

\begin{proof}
Take an arbitrary $(\tR,\tR',\ell,L)$-admissible sextuple $(\phi^\infty_1,\phi^\infty_2,\rs_1,\rs_2,g_1,g_2)$. For $i=1,2$, we may find elements $a_i\in\GL_r(E)$ and $b_i\in\Herm_r(F)$ such that $m(a_i)^{-1}n(b_i)^{-1}g_{iv}\in K_{r,v}$ for every $v\in\tV_F^{(p)}\setminus\tR'$. For $i=1,2$, put
\[
\tilde{T}_i\coloneqq\pres{\rt}{a}_i^\tc T_ia_i,\quad
\tilde\phi^\infty_i\coloneqq\prod_{v\in\tR'}\omega_{r,v}(m(a_i)^{-1}n(b_i)^{-1})\phi^\infty_i,
\]
and let $\tilde{g}_i$ be the away-from-$(\tR'\cup\tV_F^{(\infty)}\cup\tV_F^{(p)})$-component of the element $m(a_i)^{-1}n(b_i)^{-1}g_i$. Then $(\tilde\phi^\infty_1,\tilde\phi^\infty_2,\rs_1,\rs_2,\tilde{g}_1,\tilde{g}_2)$ is an $(\tR,\tR',\ell,L)$-admissible sextuple. By Lemma \ref{le:invariance}, we have
\[
I_{T_1,T_2}(\phi^\infty_1,\phi^\infty_2,\rs_1,\rs_2,g_1,g_2)^\ell_{L,u}=C\cdot
I_{\tilde{T}_1,\tilde{T}_2}(\tilde\phi^\infty_1,\tilde\phi^\infty_2,\rs_1,\rs_2,\tilde{g}_1,\tilde{g}_2)^\ell_{L,u}
\]
in which
\[
C=
\(\frac{\omega_{r,\infty}(m(a_1)^{-1}n(b_1)^{-1}g_{1\infty})\phi^0_\infty(\tilde{T}_1)}{\phi^0_\infty(\tilde{T}_1)}\)\cdot
\(\frac{\omega_{r,\infty}(m(a_2)^{-1}n(b_2)^{-1}g_{2\infty})\phi^0_\infty(\tilde{T}_2)}{\phi^0_\infty(\tilde{T}_2)}\)^\tc.
\]
On the other hand, from Definition \ref{de:analytic_kernel}, we have
\[
\fE_{T_1,T_2}((g_1,g_2),\Phi_\infty^0\otimes(\rs_1\phi^\infty_1\otimes(\rs_2\phi^\infty_2)^\tc))_u
=C\cdot\fE_{\tilde{T}_1,\tilde{T}_2}((\tilde{g}_1,\tilde{g}_2),
\Phi_\infty^0\otimes(\rs_1\tilde\phi^\infty_1\otimes(\rs_2\tilde\phi^\infty_2)^\tc))_u
\]
with the same $C$. The lemma follows.
\end{proof}

In order to deal with spherical Hecke operators, we consider the projective system of Shimura varieties $\{X_{\tilde{L}}\}$ indexed by open compact subgroups $\tilde{L}\subseteq L$ satisfying $\tilde{L}_v=L_v$ for $v\in\tV_F^{(p)}\setminus\tV_F^\spl$.

We invoke Notation \ref{co:auxiliary} together with Notation \ref{st:auxiliary}, which is possible since $\tV_F^{(p)}\cap\tV_F^\ram=\emptyset$. There is a projective system $\{\cX_{\tilde{L}}\}$ of smooth projective schemes over $O_K$ (see \cite{LZ}*{Section~11.2}) with
\[
\cX_{\tilde{L}}\otimes_{O_K}K=X'_{\tilde{L}}\otimes_{E'}K
=\(X_{\tilde{L}}\otimes_{E}Y\)\otimes_{E'}K,
\]
and finite \'{e}tale transition morphisms. In particular, $\dS^\tR$ is naturally a ring of \'{e}tale correspondences of $\cX_L$.

\begin{lem}\label{le:tempered}
If $\rs$ is a product of two elements in $(\dS^\tR_{\dQ^\ac})^{\langle\ell\rangle}_{L_\tR}$, then it gives an $\ell$-tempered $\dQ^\ac$-\'{e}tale correspondence of $\cX_L$ (Definition \ref{de:tempered_integral}).
\end{lem}

\begin{proof}
We have a short exact sequence
\begin{align}\label{eq:tempered}
\rH^{2r}_{\cX_L\otimes_{O_K}k}(\cX_L,\dQ_\ell(r))\to\rH^{2r}(\cX_L,\dQ_\ell(r))\to\rH^{2r}(X'_L,\dQ_\ell(r)),
\end{align}
in which we have
\[
\rH^{2r}_{\cX_L\otimes_{O_K}k}(\cX_L,\dQ_\ell(r))\simeq\rH^{2r-2}(\cX_L\otimes_{O_K}k,\dQ_\ell(r-1))
\]
by the absolute purity theorem \cite{Fuj02}, since $\cX_L$ is smooth over $O_K$. By the Hochschild--Serre spectral sequence and the Weil conjecture, the natural maps
\begin{align*}
\rH^{2r}(X'_L,\dQ_\ell(r)) &\to\rH^{2r}(X'_L\otimes_K\ol\dQ_p,\dQ_\ell(r)) \\
\rH^{2r-2}(\cX_L\otimes_{O_K}k,\dQ_\ell(r-1))&\to\rH^{2r-2}(\cX_L\otimes_{O_K}\ol\dF_p,\dQ_\ell(r-1))
\end{align*}
are both injective.

By definition, every element in $(\dS^\tR_{\dQ^\ac})^{\langle\ell\rangle}_{L_\tR}$ annihilates $\rH^{2r}(X'_L\otimes_K\ol\dQ_p,\dQ_\ell(r))\otimes_\dQ\dQ^\ac$. By the Poincar\'{e} duality and the smooth proper base change theorem, every element in $(\dS^\tR_{\dQ^\ac})^{\langle\ell\rangle}_{L_\tR}$ also annihilates $\rH^{2r-2}(\cX_L\otimes_{O_K}\ol\dF_p,\dQ_\ell(r-1))\otimes_\dQ\dQ^\ac$. In particular, $\rs$ annihilates $\rH^{2r}(\cX_L,\dQ_\ell(r))\otimes_\dQ\dQ^\ac$ as each factor annihilates a graded piece in the two-step filtration of $\rH^{2r}(\cX_L,\dQ_\ell(r))$ given by \eqref{eq:tempered}. The lemma is proved.
\end{proof}

We first recall the uniformization of $\{\cX_{\tilde{L}}\}$ along the supersingular locus from \cite{LZ}*{Section~13.1}. Fix a complete maximal unramified extension $\breve{K}$ of $K$. Recall that we have fixed a $u$-nearby space $\pres{u}V$ and an isomorphism $\pres{u}{V}\otimes_F\dA_F^{\ul{u}}\simeq V\otimes_{\dA_F}\dA_F^{\ul{u}}$ from Notation \ref{st:h}(H9). We have a compatible system of isomorphisms
\begin{align}\label{eq:uniformization}
\cX_{\tilde{L}}^\wedge\simeq\(\pres{u}{H}(F)\backslash\cN\times H(\dA_F^{\infty,\ul{u}})/\tilde{L}^{\ul{u}}\)
\times_{\Spf O_K}\cY^\wedge
\end{align}
of formal schemes over $O_{\breve{K}}$ for every $\tilde{L}\subseteq L$ considered as before. Here, $\cX_{\tilde{L}}^\wedge$ denotes the completion of $\cX_{\tilde{L}}\otimes_{O_K}O_{\breve{K}}$ along its supersingular locus; $\cN$ is the relative unitary Rapoport--Zink space over $\Spf O_{\breve{K}}$ as considered in \cite{LZ}*{Section~2.1}; and $\cY^\wedge$ denotes the completion of $\cY$ along its special fiber.

We then recall the notion of integral special cycles. Take an integer $m\geq 1$ and an element $\phi^\infty\in\sS(V^m\otimes_{\dA_F}\dA_F^\infty)^L$ that is $p$-basic (Definition \ref{de:basic}).

For every element $T\in\Herm_m^\circ(F)^+$, there are following constructions.

\begin{itemize}
  \item We have a cycle $Z_T(\phi^\infty)_L\in\rZ^m(X_L)_\dC$ (Definition \ref{co:special_cycle}). When $\phi^\infty$ is the characteristic function of some open compact subset of $V^m\otimes_{\dA_F}\dA_F^\infty$, we have and a morphism
      \[
      Z'_T(\phi^\infty)_L\to X'_L
      \]
      defined as the disjoint union of finite and unramified morphisms $Z'(x)_L\to X'_L$ (Definition \ref{de:moduli_cycle}) for $x\in L\backslash\supp(\phi^\infty)$, whose induced cycle coincides with the restriction of $Z_T(\phi^\infty)_L$ to $X'_L$ (Lemma \ref{le:moduli_cycle}). By moduli interpretation, the morphism $Z'_T(\phi^\infty)_L\to X'_L$ extends naturally to a finite and unramified morphism
      \[
      \cZ_T(\phi^\infty)_L\to\cX_L
      \]
      (\cite{LZ}*{Section~13.3}).

  \item When $\phi^\infty=\phi^\infty_1\otimes\cdots\otimes\phi^\infty_r$ with $\phi^\infty_j\in\sS(V\otimes_{\dA_F}\dA_F^\infty)^L$ that is $p$-basic and is the characteristic function of some open compact subset of $V\otimes_{\dA_F}\dA_F^\infty$, we denote by $\pres{\rK}\cZ_T(\phi^\infty)_L$ the component of
      \[
      \sO_{\cZ_{t_1}(\phi^\infty_1)_L}\overset{\dL}\otimes_{\sO_{\cX_L}}\cdots
      \overset{\dL}\otimes_{\sO_{\cX_L}}\sO_{\cZ_{t_r}(\phi^\infty_r)_L}
      \]
      supported on $\cZ_T(\phi^\infty)_L$,\footnote{Here, we note that $\cZ_T(\phi^\infty)_L$ is an open and closed subscheme of $\cZ_{t_1}(\phi^\infty)_L\times_{\cX_L}\cdots\times_{\cX_L}\cZ_{t_r}(\phi^\infty)_L$.} regarded as an element in $\rK_0^\cZ(\cX_L)_\dC$ (see Appendix \ref{ss:beilinson} for the notion of the K-group), where $(t_1,\dots,t_r)$ is the diagonal of $T$ and $\cZ$ denotes the image of $\cZ_T(\phi^\infty)_L$ in $\cX_L$. In general, we may always write $\phi^\infty$ as a finite complex linear combination (possibly after shrinking $L$ away from $\tV_F^{(p)}\setminus\tV_F^\spl$) of those as above, and we define $\pres{\rK}\cZ_T(\phi^\infty)_L$ by linearity. By \cite{GS87}*{Proposition~5.5}, $\pres{\rK}\cZ_T(\phi^\infty)_L$ belongs to $\rF^r\rK_0^\cZ(\cX_L)_\dC$ hence is an extension of $Z_T(\phi^\infty)'_L$ (Definition \ref{de:extension}).

  \item We denote by $\cZ_T(\phi^\infty)_L^\wedge$ the restriction of $\cZ_T(\phi^\infty)_L$ to $\cX_L^\wedge$. Then we have the following description
      \begin{align}\label{eq:uniformization_1}
      \cZ_T(\phi^\infty)_L^\wedge=\sum_{\substack{x\in\pres{u}{H}(F)\backslash\pres{u}{V}^m\\ T(x)=T}}
      \sum_{h\in H^x(F)\backslash H(\dA_F^{\infty,\ul{u}})/L^{\ul{u}}}\phi^{\infty,\ul{u}}(h^{-1}x)\cdot(\cN(x),h)_L\times_{\Spf O_K}\cY^\wedge,
      \end{align}
      where $\cN(x)$ is the special cycle of $\cN$ indexed by $x$ (\cite{KR11}*{Definition~3.2} or \cite{LZ}*{Section~2.3}); $(\cN(x),h)_L$ denotes the corresponding double coset in the expression \eqref{eq:uniformization}; and $H^x$ is the subgroup of $\pres{u}{H}$ of elements that fix every component of $x$.
\end{itemize}

In what follows, for $x=(x_1,\dots,x_m)\in \pres{u}{V}^m$ with $T(x)\in\Herm_m^\circ(F_{\ul{u}})$, we put
\[
\pres{\rK}\cN(x)\coloneqq[\cN(x_1)]\cup\cdots\cup[\cN(x_m)]
\]
as an element in $\rK_0^{\cN(x)}(\cN)$. See \cite{Zha}*{Appendix~B} for the analogue of Gillet--Soul\'{e} K-groups for formal schemes; and we denote similarly by $[\;]$ the associated element in the K-group.

\begin{proof}[Proof of Proposition \ref{pr:index_inert}]
By Lemma \ref{le:base_change} and Definition \ref{de:analytic_kernel}, it suffices to show that for every pair of $p$-basic elements $\phi^\infty_1,\phi^\infty_2\in\sS(V^m\otimes_{\dA_F}\dA_F^\infty)^L$ satisfying that $\supp(\phi^\infty_{1v}\otimes(\phi^\infty_{2v})^\tc)\subseteq(V^{2r}_v)_\reg$ for $v\in\tR'$, and every pair of elements $\rs_1,\rs_2$ each of which is a product of two elements in $(\dS^\tR_{\dQ^\ac})^{\langle\ell\rangle}_{L_\tR}$, we have
\begin{multline}\label{eq:index_inert_2}
\frac{\vol^\natural(L)}{\deg(Y/K)}\Phi^0_\infty(T_1,T_2)
\langle\rs_1^*Z_{T_1}(\phi^\infty_1)'_L,\rs_2^*Z_{T_2}(\phi^\infty_2)'_L\rangle_{X'_L,K}\\
=\sum_{\substack{T^\Box\in\Herm_{2r}^\circ(F)^+\\\Diff(T^\Box,V)=\{\ul{u}\}\\\partial_{r,r} T^\Box=(T_1,T_2)}}
\frac{1}{\log q_u}W'_{T^\Box}(0,1_{4r},\CF_{(\Lambda^\tR_{\ul{u}})^{2r}})
\prod_{v\neq\ul{u}}W_{T^\Box}(0,1_{4r},(\Phi_\infty^0\otimes(\rs_1\phi^\infty_1\otimes(\rs_2\phi^\infty_2)^\tc))_v).
\end{multline}
Now using $\vol^\natural(L_{\ul{u}})=1$ and $\prod_{v\in\tV_F\setminus\{\ul{u}\}}\gamma_{V_v,\psi_{F,v}}^{2r}=1$, \eqref{eq:index_inert_2} is equivalent to
\begin{multline}\label{eq:index_inert_4}
\frac{\vol(H(F_\infty)L^{\ul{u}})}{\deg(Y/K)}\Phi^0_\infty(T_1,T_2)
\langle\rs_1^*Z_{T_1}(\phi^\infty_1)'_L,\rs_2^*Z_{T_2}(\phi^\infty_2)'_L\rangle_{X'_L,K}\\
=\sum_{\substack{T^\Box\in\Herm_{2r}^\circ(F)^+\\\Diff(T^\Box,V)=\{\ul{u}\}\\\partial_{r,r} T^\Box=(T_1,T_2)}}
\frac{b_{2r,\ul{u}}(0)}{\log q_u}W'_{T^\Box}(0,1_{4r},\CF_{(\Lambda^\tR_{\ul{u}})^{2r}})
\prod_{v\neq\ul{u}}\frac{b_{2r,v}(0)}{\gamma_{V_v,\psi_{F,v}}^{2r}}
W_{T^\Box}(0,1_{4r},(\Phi_\infty^0\otimes(\rs_1\phi^\infty_1\otimes(\rs_2\phi^\infty_2)^\tc))_v).
\end{multline}

By Proposition \ref{pr:tempered_integral} and Lemma \ref{le:tempered}, we have
\begin{align*}
\langle\rs_1^*Z_{T_1}(\phi^\infty_1)'_L,\rs_2^*Z_{T_2}(\phi^\infty_2)'_L\rangle_{X'_L,K}
&=\(\rs_1^*\pres{\rK}\cZ_{T_1}(\phi^\infty_1)_L\).\(\rs_2^*\pres{\rK}\cZ_{T_2}(\phi^\infty_2)_L\)  \\ &=\chi\(\pi_*\(\rs_1^*\pres{\rK}\cZ_{T_1}(\phi^\infty_1)_L\cup\rs_2^*\pres{\rK}\cZ_{T_2}((\phi^\infty_2)^\tc)_L\)\),
\end{align*}
where $\pi\colon\cX_L\to\Spec O_K$ denotes the structure morphism. As $\supp(\phi^\infty_{1v}\otimes(\phi^\infty_{2v})^\tc)\subseteq(V^{2r}_v)_\reg$ for $v\in\tR'$, the support of $\rs_1^*\pres{\rK}\cZ_{T_1}(\phi^\infty_1)_L\cup\rs_2^*\pres{\rK}\cZ_{T_2}((\phi^\infty_2)^\tc)_L$ is contained in the supersingular locus of $\cX_L$. Moreover, since $\rs_1^*$ and $\rs_2^*$ preserve the supersingular locus, we have
\[
\chi\(\pi_*\(\rs_1^*\pres{\rK}\cZ_{T_1}(\phi^\infty_1)_L\cup\rs_2^*\pres{\rK}\cZ_{T_2}((\phi^\infty_2)^\tc)_L\)\)=
\chi\(\pi^\wedge_*\(\rs_1^*\pres{\rK}\cZ_{T_1}(\phi^\infty_1)_L^\wedge\cup\rs_2^*\pres{\rK}\cZ_{T_2}((\phi^\infty_2)^\tc)_L^\wedge\)\),
\]
where $\pi^\wedge\colon\cX_L^\wedge\to\Spf O_{\breve{K}}$ denotes the structure morphism. To summarize, the left-hand side of \eqref{eq:index_inert_4} equals
\begin{align}\label{eq:index_inert_3}
\frac{\vol(H(F_\infty)L^{\ul{u}})}{\deg(Y/K)}\Phi^0_\infty(T_1,T_2)\cdot
\chi\(\pi^\wedge_*\(\rs_1^*\pres{\rK}\cZ_{T_1}(\phi^\infty_1)_L^\wedge\cup\rs_2^*\pres{\rK}\cZ_{T_2}((\phi^\infty_2)^\tc)_L^\wedge\)\).
\end{align}
From \eqref{eq:uniformization_1}, it is straightforward to see that
\[
\rs_i^*\pres{\rK}\cZ_{T_i}(\phi^\infty_i)_L^\wedge=
\sum_{\substack{x_i\in\pres{u}{H}(F)\backslash\pres{u}{V}^r\\ T(x_i)=T_i}}
\sum_{h_i\in H^{x_i}(F)\backslash H(\dA_F^{\infty,\ul{u}})/L^{\ul{u}}}
(\rs_i\phi^{\infty,\ul{u}}_i)(h_i^{-1}x_i)\cdot(\pres{\rK}\cN(x_i),h_i)_L\times_{\Spf O_K}\cY^\wedge
\]
for $i=1,2$. It follows that
\begin{multline*}
\rs_1^*\pres{\rK}\cZ_{T_1}(\phi^\infty_1)_L^\wedge\cup\rs_2^*\pres{\rK}\cZ_{T_2}((\phi^\infty_2)^\tc)_L^\wedge \\
=\sum_{\substack{T^\Box\in\Herm_{2r}^\circ(F)^+\\ \partial_{r,r} T^\Box=(T_1,T_2)}}
\sum_{\substack{x\in\pres{u}{H}(F)\backslash\pres{u}{V}^{2r}\\ T(x)=T^\Box}}
\sum_{h\in H(\dA_F^{\infty,\ul{u}})/L^{\ul{u}}}
(\rs_1\phi^{\infty,\ul{u}}_1\otimes(\rs_2\phi^{\infty,\ul{u}}_2)^\tc)(h^{-1}x)\cdot(\pres{\rK}\cN(x),h)_L
\times_{\Spf O_K}\cY^\wedge.
\end{multline*}
Now by \cite{LZ}*{Theorem~3.4.1 \& Remark~3.4.2}, we have
\[
\chi\(\pi^\wedge_*\pres{\rK}\cN(x)\)=\frac{b_{2r,\ul{u}}(0)}{\log q_u}W'_{T^\Box}(0,1_{4r},\CF_{(\Lambda^\tR_{\ul{u}})^{2r}})
\]
if $T(x)=T^\Box$. Thus, we have
\begin{multline*}
\eqref{eq:index_inert_3}=\vol(H(F_\infty)L^{\ul{u}})\cdot\Phi^0_\infty(T_1,T_2)\cdot
\sum_{\substack{T^\Box\in\Herm_{2r}^\circ(F)^+\\ \partial_{r,r} T^\Box=(T_1,T_2)}}
\sum_{\substack{x\in\pres{u}{H}(F)\backslash\pres{u}{V}^{2r}\\ T(x)=T^\Box}}
\sum_{h\in H(\dA_F^{\infty,\ul{u}})/L^{\ul{u}}} \\
(\rs_1\phi^{\infty,\ul{u}}_1\otimes(\rs_2\phi^{\infty,\ul{u}}_2)^\tc)(h^{-1}x)\cdot
\(\frac{b_{2r,\ul{u}}(0)}{\log q_u}W'_{T^\Box}(0,1_{4r},\CF_{(\Lambda^\tR_{\ul{u}})^{2r}})\).
\end{multline*}
By Definition \ref{de:measure}, we have
\[
\vol(H(F_v))\cdot\Phi^0_v(T_1,T_2)=\frac{b_{2r,v}(0)}{\gamma_{V_v,\psi_{F,v}}^{2r}}
W_{T^\Box}(0,1_{4r},(\Phi_\infty^0\otimes(\rs_1\phi^\infty_1\otimes(\rs_2\phi^\infty_2)^\tc))_v)
\]
for $v\in\tV_F^{(\infty)}$. By Definition \ref{de:measure}, for (unique) $x\in\pres{u}{H}(F)\backslash\pres{u}{V}^{2r}$ with $T(x)=T^\Box$, we have
\begin{align*}
&\vol(L_v)\sum_{h_v\in H(F_v)/L_v}(\rs_1\phi^{\infty,\ul{u}}_1\otimes(\rs_2\phi^{\infty,\ul{u}}_2)^\tc)_v(h_v^{-1}x) \\
&\qquad=\frac{b_{2r,v}(0)}{\gamma_{V_v,\psi_{F,v}}^{2r}}
W_{T^\Box}(0,1_{4r},(\Phi_\infty^0\otimes(\rs_1\phi^\infty_1\otimes(\rs_2\phi^\infty_2)^\tc))_v)
\end{align*}
for $v\in\tV_F^\fin\setminus\{\ul{u}\}$.

Therefore, we obtain \eqref{eq:index_inert_4} hence \eqref{eq:index_inert_2}. The proposition is proved.
\end{proof}

\section{Local indices at inert places: almost unramified case}
\label{ss:inert2}

In this section, we compute local indices at places in $\tV_E^\inert$ above $\tS$. Our goal is to prove the following proposition.

\begin{proposition}\label{pr:index_inert_almost}
Let $\tR$, $\tR'$, $\ell$, and $L$ be as in Definition \ref{de:kernel_geometric}. Let $(\pi,\cV_\pi)$ be as in Assumption \ref{st:representation}, for which we assume Hypothesis \ref{hy:galois}. Take an element $u\in\tV_E^\inert$ such that $\ul{u}\in\tS$ and whose underlying rational prime $p$ is odd, unramified in $E$, and satisfies $\tV_F^{(p)}\cap\tR\subseteq\tV_F^\spl$. Recall that we have fixed a $u$-nearby space $\pres{u}V$ and an isomorphism $\pres{u}{V}\otimes_F\dA_F^{\ul{u}}\simeq V\otimes_{\dA_F}\dA_F^{\ul{u}}$ from Notation \ref{st:h}(H9). We also fix a $\psi_{E,\ul{u}}$-self-dual lattice $\Lambda^\star_{\ul{u}}$ of $\pres{u}{V_{\ul{u}}}$. Then there exist elements $\rs_1^u,\rs_2^u\in\dS_{\dQ^\ac}^\tR\setminus\fm_\pi^\tR$ such that
\begin{multline*}
\log q_u\cdot \vol^\natural(L)\cdot I_{T_1,T_2}(\phi^\infty_1,\phi^\infty_2,\rs_1^u\rs_1,\rs_2^u\rs_2,g_1,g_2)^\ell_{L,u} \\
=\fE_{T_1,T_2}((g_1,g_2),\Phi_\infty^0\otimes(\rs_1^u\rs_1\phi^\infty_1\otimes(\rs_2^u\rs_2\phi^\infty_2)^\tc))_u \\
-\frac{\log q_u}{q_u^r-1}E_{T_1,T_2}((g_1,g_2),\Phi_\infty^0\otimes
(\rs_1^u\rs_1\phi^{\infty,\ul{u}}_1\otimes(\rs_2^u\rs_2\phi^{\infty,\ul{u}}_2)^\tc)\otimes\CF_{(\Lambda^\star_{\ul{u}})^{2r}})
\end{multline*}
for every $(\tR,\tR',\ell,L)$-admissible sextuple $(\phi^\infty_1,\phi^\infty_2,\rs_1,\rs_2,g_1,g_2)$ and every pair $(T_1,T_2)$ in $\Herm_r^\circ(F)^+$, where the right-hand side is defined in Definition \ref{de:analytic_kernel} with the Gaussian function $\Phi_\infty^0\in\sS(V^{2r}\otimes_{\dA_F}F_\infty)$ (Notation \ref{st:h}(H3)), and $\vol^\natural(L)$ is defined in Definition \ref{de:measure}.
\end{proposition}

To prove Proposition \ref{pr:index_inert_almost}, we may rescale the hermitian form on $V$ hence assume that $\psi_{F,v}$ is unramified and that $\Lambda^\tR_v$ is either a self-dual or an almost self-dual lattice of $V_v$ for every $v\in\tV_F^{(p)}\setminus\tV_F^\spl$, and moreover that $\Lambda^\star_{\ul{u}}$ is a self-dual lattice of $\pres{u}{V_{\ul{u}}}$.

In order to deal with spherical Hecke operators, we consider the projective system of Shimura varieties $\{X_{\tilde{L}}\}$ indexed by open compact subgroups $\tilde{L}\subseteq L$ satisfying $\tilde{L}_v=L_v$ for $v\in\tV_F^{(p)}\setminus\tV_F^\spl$.

We invoke Notation \ref{co:auxiliary} together with Notation \ref{st:auxiliary}, which is possible since $\tV_F^{(p)}\cap\tV_F^\ram=\emptyset$. There is a projective system $\{\cX_{\tilde{L}}\}$ of strictly semistable projective schemes over $O_K$ (see \cite{LZ}*{Section~11.3})\footnote{This is the place where we need that $u$ is unramified over $\dQ$ (and that $K$ is unramified over $E_u$).} with
\[
\cX_{\tilde{L}}\otimes_{O_K}K=X'_{\tilde{L}}\otimes_{E'}K
=\(X_{\tilde{L}}\otimes_{E}Y\)\otimes_{E'}K,
\]
and finite \'{e}tale transition morphisms. In particular, $\dS^\tR$ is naturally a ring of \'{e}tale correspondences of $\cX_L$.

\begin{lem}\label{le:tempered_almost}
Let the situation be as in Proposition \ref{pr:index_inert_almost}. Then there exists an element in $\dS_{\dQ^\ac}^\tR\setminus\fm_\pi^\tR$ that gives an $\ell$-tempered $\dQ^\ac$-\'{e}tale correspondence of $\cX_L$ (Definition \ref{de:tempered_integral}).
\end{lem}

\begin{proof}
The proof relies on Arthur's multiplicity formula for tempered global $L$-packets \cite{KMSW}*{Theorem~1.7.1}, which we first recall, using the language for unitary groups adopted in \cite{GGP12}*{Section~25}. Recall that $\Pi=\Pi_1\boxplus\cdots\boxplus\Pi_s$ from Notation \ref{co:galois}(3) which is the automorphic base change of $\pi$ as in Assumption \ref{st:representation}. Put $A_\Pi\coloneqq\mu_2^{\{1,\dots,s\}}$. For every place $v\in\tV_F$,
\begin{itemize}
  \item $\Pi_v$ determines a conjugate-symplectic representation $M_v$ of $\r{WD}(E_v)$ of dimension $n$;

  \item there is a finite abelian $2$-group $A_{M_v}$ attached to $M_v$;

  \item every character $\chi_v\colon A_{M_v}\to\dC^\times$ gives a pair $(V^{\chi_v},\pi^{\chi_v})$ in the Langlands--Vogan packet of $M_v$, unique up to isomorphism, in which $V^{\chi_v}$ is a hermitian space over $F_v$ of rank $n$ and $\pi^{\chi_v}$ is an irreducible admissible representation of $\rU(V^{\chi_v})(F_v)$;

  \item we have a homomorphism $\alpha_v\colon A_\Pi\to A_{M_v}$.
\end{itemize}
Denote by $\alpha\colon A_\Pi\to\prod_{v\in\tV_F}A_{M_v}$ the product of $\alpha_v$ for $v\in\tV_F$. We say that a collection $\chi=\{\chi_v\res v\in\tV_F\}$ of characters in which all but finitely many are trivial is coherent (resp.\ incoherent) if the character $\prod_{v\in\tV_F}\chi_v\circ\alpha\colon A_\Pi\to\dC^\times$ is trivial (resp.\ nontrivial). Then Arthur's multiplicity formula states that
\begin{itemize}
  \item[(a)] If $\chi$ is incoherent, then either $\otimes_vV^{\chi_v}$ is incoherent or it is coherent but $\otimes_v\pi^{\chi_v}$ does not appear in the discrete spectrum. If $\chi$ is coherent, then there exists a hermitian space $V^\chi$ over $E$, unique up to isomorphism, such that $V^\chi_v\simeq V^{\chi_v}$ for every $v\in\tV_F$; and the representation $\otimes_v\pi^{\chi_v}$ appears in the discrete spectrum of $\rU(V^\chi)$ with multiplicity one. Moreover, every discrete automorphic representation of $\rU(\tilde{V})(\dA_F)$ for some hermitian space $\tilde{V}$ over $E$ of rank $n$ with $\Pi$ its automorphic base change is obtained from this way.
\end{itemize}
Now we take a special look at the places $w$ and $\ul{u}$.
\begin{itemize}
  \item[(b)] We may canonically identify $A_{M_w}$ with $\mu_2^\fI$ from Notation \ref{co:galois}(2). Then the homomorphism $\alpha_w\colon\mu_2^{\{1,\dots,s\}}\to\mu_2^\fI$ is the one induced by the map $\fI\to\{1,\dots,s\}$ given by the partition $\fI=\fI_1\sqcup\cdots\sqcup\fI_s$.

  \item[(c)] By (a), $\Pi_u$ is the standard base change of $\pi_{\ul{u}}$. By \cite{LTXZZ}*{Lemma~C.2.3}, we have $M_{\ul{u}}=M_{\ul{u}}^2+M_{\ul{u}}^{n-2}$, where $M_{\ul{u}}^2$ corresponds to the Steinberg representation of $\GL_2(E_u)$ and $M_{\ul{u}}^{n-2}$ corresponds to a tempered unramified principal series of $\GL_{n-2}(E_u)$, which implies $A_{M_{\ul{u}}}=A_{M_{\ul{u}}^2}\times A_{M_{\ul{u}}^{n-2}}$ in which $A_{M_{\ul{u}}^2}=\mu_2$.

  \item[(d)] Without lost of generality, we may assume that $\Pi_{1u}$ is ramified. Then the composition of $\alpha_{\ul{u}}$ and the projection $A_{M_{\ul{u}}}\to A_{M_{\ul{u}}^2}=\mu_2$ coincides with the projection $\mu_2^{\{1,\dots,s\}}\to\mu_2$ to the first factor.
\end{itemize}

Next, we recall some facts from \cite{LTXZZ}*{Sections~5.1,~5.2,~5.4,~\&~5.5}\footnote{Strictly speaking, \cite{LTXZZ} has more conditions on the place $\ul{u}$ and the level at $p$. However, for those facts we will use in this proof, it is straightforward to remove those extra conditions.} about the reduction of the scheme $\cX_{\tilde{L}}$. Denote by $L^\star_{\ul{u}}\subseteq\pres{u}{H}(F_{\ul{u}})$ the stabilizer of $\Lambda^\star_{\ul{u}}$, which is a hyperspecial maximal subgroup. We have a decomposition
\[
\cX_{\tilde{L}}\otimes_{O_K}\ol\dF_p=Y_{\tilde{L}}^\circ\bigcup Y_{\tilde{L}}^\bullet
\]
that is compatible with changing $\tilde{L}$, in which
\begin{itemize}
  \item $Y_{\tilde{L}}^\circ$ is a $\dP^{2r-1}$-fibration over
      \[
      \pres{u}{H}(F)\backslash\pres{u}{H}(\dA_F^\infty)/\tilde{L}^{\ul{u}}L^\star_{\ul{u}}
      \times(\cY\otimes_{O_K}\ol\dF_p);
      \]

  \item $Y_{\tilde{L}}^\bullet$ is proper and smooth over $\ol\dF_p$ of dimension $2r-1$;

  \item the intersection $Y_{\tilde{L}}^\dag\coloneqq Y_{\tilde{L}}^\circ\cap Y_{\tilde{L}}^\bullet$ is a relative Fermat hypersurface in $Y_{\tilde{L}}^\circ$.
\end{itemize}
By Corollary \ref{co:tempered} and Proposition \ref{pr:tempered_generic}(2), it suffices to show that for an arbitrary embedding $\dQ^\ac\hookrightarrow\ol\dQ_\ell$, we have
\begin{enumerate}
  \item $\rH^i(Y^\circ_L,\ol\dQ_\ell)_\fm=0$ for $i\leq 2r-2$,

  \item $\rH^i(Y^\bullet_L,\ol\dQ_\ell)_\fm=0$ for $i\leq 2r-2$,

  \item $\rH^i(Y^\dag_L,\ol\dQ_\ell)_\fm=0$ for $i\leq 2r-3$,
\end{enumerate}
where $\fm\coloneqq\fm_\pi^\tR\cap\dS^\tR_{\dQ^\ac}$. Note that we have used the Gysin exact sequence and the absolute purity theorem \cite{Fuj02} to switch the cohomology from open strata to closed strata.

For (1), we have $\rH^i(Y^\circ_L,\ol\dQ_\ell)=0$ when $i$ is odd. When $i$ is even, $\rH^i(Y^\circ_L,\ol\dQ_\ell)_\fm$ is a direct sum of $L^{\ul{u}}\times L^\star_{\ul{u}}$-invariants of $\pi'$ for finitely many cuspidal automorphic representation $\pi'$ of $\pres{u}{H}(\dA_F)$ satisfying that $\pi'_\infty$ is trivial and that $\pi'_v\simeq\pi_v$ for all but finitely many $v\in\tV_F^\spl$. For every such $\pi'$, let $\Pi'$ be its automorphic base change, which is isobaric automorphic representation of $\GL_n(\dA_E)$ \cite{KMSW}*{Theorem~1.7.1}. Since $\Pi'_u\simeq\Pi_u$ for all but finitely many $u\in\tV_E^\spl$, we must have $\Pi'\simeq\Pi$ by \cite{Ram}*{Theorem~A}. Therefore, we have
\[
\rH^i(Y^\circ_L,\ol\dQ_\ell)_\fm\simeq\bigoplus_{\substack{\chi=\{\chi^v\}\\ \chi\circ\alpha=1 \\ V^\chi\simeq\pres{u}{V}}}
\(\(\otimes_{v\in\tV_F^\fin\setminus\{\ul{u}\}}\pi^{\chi_v}\)^{L^{\ul{u}}}\otimes\(\pi^{\chi_{\ul{u}}}\)^{L^\star_{\ul{u}}}\)^{\oplus\deg(Y/K)}.
\]
However, since $\Pi_{\ul{u}}$ is ramified, we have $\(\pi^{\chi_{\ul{u}}}\)^{L^\star_{\ul{u}}}=0$ for every $\chi_{\ul{u}}$. Thus, (1) follows.

For (3), by the Lefschetz hyperplane theorem and the Poincar\'{e} duality, we have $\rH^i(Y^\dag_L,\ol\dQ_\ell)_\fm=0$ if $i\neq 2r-2$ by (1). Thus, (3) follows.

For (2), we consider the weight spectral sequence $\rE^{p,q}$ abutting to $\rH^{p+q}(\cX_{\tilde{L}}\otimes_{O_K}\ol\dQ_p,\ol\dQ_\ell)$, after localization at $\fm$. We write down the first page $\rE^{p,q}_{1,\fm}$ as follows.
\[
\boxed{
\xymatrix{
q\geq 2r+1 & 0 \ar[r] & \rH^q(Y^\bullet_L,\ol\dQ_\ell)_\fm \ar[r] & 0 \\
q=2r & \rH^{2r-2}(Y^\dag_L,\ol\dQ_\ell(-1))_\fm \ar[r]^-{\rd^{-1,2r}_{1,\fm}}
& \rH^{2r}(Y^\bullet_L,\ol\dQ_\ell)_\fm \ar[r] & 0 \\
q=2r-1 & 0\ar[r] & \rH^{2r-1}(Y^\bullet_L,\ol\dQ_\ell)_\fm \ar[r]  & 0 \\
q=2r-2 & 0 \ar[r] & \rH^{2r-2}(Y^\bullet_L,\ol\dQ_\ell)_\fm \ar[r]^-{\rd^{0,2r-2}_{1,\fm}} &
\rH^{2r-2}(Y^\dag_L,\ol\dQ_\ell)_\fm \\
q\leq 2r-3 & 0 \ar[r] & \rH^q(Y^\bullet_L,\ol\dQ_\ell)_\fm \ar[r] & 0 \\
\rE^{p,q}_1 & p=-1 & p=0 & p=1
}
}
\]
By Proposition \ref{pr:tempered_generic}(1), we have $\rH^i(\cX_{\tilde{L}}\otimes_{O_K}\ol\dQ_p,\ol\dQ_\ell)_\fm=0$ for $i\neq 2r-1$, which implies that $\rH^q(Y^\bullet_L,\ol\dQ_\ell)_\fm=0$ for $q\leq 2r-3$ and that $\rd^{0,2r-2}_{1,\fm}$ is injective. The spectral sequence then degenerates at the second page and we have $\IM(\rd^{0,2r-2}_{1,\fm})=\Ker(\rE^{1,2r-2}_{1,\fm}\to\rE^{1,2r-2}_{\infty,\fm})$. Thus, it remains to show that the canonical quotient map $\rE^{1,2r-2}_{1,\fm}\to\rE^{1,2r-2}_{\infty,\fm}$ is an isomorphism. Consider an arbitrary collection $\chi^{\infty,\ul{u}}=\{\chi^v\res v\in\tV_F^\fin\setminus\{\ul{u}\}\}$ in which all but finitely many are trivial and such that $V^{\chi_v}\simeq V_v$ for every $v\in\tV_F^\fin\setminus\{\ul{u}\}$. Put $\pi^{\chi^{\infty,\ul{u}}}\coloneqq\otimes_{v\in\tV_F^\fin\setminus\{\ul{u}\}}\pi^{\chi_v}$. By a similar argument for (1), it suffices to show the following statement:
\begin{enumerate}
  \setcounter{enumi}{3}
  \item The canonical quotient map $\rE^{1,2r-2}_1[\pi^{\chi^{\infty,\ul{u}}}]\to\rE^{1,2r-2}_\infty[\pi^{\chi^{\infty,\ul{u}}}]$ is an isomorphism.
\end{enumerate}

Now we show (4). Without lost of generality, we may replace $K$ by a finite unramified extension in $\ol\dQ_p$ such that $Y$ is a finite disjoint union of $\Spec K$. Define the character $\chi_{\ul{u}}^+$ (resp.\ $\chi_{\ul{u}}^-$) to be the inflation of the trivial (resp.\ nontrivial) character of $A_{M_{\ul{u}}^2}=\mu_2$ along the quotient homomorphism $A_{M_{\ul{u}}}\to A_{M_{\ul{u}}^2}$. Then we have $V^{\chi_{\ul{u}}^+}\simeq\pres{u}{V_{\ul{u}}}$ and $V^{\chi_{\ul{u}}^-}\simeq\pres{\bu}{V_{\ul{u}}}$. If $\rE^{1,2r-2}_1[\pi^{\chi^{\infty,\ul{u}}}]=0$, then we are done. Otherwise, we have
$\rH^{2r-2}(Y^\dag_L,\ol\dQ_\ell)[\pi^{\chi^{\infty,\ul{u}}}]\neq 0$. By (the proof of) \cite{LTXZZ}*{Proposition~5.5.4}, $\pi^{\chi^{\infty,\ul{u}}}$ can be complemented to a cuspidal automorphic representation $\pi'$ of $\pres{u}{H}(\dA_F)$ such that $\pi'_\infty$ is trivial and that $\pi'_u$ is almost unramified with respect to the hyperspecial subgroup $L^\star_{\ul{u}}$. In other words, the collection $\{\chi_v=1\res v\in\tV_F^{(\infty)}\}\cup\chi^{\infty,\ul{u}}\cup\{\chi_{\ul{u}}^+\}$ is coherent, and we have
\begin{align}\label{eq:tempered_almost}
\rE^{1,2r-2}_1[\pi^{\chi^{\infty,\ul{u}}}]=\rH^{2r-2}(Y^\dag_L,\ol\dQ_\ell)[\pi^{\chi^{\infty,\ul{u}}}]\simeq
\(\(\otimes_{v\in\tV_F^\fin\setminus\{\ul{u}\}}\pi^{\chi_v}\)^{L^{\ul{u}}}\)^{\oplus\deg(Y/K)}.
\end{align}
On the other hand, since the representation $\pi^{\chi_{\ul{u}}^-}$ is the only member in the Langlands--Vogan packet of $M_{\ul{u}}$ realized on $\pres{\bu}{V_{\ul{u}}}$ that has nonzero invariants under $L^\tR_{\ul{u}}$, we have an isomorphism
\[
\rH^{2r-1}(\cX_{\tilde{L}}\otimes_{O_K}\ol\dQ_p,\ol\dQ_\ell)[\pi^{\chi^{\infty,\ul{u}}}]
\simeq\(\rho[\tilde\pi^\infty]\res_{\Gal(\ol\dQ_p/K)}\)\otimes
\(\(\otimes_{v\in\tV_F^\fin\setminus\{\ul{u}\}}\pi^{\chi_v}\)^{L^{\ul{u}}}\)^{\oplus\deg(Y/K)}
\]
of representations of $\Gal(\ol\dQ_p/K)$, where $\tilde\pi^\infty=\pi^{\chi^{\infty,\ul{u}}}\otimes\pi^{\chi_{\ul{u}}^-}$. By (a--d) above, it is easy to see that, in the notation of Lemma \ref{le:arthur}(b), we must have $j(\tilde\pi^\infty)=1$, hence that the semisimplification of $\rho[\tilde\pi^\infty]$ is isomorphic to $\rho^\tc_{\Pi_1}$ by Hypothesis \ref{hy:galois}. Thus, $\rho[\tilde\pi^\infty]\res_{\Gal(\ol\dQ_p/K)}$ has nontrivial monodromy, which implies that the dimension of $\rE^{1,2r-2}_\infty[\pi^{\chi^{\infty,\ul{u}}}]$ is at least the dimension of
\[
\(\(\otimes_{v\in\tV_F^\fin\setminus\{\ul{u}\}}\pi^{\chi_v}\)^{L^{\ul{u}}}\)^{\oplus\deg(Y/K)}.
\]
Therefore, (4) follows from \eqref{eq:tempered_almost}. The lemma is proved.
\end{proof}

\begin{proof}[Proof of Proposition \ref{pr:index_inert_almost}]
The proof of Proposition \ref{pr:index_inert_almost} is parallel to that of Proposition \ref{pr:index_inert}. Take elements $\rs_1^u,\rs_2^u\in\dS_{\dQ^\ac}^\tR\setminus\fm_\pi^\tR$ that give $\ell$-tempered $\dQ^\ac$-\'{e}tale correspondences of $\cX_L$, which is possible by Lemma \ref{le:tempered_almost}.

By Lemma \ref{le:base_change} and Definition \ref{de:analytic_kernel}, it suffices to show that for every pair of $p$-basic elements $\phi^\infty_1,\phi^\infty_2\in\sS(V^m\otimes_{\dA_F}\dA_F^\infty)^L$ satisfying that $\supp(\phi^\infty_{1v}\otimes(\phi^\infty_{2v})^\tc)\subseteq(V^{2r}_v)_\reg$ for $v\in\tR'$, and every pair of elements $\rs_1,\rs_2\in(\dS^\tR_{\dQ^\ac})^{\langle\ell\rangle}_{L_\tR}$ that give $\ell$-tempered $\dQ^\ac$-\'{e}tale correspondences of $\cX_L$, we have
\begin{multline}\label{eq:index_inert_6}
\frac{\vol^\natural(L)}{\deg(Y/K)}\Phi^0_\infty(T_1,T_2)
\langle\rs_1^*Z_{T_1}(\phi^\infty_1)'_L,\rs_2^*Z_{T_2}(\phi^\infty_2)'_L\rangle_{X'_L,K}\\
=\sum_{\substack{T^\Box\in\Herm_{2r}^\circ(F)^+\\\Diff(T^\Box,V)=\{\ul{u}\}\\\partial_{r,r} T^\Box=(T_1,T_2)}}\frac{1}{\log q_u}
\(W'_{T^\Box}(0,1_{4r},\CF_{(\Lambda^\tR_{\ul{u}})^{2r}})-\frac{\log q_u}{q_u^r-1}W_{T^\Box}(0,1_{4r},\CF_{(\Lambda^\star_{\ul{u}})^{2r}})\) \\
\times\prod_{v\neq\ul{u}}W_{T^\Box}(0,1_{4r},(\Phi_\infty^0\otimes(\rs_1\phi^\infty_1\otimes(\rs_2\phi^\infty_2)^\tc))_v).
\end{multline}
As $\vol(L_{\ul{u}})=(q_{\ul{u}}+1)(q_{\ul{u}}^{2r}-1)^{-1}$, \eqref{eq:index_inert_6} is equivalent to
\begin{multline}\label{eq:index_inert_7}
\frac{\vol(H(F_\infty)L^{\ul{u}})}{\deg(Y/K)}\Phi^0_\infty(T_1,T_2)
\langle\rs_1^*Z_{T_1}(\phi^\infty_1)'_L,\rs_2^*Z_{T_2}(\phi^\infty_2)'_L\rangle_{X'_L,K}\\
=\sum_{\substack{T^\Box\in\Herm_{2r}^\circ(F)^+\\\Diff(T^\Box,V)=\{\ul{u}\}\\\partial_{r,r} T^\Box=(T_1,T_2)}}
\frac{b_{2r,\ul{u}}(0)}{\log q_u}
\(\frac{q_{\ul{u}}^{2r}-1}{q_{\ul{u}}+1}W'_{T^\Box}(0,1_{4r},\CF_{(\Lambda^\tR_{\ul{u}})^{2r}})
-\frac{\log q_u}{q_{\ul{u}}+1}W_{T^\Box}(0,1_{4r},\CF_{(\Lambda^\star_{\ul{u}})^{2r}})\) \\
\times\prod_{v\neq\ul{u}}
\frac{b_{2r,v}(0)}{\gamma_{V_v,\psi_{F,v}}^{2r}}W_{T^\Box}(0,1_{4r},(\Phi_\infty^0\otimes(\rs_1\phi^\infty_1\otimes(\rs_2\phi^\infty_2)^\tc))_v),
\end{multline}
parallel to \eqref{eq:index_inert_4}.

The proof of \eqref{eq:index_inert_7} is same to that of \eqref{eq:index_inert_4} except that now we have
\[
\chi\(\pi^\wedge_*\pres{\rK}\cN(x)\)=\frac{b_{2r,{\ul{u}}}(0)}{\log q_u}\(\frac{q_{\ul{u}}^{2r}-1}{q_{\ul{u}}+1}W'_{T^\Box}(0,1_{4r},\CF_{(\Lambda^\tR_{\ul{u}})^{2r}})
-\frac{\log q_u}{q_{\ul{u}}+1}W_{T^\Box}(0,1_{4r},\CF_{(\Lambda^\star_{\ul{u}})^{2r}})\)
\]
if $T(x)=T^\Box$, by \cite{LZ}*{Theorem~10.5.1 \& Remark~10.5.4}. The proposition is proved.
\end{proof}

\section{Local indices at archimedean places}
\label{ss:archimedean}

In this section, we compute local indices at places in $\tV_E^{(\infty)}$.

\begin{proposition}\label{pr:index_arch}
Let $\tR$, $\tR'$, $\ell$, and $L$ be as in Definition \ref{de:kernel_geometric}. Let $(\pi,\cV_\pi)$ be as in Assumption \ref{st:representation}. Take an element $u\in\tV_E^{(\infty)}$. Consider an $(\tR,\tR',\ell,L)$-admissible sextuple $(\phi^\infty_1,\phi^\infty_2,\rs_1,\rs_2,g_1,g_2)$ and an element $\varphi_1\in\cV_\pi^{[r]\tR}$. Let $K_1\subseteq G_r(\dA_F^\infty)$ be an open compact subgroup that fixes both $\phi^\infty_1$ and $\varphi_1$, and $\fF_1\subseteq G_r(F_\infty)$ a Siegel fundamental domain for the congruence subgroup $G_r(F)\cap g_1^\infty K_1 (g_1^\infty)^{-1}$. Then for every $T_2\in\Herm_r^\circ(F)^+$, we have
\begin{multline}\label{eq:index_arch}
\vol^\natural(L)\cdot\int_{\fF_1}\varphi^\tc(\tau_1g_1)
\sum_{T_1\in\Herm_r^\circ(F)^+}I_{T_1,T_2}(\phi^\infty_1,\phi^\infty_2,\rs_1,\rs_2,\tau_1g_1,g_2)_{L,u}\rd\tau_1 \\
=\frac{1}{2}\int_{\fF_1}\varphi^\tc(\tau_1g_1)\sum_{T_1\in\Herm_r^\circ(F)^+}
\fE_{T_1,T_2}((\tau_1g_1,g_2),\Phi_\infty^0\otimes(\rs_1\phi^\infty_1\otimes(\rs_2\phi^\infty_2)^\tc))_u\rd\tau_1,
\end{multline}
in which both sides are absolutely convergent. Here, the term $\fE_{T_1,T_2}$ is defined in Definition \ref{de:analytic_kernel} with the Gaussian function $\Phi_\infty^0\in\sS(V^{2r}\otimes_{\dA_F}F_\infty)$ (Notation \ref{st:h}(H3)), and $\vol^\natural(L)$ is defined in Definition \ref{de:measure}.
\end{proposition}

\begin{remark}
The relation between $I_{T_1,T_2}$ and $\fE_{T_1,T_2}$ for each individual pair $(T_1,T_2)$ in the style of Proposition \ref{pr:index_inert} is much more complicated, which involves the so-called holomorphic projection (see \cite{Liu12}*{Section~6A} for the case where $r=1$). The main technical innovation in the archimedean computation in this article is that we do not need to compare $I_{T_1,T_2}$ and $\fE_{T_1,T_2}$ in order to obtain the main theorems; it suffices for us to compare both sides after taking summation and convolution for any of the two variables like in Proposition \ref{pr:index_arch}, which does not require holomorphic projection.
\end{remark}

As we have promised in Section \ref{ss:height}, we start by recalling the definition of the archimedean local index $I_{T_1,T_2}(\phi^\infty_1,\phi^\infty_2,\rs_1,\rs_2,g_1,g_2)_{L,u}$.

Let $X$ be a smooth projective complex scheme of pure dimension $n-1$. For an element $Z\in\rZ^r(X)_\dC$, recall that a \emph{Green current} for $Z$ is an $(r-1,r-1)$-current $\tg_Z$ on $X(\dC)$ that is smooth away from the support of $Z$ and satisfies
\[
\r{d}\r{d}^\tc\tg_Z+\delta_Z=[\omega_Z]
\]
for a unique smooth $(r,r)$-form $\omega_Z$ on $X(\dC)$, which we call the \emph{tail form} of $\tg_Z$. If $Z\in\rZ^r(X)^0_\dC$, then we say that a Green current $\tg_Z$ for $Z$ is \emph{harmonic} if $\omega_Z=0$, and we use $\tg_Z^\heartsuit$ to indicate a harmonic Green current. For two elements $Z_1,Z_2\in\rZ^r(X)^0_\dC$ with disjoint supports, we define
\begin{align}\label{eq:index_arch_0}
\langle Z_1,Z_2\rangle_{X,\dC}\coloneqq\frac{1}{2}\int_{X(\dC)}\tg_{Z_1}^\heartsuit\wedge\delta_{Z_2^\tc},
\end{align}
which is independent of the choice of harmonic Green current $\tg_{Z_1}^\heartsuit$.

By Lemma \ref{le:moduli_archimedean}, we may assume $u=\bu$ without lost of generality in the proof of Proposition \ref{pr:index_arch}. Now we apply the above discussion to the complex scheme $X_L\otimes_E\dC$. For $i=1,2$,
\begin{itemize}
  \item we denote by $\tg_{T_i}^\heartsuit(\phi^\infty_i,\rs_i,g_i^\infty)_L$ a harmonic Green current for $\rs_i^*Z_{T_i}(\omega_r^\infty(g_i^\infty)\phi^\infty_i)_L$ on $X_L\otimes_E\dC$;

  \item for every element $g_i\in G_r(\dA_F)$ with finite part $g_i^\infty$, there is a particular Green current for $\rs_i^*Z_{T_i}(\omega_r^\infty(g_i^\infty)\phi^\infty_i)_L$ on $X_L\otimes_E\dC$, known as the \emph{Kudla--Milson Green current} (see the proof of \cite{Liu11}*{Theorem~4.20}), denoted by $\tg_{T_i}^{\r{KM}}(\phi^\infty_i,\rs_i,g_i)_L$, with the tail form $\omega_{T_i}^{\r{KM}}(\phi^\infty_i,\rs_i,g_i)_L$.
\end{itemize}

\begin{proof}[Proof of Proposition \ref{pr:index_arch}]
As we have pointed out, it suffices to prove the proposition for $u=\bu$. By \eqref{eq:index_arch_0}, we have
\begin{multline}\label{eq:index_arch_1}
I_{T_1,T_2}(\phi^\infty_1,\phi^\infty_2,\rs_1,\rs_2,g_1,g_2)_{L,\bu} \\
=\frac{1}{2}C_{T_1,T_2}(g_{1\infty},g_{2\infty})
\int_{X_L(\dC)}\tg_{T_1}^\heartsuit(\phi^\infty_1,\rs_1,g_1^\infty)_L\wedge
\delta_{(\rs_2^*Z_{T_2}(\omega_r^\infty(g_2^\infty)\phi^\infty_2)_L)^\tc},
\end{multline}
where
\[
C_{T_1,T_2}(g_{1\infty},g_{2\infty})\coloneqq
\omega_{r,\infty}(g_{1\infty})\phi^0_\infty(T_1)\cdot(\omega_{r,\infty}(g_{2\infty})\phi^0_\infty(T_2))^\tc.
\]
We need a variant of \eqref{eq:index_arch_1}. Put
\begin{multline}\label{eq:index_arch_2}
I_{T_1,T_2}(\phi^\infty_1,\phi^\infty_2,\rs_1,\rs_2,g_1,g_2)_{L,\bu}^{\r{KM}}\coloneqq \\
\frac{1}{2}C_{T_1,T_2}(g_{1\infty},g_{2\infty})\(\int_{X_L(\dC)}\tg_{T_1}^{\r{KM}}(\phi^\infty_1,\rs_1,g_1)_L\wedge
\delta_{(\rs_2^*Z_{T_2}(\omega_r^\infty(g_2^\infty)\phi^\infty_2)_L)^\tc} \right. \\
\left.+\int_{X_L(\dC)}\omega_{T_1}^{\r{KM}}(\phi^\infty_1,\rs_1,g_1)_L\wedge
\tg_{T_2}^{\r{KM}}(\phi^\infty_2,\rs_2,g_2)_L^\tc\).
\end{multline}
By \cite{Liu11}*{Theorem~4.20}\footnote{There is a sign error in \cite{Liu11}*{Theorem~4.20}: the correct sign should be $\prod_v\gamma_{\dV_v}^{2n}$, which is $1$, rather than $\prod_v\gamma_{\dV_v}$, which is $-1$ (the root of this sign error is that in the formula for $\omega_\chi(w_r)$ on \cite{Liu11}*{Page~858}, the constant $\gamma_V$ should really be $\gamma_V^r$). This result was later reproved in \cite{GS19}*{Corollary~5.12} by a different method.}, we have
\begin{align}\label{eq:index_arch_7}
\vol^\natural(L)\cdot I_{T_1,T_2}(\phi^\infty_1,\phi^\infty_2,\rs_1,\rs_2,g_1,g_2)_{L,\bu}^{\r{KM}}
=\frac{1}{2}\fE_{T_1,T_2}((g_1,g_2),\Phi_\infty^0\otimes(\rs_1\phi^\infty_1\otimes(\rs_2\phi^\infty_2)^\tc))_\bu.
\end{align}

We first check the absolute convergence of the two sides of \eqref{eq:index_arch}. It is clear that the assignment
\[
\tau_1\mapsto
\sum_{T_1\in\Herm_r^\circ(F)^+}
\left|\fE_{T_1,T_2}((\tau_1g_1,g_2),\Phi_\infty^0\otimes(\rs_1\phi^\infty_1\otimes\rs_2\phi^\infty_2))_\bu\right|
\]
is slowly increasing on $\fF_1$, which implies that the right-hand side of \eqref{eq:index_arch} is absolutely convergent since $\varphi$ is a cusp form.

For the left-hand side, by \eqref{eq:index_arch_7}, it suffices to show that the expression
\[
\sum_{T_1\in\Herm_r^\circ(F)^+}\left|I_{T_1,T_2}(\phi^\infty_1,\phi^\infty_2,\rs_1,\rs_2,\tau_1 g_1,g_2)_{L,\bu}^{\r{KM}}
-I_{T_1,T_2}(\phi^\infty_1,\phi^\infty_2,\rs_1,\rs_2,\tau_1g_1,g_2)_{L,\bu}\right|
\]
is absolutely convergent and is slowly increasing on $\tau_1$. First, since $\rs_2^*Z_{T_2}(\omega_r^\infty(g_2^\infty)\phi^\infty_2)_L$ is cohomologically trivial, by the $\partial\bar\partial$-lemma for currents, there is an $(r-1,r-1)$-current $\eta$ on $X_L(\dC)$ such that
\begin{align}\label{eq:index_arch_3}
\r{d}\r{d}^\tc\eta=\delta_{(\rs_2^*Z_{T_2}(\omega_r^\infty(g_2^\infty)\phi^\infty_2)_L)^\tc}.
\end{align}
Then for every $T_1\in\Herm_r^\circ(F)^+$,
\begin{align*}
J_{T_1}(\tau_1)&\coloneqq I_{T_1,T_2}(\phi^\infty_1,\phi^\infty_2,\rs_1,\rs_2,\tau_1g_1,g_2)_{L,\bu}^{\r{KM}}
-I_{T_1,T_2}(\phi^\infty_1,\phi^\infty_2,\rs_1,\rs_2,\tau_1g_1,g_2)_{L,\bu} \\
&=C_{T_1,T_2}(\tau_1g_{1\infty},g_{2\infty})\int_{X_L(\dC)}\(\tg_{T_1}^{\r{KM}}(\phi^\infty_1,\rs_1,\tau_1g_1)_L
-\tg_{T_1}^\heartsuit(\phi^\infty_1,\rs_1,\tau_1g_1)_L\)\wedge\r{d}\r{d}^\tc\eta \\
&\qquad+C_{T_1,T_2}(\tau_1g_{1\infty},g_{2\infty})\int_{X_L(\dC)}
\omega_{T_1}^{\r{KM}}(\phi^\infty_1,\rs_1,\tau_1g_1)_L\wedge\tg_{T_2}^{\r{KM}}(\phi^\infty_2,\rs_2,g_2)_L^\tc \\
&=C_{T_1,T_2}(\tau_1g_{1\infty},g_{2\infty})\int_{X_L(\dC)}\r{d}\r{d}^\tc\(\tg_{T_1}^{\r{KM}}(\phi^\infty_1,\rs_1,\tau_1g_1)_L
-\tg_{T_1}^\heartsuit(\phi^\infty_1,\rs_1,\tau_1g_1)_L\)\wedge\eta \\
&\qquad+C_{T_1,T_2}(\tau_1g_{1\infty},g_{2\infty})\int_{X_L(\dC)}
\omega_{T_1}^{\r{KM}}(\phi^\infty_1,\rs_1,\tau_1g_1)_L\wedge\tg_{T_2}^{\r{KM}}(\phi^\infty_2,\rs_2,g_2)_L^\tc \\
&=C_{T_1,T_2}(\tau_1g_{1\infty},g_{2\infty})\int_{X_L(\dC)}
\omega_{T_1}^{\r{KM}}(\phi^\infty_1,\rs_1,\tau_1g_1)_L\wedge\(\eta+\tg_{T_2}^{\r{KM}}(\phi^\infty_2,\rs_2,g_2)_L^\tc\).
\end{align*}
By the claim $(*)$ below, we know that $\sum_{T_1\in\Herm_r^\circ(F)^+}|J_{T_1}(\tau_1)|$ is convergent and is slowly increasing in $\tau_1\in\fF_1$, which implies that the left-hand side of \eqref{eq:index_arch} is absolutely convergent as we have pointed out.

We claim that
\begin{itemize}
  \item[$(*)$] The summation
      \[
      \sum_{T_1\in\Herm_r^\circ(F)^+}C_{T_1,T_2}(\tau_1g_{1\infty},g_{2\infty})\cdot\omega_{T_1}^{\r{KM}}(\phi^\infty_1,\rs_1,\tau_1g_1)_L
      \]
      is convergent in the space of smooth $(r,r)$-form on $X_L(\dC)$ with respect to the $C^\infty$-topology, and is locally uniformly slowly increasing in $\tau_1\in\fF_1$.
\end{itemize}

To prove the claim, note that the $C^\infty$-topology is a Fr\'{e}chet topology defined by a natural family of semi-norms. Take such a semi-norm $\|\;\|$. By the construction of the Kudla--Millson form (\cite{Mil85}*{Section~III.1} or \cite{KM86}*{Section~3}), it suffices to consider semi-norms $\|\;\|$ satisfying that there exists $\phi_\infty\in\sS(\pres{\bu}V^r\otimes_FF_\infty)$ such that
\[
\left\|\omega_{r,\infty}(\tau_1 g_{1\infty})\phi^0_\infty(T_1)\cdot
\omega_{T_1}^{\r{KM}}(\phi^\infty_1,\rs_1,\tau_1g_1)_L\right\|=\sup_{h\in\pres{\bu}{H}(F)\backslash\pres{\bu}{H}(\dA_F)}
\left\{\left|\theta_{\phi_\infty\otimes\rs_1\phi^\infty_1,T_1}(\tau_1 g_1,h)\right|\right\},
\]
for every $\tau_1\in\fF_1$ and every $T_1\in\Herm_r^\circ(F)^+$, where $\theta_{\phi_\infty\otimes\rs_1\phi^\infty_1,T_1}$ denotes the $T_1$-component of the classical theta function of $\phi_\infty\otimes\rs_1\phi^\infty_1$. Now since $\pres{\bu}{H}(F)\backslash\pres{\bu}{H}(\dA_F)$ is compact, and the assignment
\[
\tau_1\mapsto\sum_{T_1\in\Herm_r^\circ(F)^+}\left|\theta_{\phi_\infty\otimes\rs_1\phi^\infty_1,T_1}(\tau_1 g_1,h)\right|
\]
is slowly increasing on $\fF_1$, locally uniformly in $h$, the claim follows.

Now we continue to prove \eqref{eq:index_arch}. By \eqref{eq:index_arch_7}, it suffices to show that
\begin{multline}\label{eq:index_arch_4}
\int_{\fF_1}\varphi^\tc(\tau_1g_1)
\sum_{T_1\in\Herm_r^\circ(F)^+}I_{T_1,T_2}(\phi^\infty_1,\phi^\infty_2,\rs_1,\rs_2,\tau_1g_1,g_2)_{L,\bu}^{\r{KM}}\rd\tau_1 \\
=\int_{\fF_1}\varphi^\tc(\tau_1g_1)
\sum_{T_1\in\Herm_r^\circ(F)^+}I_{T_1,T_2}(\phi^\infty_1,\phi^\infty_2,\rs_1,\rs_2,\tau_1g_1,g_2)_{L,\bu}\rd\tau_1,
\end{multline}
in which we have already known that both sides are absolutely convergent.

Take an element $T_1\in\Herm_r^\circ(F)^+$. We put
\[
\pres{\varphi}\tg_{T_1}^\heartsuit\coloneqq\int_{\fF_1}\varphi^\tc(\tau_1g_1)
C_{T_1,T_2}(\tau_1g_{1\infty},g_{2\infty})\tg_{T_1}^\heartsuit(\phi^\infty_1,\rs_1,g_1^\infty)_L\rd\tau_1,
\]
which is a harmonic Green current for
\[
\pres{\varphi}Z_{T_1}\coloneqq\(\int_{\fF_1}\varphi^\tc(\tau_1g_1)C_{T_1,T_2}(\tau_1g_{1\infty},g_{2\infty})\rd\tau_1\)
\cdot\rs_1^*Z_{T_1}(\omega_r^\infty(g_1^\infty)\phi^\infty_1)_L.
\]
We also put
\[
\pres{\varphi}\tg_{T_1}^{\r{KM}}\coloneqq\int_{\fF_1}\varphi^\tc(\tau_1g_1)
C_{T_1,T_2}(\tau_1g_{1\infty},g_{2\infty})\tg_{T_1}^{\r{KM}}(\phi^\infty_1,\rs_1,\tau_1g_1)_L\rd\tau_1,
\]
which is a Green current for $\pres{\varphi}Z_{T_1}$, whose tail form is
\[
\pres{\varphi}\omega_{T_1}^{\r{KM}}\coloneqq\int_{\fF_1}\varphi^\tc(\tau_1g_1)
C_{T_1,T_2}(\tau_1g_{1\infty},g_{2\infty})\omega_{T_1}^{\r{KM}}(\phi^\infty_1,\rs_1,\tau_1g_1)_L\rd\tau_1.
\]
Then by \eqref{eq:index_arch_1}, \eqref{eq:index_arch_2}, and \eqref{eq:index_arch_3}, we have
\begin{align*}
\pres{\varphi}J_{T_1}&\coloneqq
\int_{\fF_1}\varphi^\tc(\tau_1g_1)I_{T_1,T_2}(\phi^\infty_1,\phi^\infty_2,\rs_1,\rs_2,\tau_1g_1,g_2)_{L,\bu}^{\r{KM}}\rd\tau_1  \\
&\qquad -\int_{\fF_1}\varphi^\tc(\tau_1g_1)I_{T_1,T_2}(\phi^\infty_1,\phi^\infty_2,\rs_1,\rs_2,\tau_1g_1,g_2)_{L,\bu}\rd\tau_1 \\
&=\int_{X_L(\dC)}(\pres{\varphi}\tg_{T_1}^{\r{KM}}-\pres{\varphi}\tg_{T_1}^\heartsuit)\wedge\r{d}\r{d}^\tc\eta
+\int_{X_L(\dC)}\pres{\varphi}\omega_{T_1}^{\r{KM}}\wedge\tg_{T_2}^{\r{KM}}(\phi^\infty_2,\rs_2,g_2)_L^\tc \\
&=\int_{X_L(\dC)}\r{d}\r{d}^\tc(\pres{\varphi}\tg_{T_1}^{\r{KM}}-\pres{\varphi}\tg_{T_1}^\heartsuit)\wedge\eta
+\int_{X_L(\dC)}\pres{\varphi}\omega_{T_1}^{\r{KM}}\wedge\tg_{T_2}^{\r{KM}}(\phi^\infty_2,\rs_2,g_2)_L^\tc \\
&=\int_{X_L(\dC)}\pres{\varphi}\omega_{T_1}^{\r{KM}}\wedge\(\eta+\tg_{T_2}^{\r{KM}}(\phi^\infty_2,\rs_2,g_2)_L^\tc\).
\end{align*}
Therefore, the difference between the two sides of \eqref{eq:index_arch_4} equals
\begin{align}\label{eq:index_arch_6}
\sum_{T_1\in\Herm_r^\circ(F)^+}\pres{\varphi}J_{T_1}=
\int_{X_L(\dC)}\(\sum_{T_1\in\Herm_r^\circ(F)^+}\pres{\varphi}\omega_{T_1}^{\r{KM}}\)
\wedge\(\eta+\tg_{T_2}^{\r{KM}}(\phi^\infty_2,\rs_2,g_2)_L^\tc\).
\end{align}
Here, to validate the exchange of summation and integration, it suffices to show that the summation $\sum_{T_1\in\Herm_r^\circ(F)^+}\pres{\varphi}\omega_{T_1}^{\r{KM}}$ is convergent in the space of smooth $(r,r)$-form on $X_L(\dC)$ with respect to the $C^\infty$-topology, since $X_L(\dC)$ is compact. However, this follows from the claim $(*)$.

We then continue by computing the right-hand side of \eqref{eq:index_arch_6}. Since $\supp(\phi_{1v}^\infty)\subseteq(V^r_v)_\reg$ for some $v\in\tR'$, we have
\[
\sum_{T_1\in\Herm_r^\circ(F)^+}\pres{\varphi}\omega_{T_1}^{\r{KM}}=
\sum_{T_1\in\Herm_r(F)^+}\pres{\varphi}\omega_{T_1}^{\r{KM}},
\]
where $\pres{\varphi}\omega_{T_1}^{\r{KM}}$ for $T_1\in\Herm_r(F)^+\setminus\Herm_r^\circ(F)^+$ is defined similarly. However,
\[
\sum_{T_1\in\Herm_r(F)^+}\pres{\varphi}\omega_{T_1}^{\r{KM}}=
(\omega_{r,\infty}(g_{2\infty})\phi^0_\infty(T_2))^\tc\cdot\int_{\fF_1}\varphi^\tc(\tau_1g_1)\omega^{\r{KM}}(\tau_1g_1)\rd\tau_1,
\]
where $\omega^{\r{KM}}(g_1)$ is the Kudla--Milson form for the generating function $Z_{\rs_1\phi_1^\infty}(g_1)_L$. By
\cite{Mil85}*{Theorem~III.2.1}, we know that
\begin{align}\label{eq:index_arch_5}
\int_{\fF_1}\varphi^\tc(\tau_1g_1)\omega^{\r{KM}}(\tau_1g_1)\rd\tau_1=
\int_{\Gamma_1\backslash G_r(F_\infty)}\varphi^\tc(g'_1g_1)\omega^{\r{KM}}(g'_1g_1)\rd g'_1
\end{align}
is a harmonic $(r,r)$-form on $X_L(\dC)$. Since $Z_{\rs_1\phi_1^\infty}(g_1)_L$ is cohomologically trivial, the cohomology class of \eqref{eq:index_arch_5} is also trivial, which implies that \eqref{eq:index_arch_5} vanishes. Therefore, we obtain \eqref{eq:index_arch_4}. The proposition is proved.
\end{proof}

\section{Proof of main results}
\label{ss:main}

In this section, we prove our main results in Section \ref{ss:introduction}. Thus, we put ourselves in Assumption \ref{st:main}. In particular, we have $\tV_F^\ram=\emptyset$, $\tV_F^{(2)}\subseteq\tV_F^\spl$.

Let $(\pi,\cV_\pi)$ be as in Assumption \ref{st:main} with $|\tS_\pi|$ odd, for which we assume Hypothesis \ref{hy:galois}. Take
\begin{itemize}
  \item a totally positive definite hermitian space $V$ over $\dA_E$ of rank $2r$ as in Notation \ref{st:h} satisfying that $\epsilon(V_v)=-1$ if and only if $v\in\tS_\pi$ (so that $V$ is incoherent and $V=V_\pi$ as in Section \ref{ss:introduction}),

  \item $\tS=\tS_\pi$ (so that every underlying rational prime of $\tS$ is unramified in $E$),

  \item $\tR$ a finite subset of $\tV_F^\spl$ containing $\tR_\pi$ and of cardinality at least $2$, and $\tR'=\tR$,

  \item an $\tR$-good rational prime $\ell$ (Definition \ref{de:good}),

  \item for $i=1,2$, a nonzero element $\varphi_i=\otimes_v\varphi_{iv}\in\cV_\pi^{[r]\tR}$ satisfying that $\langle\varphi_{1v}^\tc,\varphi_{2v}\rangle_{\pi_v}=1$ for $v\in\tV_F\setminus\tR$,

  \item for $i=1,2$, an element $\phi^\infty_i=\otimes_v\phi^\infty_{iv}\in\otimes_v\phi^\infty_{iv}\in\sS(V^r\otimes_{\dA_F}\dA_F^\infty)$ satisfying
      \begin{itemize}
        \item $\phi^\infty_{iv}=\CF_{(\Lambda^\tR_v)^r}$ for $v\in\tV_F^\fin\setminus\tR$;

        \item $\supp(\phi^\infty_{1v}\otimes(\phi^\infty_{2v})^\tc)\subseteq(V^{2r}_v)_\reg$ for $v\in\tR$,
      \end{itemize}

  \item an open compact subgroup $L$ of $H(\dA_F^\infty)$ of the form $L_\tR L^\tR$ where $L^\tR$ is defined in Notation \ref{st:h}(H8), that fixes both $\phi^\infty_1$ and $\phi^\infty_2$,

  \item an open compact subgroup $K\subseteq G_r(\dA_F^\infty)$ that fixes $\varphi_1,\varphi_2,\phi^\infty_1,\phi^\infty_2$,

  \item a set of representatives $\{g^{(1)},\dots,g^{(s)}\}$ of the double coset $G_r(F)\backslash G_r(\dA_F^\infty)/K$ satisfying $g^{(j)}\in G_r(\dA_F^{\infty,\tR})$ for $1\leq j\leq s$, together with a Siegel fundamental domain $\fF^{(j)}\subseteq G_r(F_\infty)$ for the congruence subgroup $G_r(F)\cap g^{(j)}K(g^{(j)})^{-1}$ for each $1\leq j\leq s$,

  \item for $i=1,2$, $\rs_i$ a product of two elements in $(\dS^\tR_{\dQ^\ac})^{\langle\ell\rangle}_{L_\tR}$ satisfying $\chi^\tR_\pi(\rs_i)=1$ (which is possible by Proposition \ref{pr:tempered_generic}(2)),

  \item for $i=1,2$, an element $\rs^u_i\in (\dS^\tR_{\dQ^\ac})^{\langle\ell\rangle}_{L_\tR}$ for every $u\in\tV_E^\spl\cup\tS_E$, where $\tS_E$ denotes the subset of $\tV_E^\inert$ above $\tS$, as in Proposition \ref{pr:index_split} and Proposition \ref{pr:index_inert_almost}, satisfying $\chi^\tR_\pi(\rs_i^u)=1$ and that $\rs_i^u=1$ for all but finitely many $u$.
\end{itemize}
In what follows, we put $\tilde\rs_i\coloneqq\rs_i\cdot\prod_{u\in\tV_E^\spl\cup\tS_E}\rs_i^u$ for $i=1,2$.

\begin{lem}\label{le:proof}
Let the situation be as above.
\begin{enumerate}
  \item For every $T_2\in\Herm_r^\circ(F)^+$ and every $\rt\in\dT^\tR_{\dQ^\ac}$, the identity
     \begin{multline*}
     \vol^\natural(L)\sum_{j=1}^s\int_{\fF^{(j)}}\varphi_1^\tc(\tau^{(j)}g^{(j)})\sum_{T_1\in\Herm_r^\circ(F)^+}
     I_{T_1,T_2}(\rt\phi^\infty_1,\phi^\infty_2,\tilde\rs_1,\tilde\rs_2,\tau^{(j)}g^{(j)},g_2)^\ell_L\rd\tau^{(j)} \\
     =\chi^\tR_\pi(\rt)^\tc\int_{G_r(F)\backslash G_r(\dA_F)}
     \varphi_1^\tc(g_1)E'(0,(g_1,g_2),\Phi_\infty^0\otimes\Phi^\infty)_{-,T_2}\rd g_1 \\
     -\chi^\tR_\pi(\rt)^\tc\sum_{u\in\tS_E}\frac{\log q_u}{q_u^r-1}\int_{G_r(F)\backslash G_r(\dA_F)}
     \varphi_1^\tc(g_1)E(0,(g_1,g_2),
     \Phi_\infty^0\otimes\Phi^{\infty,\ul{u}}\otimes\CF_{(\Lambda^\star_{\ul{u}})^{2r}})_{-,T_2}\rd g_1
     \end{multline*}
     holds for every $g_2\in G_r(\dA_F^\tR)$, where $\Phi^\infty\coloneqq\tilde\rs_1\phi^\infty_1\otimes(\tilde\rs_2\phi^\infty_2)^\tc$; $\Lambda^\star_{\ul{u}}$ is the lattice in Proposition \ref{pr:index_inert_almost}; and $E(s,(g_1,g_2),\Phi)_{-,T_2}$ denotes the $T_2$-Siegel Fourier coefficient of the Eisenstein series $E(s,(g_1,g_2),\Phi)$ with respect to the second variable $g_2$.

  \item The identity
     \begin{multline*}
     \vol^\natural(L)
     \sum_{j_2=1}^s\sum_{j_1=1}^s\int_{\fF^{(j_2)}}\int_{\fF^{(j_1)}}\varphi_2(\tau^{(j_2)}g^{(j_2)})\varphi_1^\tc(\tau^{(j_1)}g^{(j_1)})  \\
     \sum_{T_2\in\Herm_r^\circ(F)^+}\sum_{T_1\in\Herm_r^\circ(F)^+}
     I_{T_1,T_2}(\phi^\infty_1,\phi^\infty_2,\tilde\rs_1,\tilde\rs_2,\tau^{(j_1)}g^{(j_1)},\tau^{(j_2)}g^{(j_2)})^\ell_L
     \rd\tau^{(j_1)}\rd\tau^{(j_2)} \\
     =\frac{L'(\tfrac{1}{2},\pi)}{b_{2r}(0)}\cdot C_r^{[F:\dQ]}
     \cdot\prod_{v\in\tV_F^\fin}\fZ^\natural_{\pi_v,V_v}(\varphi^\tc_{1v},\varphi_{2v},\phi^\infty_{1v}\otimes(\phi^\infty_{2v})^\tc)
     \end{multline*}
     holds.
\end{enumerate}

\end{lem}

\begin{proof}
For (1), pick an element $\rh\in\cH^\tR_{W_r}$ such that $\theta^\tR(\rh)=\rt$ as in Definition \ref{de:hecke}. Then there exist finitely many pairs $(c_k,h_k)\in\dC\times G_r(\dA_F^{\infty,\tR})$ such that $\rh\phi^\infty_1=\sum_kc_k\omega_r^\infty(h_k)\phi^\infty_1$ and $\rh\varphi_1=\sum_kc_k\pi(h_k)\varphi_1$. By \cite{Liu20}*{Theorem~1.1} for inert places and \cite{Liu11}*{Proposition~A.5} for split places (see also \cite{Ral82}*{Page~511}), we have
\[
\rt\phi^\infty_1=\rh\phi^\infty_1=\sum_kc_k\omega_r^\infty(h_k)\phi^\infty_1.
\]
Thus, we have
\begin{align}\label{eq:proof_1}
I_{T_1,T_2}(\rt\phi^\infty_1,\phi^\infty_2,\tilde\rs_1,\tilde\rs_2,\tau^{(j)}g^{(j)},g_2)^\ell_L
=\sum_k c_kI_{T_1,T_2}(\rt\phi^\infty_1,\phi^\infty_2,\tilde\rs_1,\tilde\rs_2,\tau^{(j)}g^{(j)}h_k,g_2)^\ell_L.
\end{align}
By Lemma \ref{le:disjoint}, we have
\begin{multline}\label{eq:proof_2}
I_{T_1,T_2}(\phi^\infty_1,\phi^\infty_2,\tilde\rs_1,\tilde\rs_2,\tau^{(j)}g^{(j)}h_k,g_2)^\ell_L
=\sum_{u\in\tV_E^{(\infty)}}2I_{T_1,T_2}(\phi^\infty_1,\phi^\infty_2,\tilde\rs_1,\tilde\rs_2,\tau^{(j)}g^{(j)}h_k,g_2)_{L,u} \\
+\sum_{u\in\tV_E^\fin}\log q_u\cdot I_{T_1,T_2}(\phi^\infty_1,\phi^\infty_2,\tilde\rs_1,\tilde\rs_2,\tau^{(j)}g^{(j)}h_k,g_2)^\ell_{L,u}.
\end{multline}
Combining \eqref{eq:proof_1}, \eqref{eq:proof_2}, Proposition \ref{pr:index_split}, Proposition \ref{pr:index_inert}, Proposition \ref{pr:index_inert_almost}, and Proposition \ref{pr:index_arch}, we have
\begin{multline}\label{eq:proof_3}
\vol^\natural(L)\sum_{j=1}^s\int_{\fF^{(j)}}\varphi_1^\tc(\tau^{(j)}g^{(j)})\sum_{T_1\in\Herm_r^\circ(F)^+}
I_{T_1,T_2}(\rt\phi^\infty_1,\phi^\infty_2,\tilde\rs_1,\tilde\rs_2,\tau^{(j)}g^{(j)},g_2)^\ell_L\rd\tau^{(j)} \\
=\sum_{k}\sum_{j=1}^sc_k\int_{\fF^{(j)}}\varphi_1^\tc(\tau^{(j)}g^{(j)})\sum_{T_1\in\Herm_r^\circ(F)^+}
\fE_{T_1,T_2}^\tS((\tau^{(j)}g^{(j)}h_k,g_2),\Phi_\infty^0\otimes\Phi^\infty),
\end{multline}
where we put
\begin{multline*}
\fE_{T_1,T_2}^\tS((g_1,g_2),\Phi_\infty^0\otimes\Phi^\infty)\coloneqq \\
\sum_{u\in\tV_E\setminus\tV_E^\spl}\fE_{T_1,T_2}((g_1,g_2),\Phi_\infty^0\otimes\Phi^\infty)_u
-\sum_{u\in\tS_E}\frac{\log q_u}{q_u^r-1}E_{T_1,T_2}((g_1,g_2),\Phi_\infty^0\otimes
\Phi^{\infty,\ul{u}}\otimes\CF_{(\Lambda^\star_{\ul{u}})^{2r}}).
\end{multline*}
By Proposition \ref{pr:analytic_kernel} and Remark \ref{re:hermitian}, we have
\begin{multline*}
\sum_{T_1\in\Herm_r^\circ(F)^+}\fE_{T_1,T_2}^\tS((\tau^{(j)}g^{(j)}h_k,g_2),\Phi_\infty^0\otimes\Phi^\infty) =E'(0,(\tau^{(j)}g^{(j)}h_k,g_2),\Phi_\infty^0\otimes\Phi^\infty)_{-,T_2} \\
-\sum_{u\in\tS_E}\frac{\log q_u}{q_u^r-1}E(0,(\tau^{(j)}g^{(j)}h_k,g_2),
\Phi_\infty^0\otimes\Phi^{\infty,\ul{u}}\otimes\CF_{(\Lambda^\star_{\ul{u}})^{2r}})_{-,T_2}
\end{multline*}
for every $1\leq j\leq s$ and every $k$. It follows that
\begin{multline*}
\eqref{eq:proof_3}=\sum_{k}c_k\int_{G_r(F)\backslash G_r(\dA_F)}\varphi_1^\tc(g_1)
E'(0,(g_1h_k,g_2),\Phi_\infty^0\otimes\Phi^\infty)_{-,T_2}\rd g_1 \\
-\sum_{u\in\tS_E}\frac{\log q_u}{q_u^r-1}\sum_{k}c_k\int_{G_r(F)\backslash G_r(\dA_F)}\varphi_1^\tc(g_1)
E(0,(g_1h_k,g_2),\Phi_\infty^0\otimes\Phi^{\infty,\ul{u}}\otimes\CF_{(\Lambda^\star_{\ul{u}})^{2r}})_{-,T_2}
\rd g_1 \\
=\int_{G_r(F)\backslash G_r(\dA_F)}(\rh\varphi_1^\tc)(g_1)
E'(0,(g_1,g_2),\Phi_\infty^0\otimes\Phi^\infty)_{-,T_2}\rd g_1 \\
-\sum_{u\in\tS_E}\frac{\log q_u}{q_u^r-1}\int_{G_r(F)\backslash G_r(\dA_F)}(\rh\varphi_1^\tc)(g_1)
E(0,(g_1,g_2),\Phi_\infty^0\otimes\Phi^{\infty,\ul{u}}\otimes\CF_{(\Lambda^\star_{\ul{u}})^{2r}})_{-,T_2}\rd g_1.
\end{multline*}
Part (1) follows as $\rh\varphi_1^\tc=\chi^\tR_\pi(\rt)^\tc\cdot\varphi_1^\tc$.

For (2), we apply (1) to $\rt=1$ and $g_2=\tau^{(j_2)}g^{(j_2)}$ for $1\leq j_2\leq s$ hence obtain
\begin{multline}\label{eq:proof_4}
\vol^\natural(L)
\sum_{j_2=1}^s\sum_{j_1=1}^s\int_{\fF^{(j_2)}}\int_{\fF^{(j_1)}}\varphi_2(\tau^{(j_2)}g^{(j_2)})\varphi_1^\tc(\tau^{(j_1)}g^{(j_1)})  \\
\sum_{T_2\in\Herm_r^\circ(F)^+}\sum_{T_1\in\Herm_r^\circ(F)^+}
I_{T_1,T_2}(\phi^\infty_1,\phi^\infty_2,\tilde\rs_1,\tilde\rs_2,\tau^{(j_1)}g^{(j_1)},\tau^{(j_2)}g^{(j_2)})^\ell_L
rd\tau^{(j_1)}\rd\tau^{(j_2)} \\
=\iint_{[G_r(F)\backslash G_r(\dA_F)]^2}\varphi_2(g_2)\varphi_1^\tc(g_1)
E'(0,(g_1,g_2),\Phi_\infty^0\otimes\Phi^\infty)\rd g_1\rd g_2 \\
-\sum_{u\in\tS_E}\frac{\log q_u}{q_u^r-1}
\iint_{[G_r(F)\backslash G_r(\dA_F)]^2}\varphi_2(g_2)\varphi_1^\tc(g_1)E(0,(g_1,g_2),
\Phi_\infty^0\otimes\Phi^{\infty,\ul{u}}\otimes\CF_{(\Lambda^\star_{\ul{u}})^{2r}})\rd g_1\rd g_2.
\end{multline}
By the classical Rallis inner product formula (see, for example, \cite{Liu11}*{(2-6)}) and Proposition \ref{pr:uniqueness}(2), we have
\[
\int_{G_r(F)\backslash G_r(\dA_F)}\int_{G_r(F)\backslash G_r(\dA_F)}\varphi_2(g_2)\varphi_1^\tc(g_1)E(0,(g_1,g_2),
\Phi_\infty^0\otimes\Phi^{\infty,\ul{u}}\otimes\CF_{(\Lambda^\star_{\ul{u}})^{2r}})\rd g_1\rd g_2=0
\]
for every $u\in\tS_E$. Together with $\chi^\tR_\pi(\tilde\rs_1)=\chi^\tR_\pi(\tilde\rs_2)=1$, we have
\begin{align}\label{eq:proof_5}
\eqref{eq:proof_4}=
\int_{G_r(F)\backslash G_r(\dA_F)}\int_{G_r(F)\backslash G_r(\dA_F)}\varphi_2(g_2)\varphi_1^\tc(g_1)
E'(0,(g_1,g_2),\Phi_\infty^0\otimes(\phi^\infty_1\otimes(\phi^\infty_2)^\tc))\rd g_1\rd g_2.
\end{align}
By Proposition \ref{pr:eisenstein}, we have
\begin{align*}
\eqref{eq:proof_5}=\frac{L'(\tfrac{1}{2},\pi)}{b_{2r}(0)}\cdot C_r^{[F:\dQ]}
\cdot\prod_{v\in\tV_F^\fin}\fZ^\natural_{\pi_v,V_v}(\varphi^\tc_{1v},\varphi_{2v},\phi^\infty_{1v}\otimes(\phi^\infty_{2v})^\tc).
\end{align*}
Part (2) is proved.
\end{proof}

\begin{proof}[Proof of Theorem \ref{th:main}]
First, it suffices to prove the theorem for $\tR$ satisfying $\tR_\pi\subseteq\tR\subseteq\tV_F^\fin$ and $|\tR|\geq 2$. Take an element $w\in\tV_F^{(\infty)}$ and put ourselves in the setup of Section \ref{ss:special}. We prove that the localization of the $\dT^\tR_\dC$-module $\CH^r(X_L)^0_\dC$ at $\fm_\pi^\tR$, is nonvanishing.

Assume the converse. Then for every element $T_2\in\Herm_r^\circ(F)^+$, we can find $\rt_{T_2}\in\dT^\tR_{\dQ^\ac}$ satisfying $\chi^\tR_\pi(\rt_{T_2})=1$ and $\rt_{T_2}^*Z_{T_2}(\omega_r^\infty(g^{(j)})\phi^\infty_2)_L=0$ for every $1\leq j\leq s$. Let $\hat\rt_{T_2}$ be the adjoint of $\rt_{T_2}$. Then we have
\begin{multline}\label{eq:main_1}
\langle\tilde\rs_1^*Z_{T_1}(\omega_r^\infty(g_1^\infty)(\hat\rt_{T_2}\phi^\infty_1))_L,
\tilde\rs_2^*Z_{T_2}(\omega_r^\infty(g^{(j)})\phi^\infty_2)_L
\rangle_{X_L,E}^\ell \\
=\langle\hat\rt_{T_2}^*\tilde\rs_1^*Z_{T_1}(\omega_r^\infty(g_1^\infty)\phi^\infty_1)_L,
\tilde\rs_2^*Z_{T_2}(\omega_r^\infty(g^{(j)})\phi^\infty_2)_L
\rangle_{X_L,E}^\ell \\
=\langle\tilde\rs_1^*Z_{T_1}(\omega_r^\infty(g_1^\infty)\phi^\infty_1)_L,
\rt_{T_2}^*\tilde\rs_2^*Z_{T_2}(\omega_r^\infty(g^{(j)})\phi^\infty_2)_L
\rangle_{X_L,E}^\ell=0
\end{multline}
for every $T_1\in\Herm_r^\circ(F)^+$, $g_1^\infty\in G_r(\dA_F^{\infty,\tR})$, and $1\leq j\leq s$. In particular, we have
\[
I_{T_1,T_2}(\hat\rt_{T_2}\phi^\infty_1,\phi^\infty_2,\tilde\rs_1,\tilde\rs_2,\tau^{(j)}g^{(j)},g_2)^\ell_L=0
\]
for every $g_2\in G_r(\dA_F^\tR)$ with $g_2^\infty\in\{g^{(1)},\dots,g^{(s)}\}$. It follows that
\[
\vol^\natural(L)\sum_{j=1}^s\int_{\fF^{(j)}}\varphi_1^\tc(\tau^{(j)}g^{(j)})\sum_{T_1\in\Herm_r^\circ(F)^+}
I_{T_1,T_2}(\hat\rt_{T_2}\phi^\infty_1,\phi^\infty_2,\tilde\rs_1,\tilde\rs_2,\tau^{(j)}g^{(j)},g_2)^\ell_L\rd\tau^{(j)}=0
\]
for every $g_2\in G_r(\dA_F^\tR)$ with $g_2^\infty\in\{g^{(1)},\dots,g^{(s)}\}$ and every $T_2\in\Herm_r^\circ(F)^+$. Now applying Lemma \ref{le:proof}(1) twice with $\rt=\hat\rt_{T_2}$ and $\rt=1$, respectively, we obtain
\[
\vol^\natural(L)\sum_{j=1}^s\int_{\fF^{(j)}}\varphi_1^\tc(\tau^{(j)}g^{(j)})\sum_{T_1\in\Herm_r^\circ(F)^+}
I_{T_1,T_2}(\phi^\infty_1,\phi^\infty_2,\tilde\rs_1,\tilde\rs_2,\tau^{(j)}g^{(j)},g_2)^\ell_L\rd\tau^{(j)}=0
\]
for every $g_2\in G_r(\dA_F^\tR)$ with $g_2^\infty\in\{g^{(1)},\dots,g^{(s)}\}$ and every $T_2\in\Herm_r^\circ(F)^+$. By Lemma \ref{le:proof}(2), we obtain
\[
\frac{L'(\tfrac{1}{2},\pi)}{b_{2r}(0)}\cdot C_r^{[F:\dQ]}
\cdot\prod_{v\in\tV_F^\fin}\fZ^\natural_{\pi_v,V_v}(\varphi^\tc_{1v},\varphi_{2v},\phi^\infty_{1v}\otimes(\phi^\infty_{2v})^\tc)=0,
\]
that is,
\[
\prod_{v\in\tS}\frac{(-1)^rq_v^{r-1}(q_v+1)}{(q_v^{2r-1}+1)(q_v^{2r}-1)}
\cdot\prod_{v\in\tR}\fZ^\natural_{\pi_v,V_v}(\varphi^\tc_{1v},\varphi_{2v},\phi^\infty_{1v}\otimes(\phi^\infty_{2v})^\tc)=0.
\]
Now by Proposition \ref{pr:regular}, we may choose $\varphi_1,\varphi_2,\phi^\infty_1,\phi^\infty_2$ such that
\[
\fZ^\natural_{\pi_v,V_v}(\varphi^\tc_{1v},\varphi_{2v},\phi^\infty_{1v}\otimes(\phi^\infty_{2v})^\tc)\neq 0
\]
for every $v\in\tR$. As $L'(\tfrac{1}{2},\pi)\neq 0$, we obtain a contradiction. The theorem is proved.
\end{proof}

\begin{proof}[Proof of Theorem \ref{th:aipf} and Corollary \ref{co:aipf}]
By Definition \ref{de:natural} and Proposition \ref{pr:arithmetic_theta}(1), we have
\begin{multline*}
\langle\Theta_{\phi^\infty_1}(\varphi_1),\Theta_{\phi^\infty_2}(\varphi_2)\rangle_{X,E}^\natural
=\vol^\natural(L)
\sum_{j_2=1}^s\sum_{j_1=1}^s\int_{\fF^{(j_2)}}\int_{\fF^{(j_1)}}\varphi_2(\tau^{(j_2)}g^{(j_2)})\varphi_1^\tc(\tau^{(j_1)}g^{(j_1)})  \\
\sum_{T_2\in\Herm_r^\circ(F)^+}\sum_{T_1\in\Herm_r^\circ(F)^+}
I_{T_1,T_2}(\phi^\infty_1,\phi^\infty_2,\tilde\rs_1,\tilde\rs_2,\tau^{(j_1)}g^{(j_1)},\tau^{(j_2)}g^{(j_2)})^\ell_L
\rd\tau^{(j_1)}\rd\tau^{(j_2)}.
\end{multline*}
By Lemma \ref{le:proof}(2) and Proposition \ref{pr:eisenstein}, we obtain
\begin{align}\label{eq:aipf}
\langle\Theta_{\phi^\infty_1}(\varphi_1),\Theta_{\phi^\infty_2}(\varphi_2)\rangle_{X,E}^\natural
=\frac{L'(\tfrac{1}{2},\pi)}{b_{2r}(0)}\cdot C_r^{[F:\dQ]}
\cdot\prod_{v\in\tV_F^\fin}\fZ^\natural_{\pi_v,V_v}(\varphi^\tc_{1v},\varphi_{2v},\phi^\infty_{1v}\otimes(\phi^\infty_{2v})^\tc).
\end{align}
By Proposition \ref{pr:regular}, we may choose $\varphi_1,\varphi_2,\phi^\infty_1,\phi^\infty_2$ such that
\[
\prod_{v\in\tV_F^\fin}\fZ^\natural_{\pi_v,V_v}(\varphi^\tc_{1v},\varphi_{2v},\phi^\infty_{1v}\otimes(\phi^\infty_{2v})^\tc)\neq 0.
\]

Now we claim that \eqref{eq:aipf} holds for arbitrary vectors $\varphi_1,\varphi_2,\phi^\infty_1,\phi^\infty_2$ (not necessarily those in the beginning of this section) as in the statement of Theorem \ref{th:aipf}(1). This is a consequence of Proposition \ref{pr:uniqueness}(1) as both sides of \eqref{eq:aipf} give elements in the space
\[
\bigotimes_{v\in\tV_F^\fin}\Hom_{G_r(F_v)\times G_r(F_v)}(\rI^\Box_{r,v}(0),\pi_v\boxtimes\pi_v^\vee),
\]
which is of dimension one. Thus, Theorem \ref{th:aipf}(1) follows.

By Proposition \ref{pr:arithmetic_theta}(2), the assignment $(\varphi,\phi^\infty)\mapsto\Theta_{\phi^\infty}(\varphi)$ gives an element in
\[
\Hom_{H(\dA_F^\infty)}\(\Hom_{G_r(\dA_F^\infty)}(\sS(V^r\otimes_{\dA_F}\dA_F^\infty),\pi^\infty),\varinjlim_{L}\CH^r(X_L)^0_\dC\),
\]
in which $\Hom_{G_r(\dA_F^\infty)}(\sS(V^r\otimes_{\dA_F}\dA_F^\infty),\pi^\infty)$ is simply the theta lifting of $\pi^\infty$ to $H(\dA_F^\infty)$ by Proposition \ref{pr:uniqueness}(3). Thus, Theorem \ref{th:aipf}(2) is a consequence of \eqref{eq:aipf} and the fact that $\prod_{v\in\tV_F^\fin}\fZ^\natural_{\pi_v,V_v}$ is nontrivial.

Finally, Corollary \ref{co:aipf} is a consequence of \eqref{eq:aipf} and Proposition \ref{pr:eisenstein} (where one may take $\tR=\emptyset$).
\end{proof}

\appendix

\section{Two lemmas in Fourier analysis}
\label{ss:fourier}

In this appendix, we prove two lemmas in Fourier analysis that are only used in the proof of Proposition \ref{pr:regular}. Both the lemmas and their proofs are variants of \cite{AN04}*{Theorem~1} (in the non-archimedean setting).

Let $F$ be a non-archimedean local field of characteristic zero. Denote the maximal ideal of $O_F$ by $\fp_F$ and put $q\coloneqq|O_F/\fp_F|$. We fix a nontrivial additive character $\psi\colon F\to\dC^\times$ that is used to define the Fourier transform.

\begin{lem}\label{le:fourier1}
Consider a finite dimensional $F$-vector space $X$, a nonzero homogeneous polynomial $\Delta$ on $X$, and a real number $r\geq 0$. Let $f$ be a nonzero locally constant function on an open subset $\Omega\subseteq X$ on which $\Delta$ is nonvanishing. Suppose that $f$ is locally integrable on $X$ and satisfies that for every $\varepsilon>0$, we have
\[
\int_\Omega|f(x)|^{2+\varepsilon}|\Delta(x)|_F^{r\varepsilon}\rd x<\infty.
\]
Then the support of the Fourier transform of $f$, as a distribution on $X^\vee$, can not be contained in an analytic hypersurface.
\end{lem}

\begin{proof}
Let $n\geq 1$ be the dimension of $X$. Without lost of generality, we may identify both $X$ and $X^\vee$ with $F^n$, take $\r{d}x$ to be the measure that gives $O_F$ volume $1$, and assume that $\psi_F$ has conductor $O_F$. For every integer $N$, we put $B^n_N\coloneqq(\fp_F^N)^n$, which is an open compact subset of $F^n$.

Let $u$ be the Fourier transform of $f$. For every integer $N\geq 0$, put $\chi_N\coloneqq q^{Nn}\CF_{B^n_N}\in\sS(F^n)$, and put $u_N\coloneqq u\ast\chi_N$, which is a locally constant function on $F^n$. Take two real numbers $0\leq\delta<1$ and $\varepsilon>0$ to be determined later. Let $p\geq 2$ satisfy $\frac{1}{2-\delta}+\frac{1}{p}=1$. Since the Fourier transform is a bounded operator from $L^{2-\delta}(F^n)$ to $L^p(F^n)$, we have
\begin{align*}
\|u_N\|_p^{2-\delta}
&\leq C_\delta\int_{F^n}|f(x)|^{2-\delta}|\widehat{\chi_N}(x)|^{2-\delta}\rd x \\
&= C_\delta\int_{F^n}|f(x)|^{2-\delta}|\CF_{B^n_{-N}}(x)|^{2-\delta}\rd x \\
&= C_\delta\int_{B^n_{-N}}|f(x)|^{2-\delta}\rd x
\end{align*}
for some constant $C_\delta>0$ depending only on $\delta$. By H\"{o}lder's inequality, we have
\begin{align*}
\|u_N\|_p^{2-\delta}\leq C_\delta
\(\int_{B^n_{-N}}|\Delta(x)|_F^{-r\frac{(2-\delta)\varepsilon}{\delta+\varepsilon}}\rd x\)^{\frac{\delta+\varepsilon}{2+\varepsilon}}
\(\int_{B^n_{-N}}|f(x)|^{2+\varepsilon}|\Delta(x)|_F^{r\varepsilon}\rd x\)^{\frac{2-\delta}{2+\varepsilon}}.
\end{align*}
Let $d$ be the degree of $\Delta$. There exists a real number $0<\rho_\Delta<n/d$ depending only on $\Delta$ such that as long as $r\frac{(2-\delta)\varepsilon}{\delta+\varepsilon}<\rho_\Delta$, the function $|\Delta(x)|_F^{-r\frac{(2-\delta)\varepsilon}{\delta+\varepsilon}}$ is locally integrable. In this case, there exists a constant $C_{\delta,\varepsilon}>0$ such that
\[
\int_{B^n_{-N}}|\Delta(x)|_F^{-r\frac{(2-\delta)\varepsilon}{\delta+\varepsilon}}\rd x=C_{\delta,\varepsilon}\cdot
q^{N\(n-dr\frac{(2-\delta)\varepsilon}{\delta+\varepsilon}\)}.
\]
By the integrability condition on $f$, there is a new constant $C'_{\delta,\varepsilon}>0$ depending only on $\delta$ and $\varepsilon$ such that
\begin{align}\label{eq:fourier1}
\|u_N\|_p^{2-\delta}\leq C'_{\delta,\varepsilon}\cdot q^{N\(n-dr\frac{(2-\delta)\varepsilon}{\delta+\varepsilon}\)\frac{\delta+\varepsilon}{2+\varepsilon}}
\end{align}
holds for every $N\geq 0$.

Now suppose that the support of $u$ is contained in an analytic hypersurface $U$. For $N\geq 0$, put $U_N\coloneqq U+B^n_N\subseteq F^n$ as a tubular neighbourhood of $U$, which contains the support of $u_N$. Then for every $g\in\sS(F^n)$, we have
\begin{align}\label{eq:fourier2}
\lim_{N\to\infty}q^N\int_{U_N}g(x)\rd x = \int_{U}g(y)\rd y.
\end{align}
Then by H\"{o}lder's inequality, \eqref{eq:fourier1}, and \eqref{eq:fourier2}, we have
\begin{align*}
|\langle u,g\rangle|^{2-\delta}
&=\lim_{N\to\infty}|\langle u_N,g\rangle|^{2-\delta} \\
&\leq\lim_{N\to\infty}\|u_N\|_p^{2-\delta}\cdot\int_{U_N}|g(x)|^{2-\delta}\rd x \\
&\leq C'_{\delta,\varepsilon}\cdot\lim_{N\to\infty}
q^{N\(\(n-dr\frac{(2-\delta)\varepsilon}{\delta+\varepsilon}\)\frac{\delta+\varepsilon}{2+\varepsilon}-1\)}\cdot
q^N\int_{U_N}|g(x)|^{2-\delta}\rd x \\
&=C'_{\delta,\varepsilon}\cdot\int_{U}|g(y)|^{2-\delta}\rd y\cdot\lim_{N\to\infty}
q^{N\(\(n-dr\frac{(2-\delta)\varepsilon}{\delta+\varepsilon}\)\frac{\delta+\varepsilon}{2+\varepsilon}-1\)}.
\end{align*}
Choose suitable $\delta,\varepsilon$ such that
\[
r\frac{(2-\delta)\varepsilon}{\delta+\varepsilon}<\rho_\Delta,\qquad
n-dr\frac{(2-\delta)\varepsilon}{\delta+\varepsilon}<\frac{2+\varepsilon}{\delta+\varepsilon}.
\]
Then the above limit is zero, that is, $\langle u,g\rangle=0$ for every $g\in\sS(F^n)$. Thus, we have $u=0$. The lemma is proved.
\end{proof}

\begin{lem}\label{le:fourier2}
Consider a finite dimensional $F$-vector space $X$, a nonzero homogeneous polynomial $\Delta$ on $X$, and a real number $r\geq 0$. Denote by $\Omega\subseteq X$ the nonvanishing locus of $\Delta$. Let $f$ be a nonzero locally constant function on $\Omega$ that is locally integrable on $X$, satisfying the following condition: there exists a decomposition $X=X_1\oplus X_2\oplus X_3$ with $\dim_FX_1=\dim_FX_2>0$ such that
\begin{enumerate}
  \item $\Omega$ is disjoint from $X_1\oplus X_3\cup X_2\oplus X_3$;

  \item $|\Delta(\alpha x_1,\alpha^{-1}x_2,x_3)|_F=|\Delta(x_1,x_2,x_3)|_F$ for every $\alpha\in F^\times$ and $x_i\in X_i$;

  \item $|f(\alpha x_1,\alpha^{-1}x_2,x_3)|=|f(x_1,x_2,x_3)|$ for every $\alpha\in F^\times$ and $x_i\in X_i$;

  \item for every $\varepsilon>0$, we have
      \[
      \int_{F^\times\backslash\Omega}|f(x)|^{2+\varepsilon}|\Delta(x)|_F^{r\varepsilon}\rd x<\infty,
      \]
      where the action of $\alpha\in F^\times$ on $\Omega$ is given by $\alpha.(x_1,x_2,x_3)=(\alpha x_1,\alpha^{-1}x_2,x_3)$.
\end{enumerate}
Then the support of the Fourier transform of $f$, as a distribution on $X^\vee$, can not be contained in an analytic hypersurface.
\end{lem}

\begin{proof}
Let $n\geq 1$ be the dimension of $X$. Without lost of generality, we may identify the decomposition $X=X_1\oplus X_2\oplus X_3$ with $F^n=F^m\oplus F^m\oplus F^{n-2m}$, identify $X^\vee$ with $F^n$, take $\r{d}x$ to be the measure that gives $O_F$ volume $1$, and assume that $\psi_F$ has conductor $O_F$. For every integers $N$ and $l\geq 0$, we put $B^l_N\coloneqq(\fp_F^N)^l$ and $A^l_N\coloneqq B^l_N\setminus B^l_{N+1}$, which are open compact subsets of $F^l$. It is clear that the natural map $\varpi^\dZ\times(A^m_0\times F^m\times F^{n-2m})\to F^n$ given by the action in (4) is injective; and by (1) that $\Omega$ is contained in $\varpi^\dZ.(A^m_0\times F^m\times F^{n-2m})$.

Let $u$ be the Fourier transform of $f$. For every integer $N\geq 0$, put $\chi_N\coloneqq q^{Nn}\CF_{B^n_N}\in\sS(F^n)$, and put $u_N\coloneqq u\ast\chi_N$, which is a locally constant function on $F^n$. Take three real numbers $0\leq\delta<\gamma<1$ and $\varepsilon>0$ to be determined later. Let $p>2$ satisfy $\frac{1}{2-\gamma}+\frac{1}{p}=1$. Since the Fourier transform is a bounded operator from $L^{2-\gamma}(F^n)$ to $L^p(F^n)$, we have
\begin{align*}
\|u_N\|_p^{2-\gamma}
&\leq C_\gamma\int_{F^n}|f(x)|^{2-\gamma}|\widehat{\chi_N}(x)|^{2-\gamma}\rd x \\
&= C_\gamma\int_{F^n}|f(x)|^{2-\gamma}|\CF_{B^n_{-N}}(x)|^{2-\gamma}\rd x \\
&= C_\gamma\int_{B^n_{-N}}|f(x)|^{2-\gamma}\rd x
\end{align*}
for some constant $C_\gamma>0$ depending only on $\gamma$. By (3) and H\"{o}lder's inequality, we have
\begin{align*}
&\quad\int_{B^n_{-N}}|f(x)|^{2-\gamma}\rd x  \\
&=\sum_{i=-2N}^\infty (i+2N+1)\int_{A^m_0\times A^m_i\times B^{n-2m}_{-N}}|f(x)|^{2-\gamma}\rd x \\
&\leq\sum_{i=-2N}^\infty (i+2N+1)\(\int_{A^m_0\times A^m_i\times B^{n-2m}_{-N}}\rd x\)^{\frac{\gamma-\delta}{2-\delta}}
\(\int_{A^m_0\times A^m_i\times B^{n-2m}_{-N}}|f(x)|^{2-\delta}\rd x\)^{\frac{2-\gamma}{2-\delta}} \\
&\leq\sum_{i=-2N}^\infty (i+2N+1)(q^{-im}q^{N(n-2m)})^{\frac{\gamma-\delta}{2-\delta}}
\(\int_{A^m_0\times A^m_i\times B^{n-2m}_{-N}}|f(x)|^{2-\delta}\rd x\)^{\frac{2-\gamma}{2-\delta}} \\
&=C_{\gamma,\delta}\cdot q^{Nn\frac{\gamma-\delta}{2-\delta}}\(\int_{A^m_0\times A^m_i\times B^{n-2m}_{-N}}|f(x)|^{2-\delta}\rd x\)^{\frac{2-\gamma}{2-\delta}}
\end{align*}
for some constant $C_{\gamma,\delta}>0$. Together, we obtain
\begin{align}\label{eq:fourier3}
\|u_N\|_p^{2-\delta}\leq C'_{\gamma,\delta}\cdot q^{Nn\frac{\gamma-\delta}{2-\gamma}}\cdot\int_{A^m_0\times A^m_i\times B^{n-2m}_{-N}}|f(x)|^{2-\delta}\rd x
\end{align}
for a new constant $C'_{\gamma,\delta}>0$ depending only on $\gamma$ and $\delta$. By H\"{o}lder's inequality, we have
\begin{align*}
&\quad\int_{A^m_0\times A^m_i\times B^{n-2m}_{-N}}|f(x)|^{2-\delta}\rd x \\
&\leq \(\int_{A^m_0\times A^m_i\times B^{n-2m}_{-N}}|\Delta(x)|_F^{-r\frac{(2-\delta)\varepsilon}{\delta+\varepsilon}}\rd x\)^{\frac{\delta+\varepsilon}{2+\varepsilon}}
\(\int_{A^m_0\times A^m_i\times B^{n-2m}_{-N}}|f(x)|^{2+\varepsilon}|\Delta(x)|_F^{r\varepsilon}\rd x\)^{\frac{2-\delta}{2+\varepsilon}} \\
&\leq \(\int_{B^n_{-2N}}|\Delta(x)|_F^{-r\frac{(2-\delta)\varepsilon}{\delta+\varepsilon}}\rd x\)^{\frac{\delta+\varepsilon}{2+\varepsilon}}
\(\int_{A^m_0\times A^m_i\times B^{n-2m}_{-N}}|f(x)|^{2+\varepsilon}|\Delta(x)|_F^{r\varepsilon}\rd x\)^{\frac{2-\delta}{2+\varepsilon}}.
\end{align*}
Let $d$ be the degree of $\Delta$. There exists a real number $0<\rho_\Delta<n/d$ depending only on $\Delta$ such that as long as $r\frac{(2-\delta)\varepsilon}{\delta+\varepsilon}<\rho_\Delta$, the function $|\Delta(x)|_F^{-r\frac{(2-\delta)\varepsilon}{\delta+\varepsilon}}$ is locally integrable. In this case, there exists a constant $C_{\delta,\varepsilon}>0$ such that
\[
\int_{B^n_{-2N}}|\Delta(x)|_F^{-r\frac{(2-\delta)\varepsilon}{\delta+\varepsilon}}\rd x=C_{\delta,\varepsilon}\cdot
q^{2N\(n-dr\frac{(2-\delta)\varepsilon}{\delta+\varepsilon}\)}.
\]
On the other hand, by (4), we have
\[
\(\int_{A^m_0\times A^m_i\times B^{n-2m}_{-N}}|f(x)|^{2+\varepsilon}|\Delta(x)|_F^{r\varepsilon}\rd x\)^{\frac{2-\delta}{2+\varepsilon}}\leq C'_{\delta,\varepsilon}
\]
for a constant $C'_{\delta,\varepsilon}>0$. Thus, continuing \eqref{eq:fourier3}, we have a constant $C_{\gamma,\delta,\varepsilon}>0$ depending only on $\gamma,\delta,\varepsilon$ such that
\begin{align}\label{eq:fourier4}
\|u_N\|_p^{2-\delta}\leq C_{\gamma,\delta,\varepsilon}\cdot q^{N\(n\frac{\gamma-\delta}{2-\gamma}+2\(n-dr\frac{(2-\delta)\varepsilon}{\delta+\varepsilon}\)\frac{\delta+\varepsilon}{2+\varepsilon}\)}
\end{align}
holds for all $N\geq 0$.

Now suppose that the support of $u$ is contained in an analytic hypersurface $U$. For $N\geq 0$, put $U_N\coloneqq U+B^n_N\subseteq F^n$ as a tubular neighbourhood of $U$, which contains the support of $u_N$. Then for every $g\in\sS(F^n)$, we have
\begin{align}\label{eq:fourier5}
\lim_{N\to\infty}q^N\int_{U_N}g(x)\rd x = \int_{U}g(y)\rd y.
\end{align}
Then by H\"{o}lder's inequality, \eqref{eq:fourier4}, and \eqref{eq:fourier5}, we have
\begin{align*}
|\langle u,g\rangle|^{2-\delta}
&=\lim_{N\to\infty}|\langle u_N,g\rangle|^{2-\delta} \\
&\leq\lim_{N\to\infty}\|u_N\|_p^{2-\delta}\cdot\(\int_{U_N}|g(x)|^{2-\gamma}\rd x\)^{\frac{2-\delta}{2-\gamma}} \\
&\leq C_{\gamma,\delta,\varepsilon}\cdot\lim_{N\to\infty}
q^{N\(n\frac{\gamma-\delta}{2-\gamma}+2\(n-dr\frac{(2-\delta)\varepsilon}{\delta+\varepsilon}\)\frac{\delta+\varepsilon}{2+\varepsilon}
-\frac{2-\delta}{2-\gamma}\)}\cdot\(q^N\int_{U_N}|g(x)|^{2-\gamma}\rd x\)^{\frac{2-\delta}{2-\gamma}} \\
&=C_{\gamma,\delta,\varepsilon}\cdot\(\int_{U}|g(y)|^{2-\gamma}\rd y\)^{\frac{2-\delta}{2-\gamma}}\cdot\lim_{N\to\infty}
q^{N\(n\frac{\gamma-\delta}{2-\gamma}+2\(n-dr\frac{(2-\delta)\varepsilon}{\delta+\varepsilon}\)\frac{\delta+\varepsilon}{2+\varepsilon}
-\frac{2-\delta}{2-\gamma}\)}.
\end{align*}
Choose suitable $\gamma,\delta,\varepsilon$ such that
\[
r\frac{(2-\delta)\varepsilon}{\delta+\varepsilon}<\rho_\Delta,\qquad
n\frac{\gamma-\delta}{2-\gamma}+2\(n-dr\frac{(2-\delta)\varepsilon}{\delta+\varepsilon}\)\frac{\delta+\varepsilon}{2+\varepsilon}
<\frac{2-\delta}{2-\gamma}.
\]
Then the above limit is zero, that is, $\langle u,g\rangle=0$ for every $g\in\sS(F^n)$. Thus, we have $u=0$. The lemma is proved.
\end{proof}

\section{Remarks on Beilinson's non-archimedean local indices}
\label{ss:beilinson}

In this appendix, we review Beilinson's notion of non-archimedean local indices between algebraic cycles \cite{Bei87} and make some complementary remarks.

Let $K$ be a non-archimedean local field, with the ring of integers $O_K$ and the residue field $k$. Take a rational prime $\ell$ that is invertible on $k$. Let $X$ be a smooth projective scheme over $K$ of pure dimension $n-1$. For every integer $d\geq 0$, we have the cycle class map
\[
\cl_{X,\ell}\colon\rZ^d(X)\to\rH^{2d}(X,\dQ_\ell(d)),
\]
whose kernel we denote by $\rZ^d(X)^{\langle\ell\rangle}$.

\begin{remark}
A priori, $\rZ^d(X)^{\langle\ell\rangle}$ depends on the rational prime $\ell$. However, if $K$ is of characteristic zero and we assume the monodromy--weight conjecture for $X$, then one can replace $\cl_{X,\ell}$ by the geometric cycle class map, hence $\rZ^d(X)^{\langle\ell\rangle}$ does not depend on $\ell$.
\end{remark}

For a Zariski closed subset $Z$ of $X$, we denote by $\rZ^d_Z(X)$ the subgroup of $\rZ^d(X)$ consisting of cycles whose support is contained in $Z$,

\begin{definition}
For every pair of integers $d_1,d_2\geq 0$ satisfying $d_1+d_2=n$, we define the subgroups
\begin{align*}
\rZ^{d_1,d_2}(X)&\coloneqq
\sum_{Z_1,Z_2}\rZ^{d_1}_{Z_1}(X)_\dC\times\rZ^{d_2}_{Z_2}(X)_\dC\subseteq\rZ^{d_1}(X)_\dC\times\rZ^{d_2}(X)_\dC, \\
\rZ^{d_1,d_2}(X)^{\langle\ell\rangle}&\coloneqq
\sum_{Z_1,Z_2}(\rZ^{d_1}_{Z_1}(X)_\dC\cap\rZ^{d_1}(X)^{\langle\ell\rangle}_\dC)
\times(\rZ^{d_2}_{Z_2}(X)_\dC\cap\rZ^{d_2}(X)^{\langle\ell\rangle}_\dC)
\subseteq\rZ^{d_1}(X)_\dC\times\rZ^{d_2}(X)_\dC,
\end{align*}
where the sum is taken over all pairs $(Z_1,Z_2)$ of disjoint Zariski closed subsets of $X$. It is clear that $\rZ^{d_1,d_2}(X)^{\langle\ell\rangle}$ is stable under switching the two factors.
\end{definition}

Take a pair of integers $d_1,d_2\geq 0$ satisfying $d_1+d_2=n$. In \cite{Bei87}*{Section~2}, Beilinson defined a map
\begin{align}\label{eq:index}
\langle\;,\;\rangle_{X,K}^\ell\colon \rZ^{d_1,d_2}(X)^{\langle\ell\rangle} \to \dC\otimes_\dQ\dQ_\ell
\end{align}
called \emph{local index}, satisfying the following properties
\begin{itemize}
  \item its restriction to every subspace $(\rZ^{d_1}_{Z_1}(X)_\dC\cap\rZ^{d_1}(X)^{\langle\ell\rangle}_\dC)\times(\rZ^{d_2}_{Z_2}(X)_\dC\cap\rZ^{d_2}(X)^{\langle\ell\rangle}_\dC)$ is complex linear in the first variable;

  \item $\langle\;,\;\rangle_{X,K}^\ell$ is conjugate symmetric.
\end{itemize}

We briefly recall the definition. Take a pair $(c_1,c_2)\in\rZ^{d_1,d_2}(X)^{\langle\ell\rangle}$. By linearity, we may assume $c_1\in\rZ^{d_1}_{Z_1}(X)$ and $c_2\in\rZ^{d_2}_{Z_2}(X)$ with $Z_1\cap Z_2=\emptyset$. For $i=1,2$, put $U_i\coloneqq X\setminus Z_i$. Then we have the refined cycle class $\cl_{X,\ell}^{Z_i}(c_i)\in\rH^{2d_i}_{Z_i}(X,\dQ_\ell(d_i))$, which goes to $0$ under the natural map $\rH^{2d_i}_{Z_i}(X,\dQ_\ell(d_i))\to\rH^{2d_i}(X,\dQ_\ell(d_i))$. Thus, we can choose a class $\gamma_i\in\rH^{2d_i-1}(U_i,\dQ_\ell(d_i))$ that goes to $\cl_{X,\ell}^{Z_i}(c_i)$ under the coboundary map $\rH^{2d_i-1}(U_i,\dQ_\ell(d_i))\to\rH^{2d_i}_{Z_i}(X,\dQ_\ell(d_i))$. Then we define $\langle c_1,c_2\rangle_{X,K}^\ell$ to be the image of $\gamma_1\cup\gamma_2$ under the composite map
\[
\rH^{2n-2}(U_1\cap U_2,\dQ_\ell(n))\to\rH^{2n-1}(X,\dQ_\ell(n))\xrightarrow{\Tr_X}\rH^1(\Spec K,\dQ_\ell(1))=\dQ_\ell,
\]
in which the first map is the coboundary in the Mayer--Vietoris exact sequence for the covering $X=U_1\cup U_2$. It is easy to check that $\langle c_1,c_2\rangle_{X,K}^\ell$ does not depend on the choices of $\gamma_1,\gamma_2$, and that $\langle c_1,c_2\rangle_{X,K}^\ell=\langle c_2,c_1\rangle_{X,K}^\ell$.

\begin{lem}\label{le:base_change}
Take a pair $(c_1,c_2)\in\rZ^{d_1,d_2}(X)^{\langle\ell\rangle}$.
\begin{enumerate}
  \item Let $K'$ be a finite extension of $K$. Put $X'\coloneqq X\otimes_KK'$ regarded as a scheme over $K'$. Then we have $(c'_1,c'_2)\in\rZ^{d_1,d_2}(X')^{\langle\ell\rangle}$ and $\langle c'_1,c'_2\rangle_{X',K'}^\ell=\langle c_1,c_2\rangle_{X,K}^\ell$, where $c'_i$ is the restriction of $c_i$ on $X'$ for $i=1,2$.

  \item Let $u\colon X'\to X$ be a finite \'{e}tale morphism. Then we have $(c'_1,c'_2)\in\rZ^{d_1,d_2}(X')^{\langle\ell\rangle}$ and  $\langle c'_1,c'_2\rangle_{X',K}^\ell=\deg u\cdot\langle c_1,c_2\rangle_{X,K}^\ell$, where $c'_i$ is the restriction of $c_i$ on $X'$ for $i=1,2$.
\end{enumerate}
\end{lem}

\begin{proof}
In both statements, it is clear that $(c'_1,c'_2)\in\rZ^{d_1,d_2}(X')^{\langle\ell\rangle}$.

Part (1) follows from the following commutative diagram
\[
\xymatrix{
\rH^{2n-2}(U_1\cap U_2,\dQ_\ell(n)) \ar[r]\ar[d] & \rH^{2n-1}(X,\dQ_\ell(n)) \ar[r]^-{\Tr_X}\ar[d] & \rH^1(\Spec K,\dQ_\ell(1))
\ar@{=}[r]\ar[d] & \dQ_\ell \ar@{=}[d] \\
\rH^{2n-2}(U'_1\cap U'_2,\dQ_\ell(n)) \ar[r] & \rH^{2n-1}(X',\dQ_\ell(n)) \ar[r]^-{\Tr_{X'}} & \rH^1(\Spec K',\dQ_\ell(1))
\ar@{=}[r] & \dQ_\ell
}
\]
in which $U'_i$ is the restriction of $U_i$ on $X'$, and the construction of the local index.

Part (2) follows from the following commutative diagram
\[
\xymatrix{
\rH^{2n-2}(U_1\cap U_2,\dQ_\ell(n)) \ar[r]\ar[d]^-{u^*} & \rH^{2n-1}(X,\dQ_\ell(n)) \ar[r]^-{\Tr_X}\ar[d]^-{u^*} & \rH^1(\Spec K,\dQ_\ell(1)) \ar[d]^-{\deg u\cdot\id}\\
\rH^{2n-2}(U'_1\cap U'_2,\dQ_\ell(n)) \ar[r] & \rH^{2n-1}(X',\dQ_\ell(n)) \ar[r]^-{\Tr_{X'}} & \rH^1(\Spec K,\dQ_\ell(1))
}
\]
in which $U'_i$ is the restriction of $U_i$ on $X'$, and the construction of the local index.
\end{proof}

In what follows, $\pi\colon\cX\to\Spec O_K$ is a projective morphism such that $\cX\otimes_{O_K}K=X$. We put $Y\coloneqq\cX\otimes_{O_K}k$.

\begin{lem}\label{le:beilinson}
Consider elements $c_1\in\rZ^{d_1}_{Z_1}(X)$ and $c_2\in\rZ^{d_2}_{Z_2}(X)$ with $Z_1\cap Z_2=\emptyset$. For every $\beta_1\in\rH^{2d_1}_{Y\cup Z_1}(\cX,\dQ_\ell(d_1))$ whose image in $\rH^{2d_1}_{Z_1}(X,\dQ_\ell(d_1))$ coincides with $\cl_{X,\ell}^{Z_1}(c_1)$ and whose image in $\rH^{2d_1}(\cX,\dQ_\ell(d_1))$ vanishes, and every $\beta_2\in\rH^{2d_2-2(n-1)}_{Y\cup Z_2}(\cX,\pi^!\dQ_\ell(d_2-n+1))$ whose image in $\rH^{2d_2-2(n-1)}_{Z_2}(X,\pi^!\dQ_\ell(d_2-n+1))=\rH^{2d_2}_{Z_2}(X,\dQ_\ell(d_2))$ coincides with $\cl_{X,\ell}^{Z_2}(c_2)$, the image of $\beta_1\cup\beta_2\in\rH^2_Y(\cX,\pi^!\dQ_\ell(1))$ under the trace map $\rH^2_Y(\cX,\pi^!\dQ_\ell(1))\to\rH^2_{\Spec k}(\Spec O_K,\dQ_\ell(1))=\dQ_\ell$ coincides with $\langle c_1,c_2\rangle_{X,K}^\ell$.
\end{lem}

This is claimed in \cite{Bei87}*{Lemma-definition~2.1.1} without proof. For completeness, we include a proof here (though straightforward).

\begin{proof}
Before the proof, let us make a remark on cup products. Let $S$ be a Noetherian scheme on which $\ell$ is invertible. Given $F,G,H\in\rD^b(S,\dQ_\ell)$, the bounded derived category of $\ell$-adic sheaves on $S$, together with a map
\[
\kappa\colon F\overset{\dL}\otimes G\to H,
\]
we have a cup product map
\[
\cup_\kappa\colon\rH^i(S,F)\times\rH^j(S,G)\to\rH^{i+j}(S,H)
\]
for every integers $i,j$, which is the composition of the cup product for (hyper)cohomology
\[
\rH^i(S,F)\times\rH^j(S,G)\to\rH^{i+j}(S,F\overset{\dL}\otimes G)
\]
and the induced map $\rH^{i+j}(S,\kappa)\colon\rH^{i+j}(S,F\overset{\dL}\otimes G)\to\rH^{i+j}(S,H)$. In particular, if we have maps $f\colon F\to F'$ and $h\colon H\to H'$ rendering the following diagram
\[
\xymatrix{
F\overset{\dL}\otimes G \ar[r]^-{\kappa}\ar[d]_-{f\otimes\id} & H \ar[d]^-{h} \\
F'\overset{\dL}\otimes G \ar[r]^-{\kappa'} & H'
}
\]
commutative (in $\rD^b(S,\dQ_\ell)$), then the induced diagram
\[
\xymatrix{
\rH^i(S,F)\times\rH^j(S,G) \ar[r]^-{\cup_\kappa}\ar[d]_-{\rH^i(S,f)\times\id} & \rH^{i+j}(S,H) \ar[d]^-{\rH^{i+j}(S,h)} \\
\rH^i(S,F')\times\rH^j(S,G) \ar[r]^-{\cup_{\kappa'}} & \rH^{i+j}(S,H')
}
\]
commutes.

Put $U_i\coloneqq X\setminus Z_i$ for $i=1,2$ as before. For $i=1,2$, choose a class $\gamma_i\in\rH^{2d_i-1}(U_i,\dQ_\ell(d_i))$ that goes to $\cl_{X,\ell}^{Z_i}(c_i)$ under the coboundary map $\rH^{2d_i-1}(U_i,\dQ_\ell(d_i))\to\rH^{2d_i}_{Z_i}(X,\dQ_\ell(d_i))$. Denote by $\langle\gamma_1,c_2\rangle$ to be the image of $\gamma_1\cup\cl_{X,\ell}^{Z_2}(c_2)\in\rH^{2n-1}_{Z_2}(U_1,\dQ_\ell(n))=\rH^{2n-1}_{Z_2}(X,\dQ_\ell(n))$ under the composite map
\[
\rH^{2n-1}_{Z_2}(X,\dQ_\ell(n))\to\rH^{2n-1}(X,\dQ_\ell(n))\xrightarrow{\Tr_X}\rH^1(\Spec K,\dQ_\ell(1))=\dQ_\ell.
\]
We break the proof into two steps.
\begin{enumerate}
  \item $\langle\gamma_1,c_2\rangle=\langle c_1,c_2\rangle_{X,K}^\ell$;

  \item the image of $\beta_1\cup\beta_2\in\rH^2_Y(\cX,\pi^!\dQ_\ell(1))$ under the trace map $\rH^2_Y(\cX,\pi^!\dQ_\ell(1))\to\dQ_\ell$ coincides with $\langle\gamma_1,c_2\rangle$.
\end{enumerate}

For (1), it is easy to see that the coboundary map $\rH^{2n-2}(U_1\cap U_2,\dQ_\ell(n))\to\rH^{2n-1}(X,\dQ_\ell(n))$ in the Mayer--Vietoris exact sequence is the composition of the coboundary map $\delta\colon\rH^{2n-2}(U_1\cap U_2,\dQ_\ell(n))\to\rH^{2n-1}_{Z_2}(U_1,\dQ_\ell(n))$ in the Gysin sequence and the natural map $\rH^{2n-1}_{Z_2}(U_1,\dQ_\ell(n))=\rH^{2n-1}_{Z_2}(X,\dQ_\ell(n))\to\rH^{2n-1}(X,\dQ_\ell(n))$. Thus, it suffices to show that the following diagram
\begin{align}\label{eq:beilinson1}
\xymatrix{
\rH^{2d_1-1}(U_1,\dQ_\ell(d_1))\times\rH^{2d_2-1}(U_2,\dQ_\ell(d_2)) \ar[r]\ar[d]_-{\id\times\delta}
& \rH^{2n-2}(U_1\cap U_2,\dQ_\ell(n)) \ar[d]^-{\delta} \\
\rH^{2d_1-1}(U_1,\dQ_\ell(d_1))\times\rH^{2d_2}_{Z_2}(X,\dQ_\ell(d_2)) \ar[r]
& \rH^{2n-1}_{Z_2}(U_1,\dQ_\ell(n))
}
\end{align}
commutes. Denote morphisms $\iota_i\colon U_i\to X$ and $\jmath_i\colon Z_i\to X$ for $i=1,2$, and $\iota\colon U_1\cap U_2\to X$. Then in view of the remark on the cup products, the first row of \eqref{eq:beilinson1} is induced by the natural map
\[
\iota_{1*}\iota_1^*\dQ_\ell(d_1)\overset{\dL}\otimes\iota_{2*}\iota_2^*\dQ_\ell(d_2)\to\iota_*\iota^*\dQ_\ell(n),
\]
and the second row of \eqref{eq:beilinson1} is induced by the map
\[
\iota_{1*}\iota_1^*\dQ_\ell(d_1)\overset{\dL}\otimes\jmath_{2!}\jmath_2^!\dQ_\ell(d_2)\to\jmath_{2!}\jmath_2^!\iota_1^*\dQ_\ell(n)
\]
which is the cone (of columns) of the natural commutative diagram
\[
\xymatrix{
\iota_{1*}\iota_1^*\dQ_\ell(d_1)\overset{\dL}\otimes\dQ_\ell(d_2) \ar[r]\ar[d]& \iota_{1*}\iota_1^*\dQ_\ell(n) \ar[d] \\
\iota_{1*}\iota_1^*\dQ_\ell(d_1)\overset{\dL}\otimes\iota_{2*}\iota_2^*\dQ_\ell(d_2) \ar[r]& \iota_*\iota^*\dQ_\ell(n)
}
\]
in $\rD^b(X,\dQ_\ell)$. It follows that \eqref{eq:beilinson1} commutes. In particular, $\langle c_1,c_2\rangle_{X,K}^\ell$ does not depend on the choices of $\gamma_1$ and $\gamma_2$, which justifies the notation.

For (2), we may assume that $\beta_1$ is the coboundary of $\gamma_1$. We have the commutative diagram
\[
\xymatrix{
\rH^1(X,\pi^!\dQ_\ell(1)) \ar[r]^-{\Tr_X}\ar[d] & \rH^1(\Spec K,\dQ_\ell(1)) \ar[d] \\
\rH^2_Y(\cX,\pi^!\dQ_\ell(1)) \ar[r]^-{\Tr_\cX} & \rH^2_{\Spec k}(\Spec O_K,\dQ_\ell(1))
}
\]
for the trace maps. Thus, as $\rH^1(X,\pi^!\dQ_\ell(1))\simeq\rH^{2n-1}(X,\dQ_\ell(n))$, it remains to show that the following diagram
\begin{align*}
\xymatrix{
\rH^{2d_1-1}(U_1,\dQ_\ell(d_1))\times\rH^{2d_2-2(n-1)}_{Y\cup Z_2}(\cX,\pi^!\dQ_\ell(d_2-n+1)) \ar[r]\ar[d]_-{\delta\times\id}
& \rH^1_{Z_2}(U_1,\pi^!\dQ_\ell(1)) \ar[d]^-{\delta} \\
\rH^{2d_1}_{Y\cup Z_1}(\cX,\dQ_\ell(d_1))\times\rH^{2d_2-2(n-1)}_{Y\cup Z_2}(\cX,\pi^!\dQ_\ell(d_2-n+1)) \ar[r]
& \rH^2_Y(\cX,\pi^!\dQ_\ell(1))
}
\end{align*}
commutes. The argument is similar to (1), which we leave to readers.

The lemma is proved.
\end{proof}

\begin{remark}\label{re:beilinson}
In Lemma \ref{le:beilinson}, when $\cX$ is regular, the natural map $\pi^!\dQ_\ell[2-2n](1-n)\to\dQ_\ell$ is an isomorphism, which is a consequence of the absolute purity theorem \cite{Fuj02}.
\end{remark}

Now we provide a refined method to compute \eqref{eq:index} in the presence of a regular model of $X$. Till the end of this section, $\cX$ will be regular.

We first review some constructions from \cite{GS87}. For every Zariski closed subset $\cZ$ of $\cX$, we have the K-group $\rK_0^\cZ(\cX)$ of complexes with support in $\cZ$ defined in \cite{GS87}*{Section~1.1}, equipped with the codimension filtration
\[
\cdots \supset \rF^{d-1}\rK_0^\cZ(\cX) \supset \rF^d\rK_0^\cZ(\cX) \supset \rF^{d+1}\rK_0^\cZ(\cX) \supset \cdots.
\]
We have
\begin{itemize}
  \item the pushforward map $\pi_*\colon\rK_0^Y(\cX)\to\rK_0^{\Spec k}(\Spec O_K)=\rK_0(\Spec k)$;

  \item for $\cZ'\subseteq\cZ$, a natural linear map $\rK_0^{\cZ'}(\cX)\to\rK_0^\cZ(\cX)$ which preserves the codimension filtration;

  \item a cup-product map $\cup\colon\rK_0^{\cZ_1}(\cX)\times\rK_0^{\cZ_2}(\cX)\to\rK_0^{\cZ_1\cap\cZ_2}(\cX)$;

  \item a natural linear map
     \begin{align}\label{eq:kgroup}
     [\phantom{\cZ}]\colon\bigoplus_{d'\geq d}\rZ^{d'}_\cZ(\cX)\to\rF^d\rK_0^\cZ(\cX)
     \end{align}
     sending a closed subscheme $\cZ'$ of $\cX$ contained in $\cZ$ to the class of the structure sheaf $\sO_{\cZ'}$.
\end{itemize}
See \cite{GS87}*{Section~1 \& Section~5} for more details.

Note that since $\cX$ is regular, $\rK_0^\cZ(\cX)$ coincides with Quillen's K-theory with support (see the proof of \cite{GS87}*{Theorem~8.2}). Then by \cite{Gil81}*{Definition~2.34(ii)} in which we take the base scheme $S$ to be $\Spec O_K$ and $\Gamma$ to be the $\ell$-adic cohomology theory, we obtain the $d$-th Chern class map
\[
\cl_{\cX,\ell}^\cZ\colon\rF^d\rK_0^\cZ(\cX)\to\rH^{2d}_\cZ(\cX,\dQ_\ell(d))
\]
for every integer $d\geq 0$.

For the generic fiber $X$, we have K-groups the $\rK_0^Z(X)$ for Zariski closed subsets $Z$ of $X$ with similar properties as well. The following lemma is probably known, but we can not find an exact reference.

\begin{lem}\label{re:chern}
For every $c\in\rZ^d_Z(X)$, the element $\cl_{X,\ell}^Z([c])\in\rH^{2d}_Z(X,\dQ_\ell(d))$ coincides with the refined cycle class $\cl_{X,\ell}^Z(c)$ of $c$.
\end{lem}

\begin{proof}
We may assume $Z$ irreducible and $c=Z$. Let $Z'$ be the smooth locus of $Z$ over $K$ and put $X'\coloneqq X\setminus(Z\setminus Z')$. As a consequence of the semi-purity theorem \cite{Fuj02}*{Section~8}, the restriction map $\rH^{2d}_Z(X,\dQ_\ell(d))\to\rH^{2d}_{Z'}(X',\dQ_\ell(d))$ is an isomorphism. Thus, we may assume $Z$ smooth over $K$ as well. Then the lemma follows from \cite{Gil81}*{Theorem~3.1} (with $S=\Spec K$ and $k=q=0$).
\end{proof}

\begin{definition}
Let $\cZ_1$ and $\cZ_2$ be two Zariski closed subsets $\cZ$ of $\cX$ satisfying $\cZ_1\cap\cZ_2\subseteq Y$. We define a pairing
\begin{align*}
\rF^{d_1}\rK_0^{\cZ_1}(\cX)_\dC\times\rF^{d_2}\rK_0^{\cZ_2}(\cX)_\dC&\to\dC \\
(\cC_1,\cC_2)&\mapsto \cC_1.\cC_2
\end{align*}
that is complex linear in the first variable, conjugate complex linear in the second variable, and such that for $\cC_i\in\rF^{d_i}\rK_0^{\cZ_i}(\cX)$ with $i=1,2$, we have
\[
\cC_1.\cC_2\coloneqq\chi\(\pi_*(\cC_1\cup\cC_2)\),
\]
where $\chi$ denotes the Euler characteristic function on $\rK_0(\Spec k)$. Note that as $\cZ_1\cap\cZ_2\subseteq Y$, $\cC_1\cup\cC_2$ can be regarded as element in $\rK_0^Y(\cX)_\dC$.
\end{definition}

\begin{lem}\label{pr:intersection}
Let $\cZ_1$ and $\cZ_2$ be two Zariski closed subsets of $\cX$ satisfying $\cZ_1\cap\cZ_2\subseteq Y$. For $\cC_i\in\rF^{d_i}\rK_0^{\cZ_i}(\cX)$ with $i=1,2$, $\cC_1.\cC_2$ coincides with the image of $\cl^{\cZ_1}_{\cX,\ell}(\cC_1)\cup\cl^{\cZ_2}_{\cX,\ell}(\cC_2)\in\rH^{2n}_{\cZ_1\cap\cZ_2}(\cX,\dQ_\ell(n))$ under the natural composite map
\begin{align*}
\rH^{2n}_{\cZ_1\cap\cZ_2}(\cX,\dQ_\ell(n))\to\rH^{2n}_{Y}(\cX,\dQ_\ell(n))
\xrightarrow{\pi_*}\rH^2_{\Spec k}(\Spec O_K,\dQ_\ell(1))=\rH^0(\Spec k,\dQ_\ell)=\dQ_\ell.
\end{align*}
\end{lem}

\begin{proof}
By \cite{GS87}*{Proposition~5.5}, we have $\cC_1\cup\cC_2\in\rF^n\rK_0^{\cZ_1\cap\cZ_2}(\cX)_\dQ$. By \cite{Gil81}*{Proposition~2.35}, $\cl^{\cZ_1}_{\cX,\ell}(\cC_1)\cup\cl^{\cZ_2}_{\cX,\ell}(\cC_2)$ and $\cl^{\cZ_1\cap\cZ_2}_{\cX,\ell}(\cC_1\cup\cC_2)$ have the same image under the natural map $\rH^{2n}_{\cZ_1\cap\cZ_2}(\cX,\dQ_\ell(n))\to\rH^{2n}_Y(\cX,\dQ_\ell(n))$. Since the map $\rZ^n(\cX)_\dQ\to\rF^n\rK_0^Y(\cX)_\dQ$ is surjective, the diagram
\begin{align*}
\xymatrix{
\rF^n\rK_0^Y(\cX)_\dQ \ar[rr]^-{\cl_{\cX,\ell}^Y}\ar[d]_-{\pi_*} && \rH^{2n}_Y(\cX,\dQ_\ell(n)) \ar[d]^-{\pi_*} \\
\rF^1\rK_0^{\Spec k}(\Spec O_K)_\dQ \ar[rr]^-{\cl_{\Spec O_K,\ell}^{\Spec k}} && \rH^2_{\Spec k}(\Spec O_K,\dQ_\ell(1))
}
\end{align*}
commutes. Thus, the proposition follows since the diagram
\begin{align*}
\xymatrix{
\rF^1\rK_0^{\Spec k}(\Spec O_K)_\dQ \ar[rr]^-{\cl_{\Spec O_K,\ell}^{\Spec k}}\ar[d]_-{=} && \rH^2_{\Spec k}(\Spec O_K,\dQ_\ell(1))\ar[d]^-{=} \\
\rK_0(\Spec k)_\dQ \ar[rr]^-{\cl_{\Spec k,\ell}}\ar[d]_-{\chi} && \rH^0(\Spec k,\dQ_\ell) \ar[d]^-{=}\\
\dQ \ar[rr] && \dQ_\ell
}
\end{align*}
commutes.
\end{proof}

\if false

\begin{corollary}
In the situation of Lemma \ref{pr:intersection}, if $\cC_1$ satisfies that the image of $\cl^{\cZ_1}_{\cX,\ell}(\cC_1)$ in $\rH^{2d_1}(\cX,\dQ_\ell(d_1))$ vanishes, then we have $\cC_1.\cC_2=0$ for every $\cC_2\in\rF^{d_2}\rK_0^Y(\cX)$.
\end{corollary}

\begin{proof}
It follows from Proposition \ref{pr:intersection} and the observation that the cup-product map
\[
\rH^{2d_1}_{\cZ_1}(\cX,\dQ_\ell(d_1))\times\rH^{2d_2}_Y(\cX,\dQ_\ell(d_2))\to\rH^{2n}_{Y}(\cX,\dQ_\ell(n))
\]
factors through $\rH^{2d_1}(\cX,\dQ_\ell(d_1))\times\rH^{2d_2}_Y(\cX,\dQ_\ell(d_2))$.
\end{proof}

\fi

\begin{definition}\label{de:extension}
For an element $c\in\rZ^d(X)_\dC$, we say that an element $\cC\in\rF^d\rK_0^{Y\cup\supp(c)}(\cX)_\dC$ is an \emph{extension} of $c$ if $\cC\res_X\in\rF^d\rK_0(X)_\dC$ coincides with $[c]$ under the map \eqref{eq:kgroup}, and that $\cC$ is an \emph{$\ell$-flat extension} if the image of $\cl_{\cX,\ell}^{Y\cup\supp(c)}(\cC)$ in $\rH^{2d}(\cX,\dQ_\ell(d))\otimes_\dQ\dC$ vanishes.
\end{definition}

\begin{proposition}\label{pr:index_smooth}
Consider a pair $(c_1,c_2)\in\rZ^{d_1,d_2}(X)^{\langle\ell\rangle}$ satisfying $\supp(c_i)\cap\supp(c_2)=\emptyset$. If $\cZ_1$ and $\cZ_2$ are two Zariski closed subsets of $\cX$ satisfying $\cZ_1\cap\cZ_2\subseteq Y$, and $\cC_i\in\rF^{d_i}\rK_0^{\cZ_i}(\cX)_\dC$ is an extension of $c_i$ for $i=1,2$ in which at least one is $\ell$-flat, then we have
\[
\langle c_1,c_2\rangle_{X,K}^\ell=\cC_1.\cC_2.
\]
In particular, when $\pi$ is smooth, we can take $\cC_i$ to be the one given by the Zariski closure of $c_i$ in $\cX$ via \eqref{eq:kgroup}, hence $\langle c_1,c_2\rangle_{X,K}^\ell$ belongs to $\dC$ and is independent of $\ell$.
\end{proposition}

\begin{proof}
By Lemma \ref{re:chern}, for $i=1,2$, the image of $\cl_{\cX,\ell}^{\cZ_i}(\cC_i)$ in $\rH^{2d_i}_{\cZ_i\cap X}(X,\dQ_\ell(d_i))\otimes_\dQ\dC$ coincides with $\cl_{X,\ell}^{\cZ_i\cap X}(c_i)$. Without lost of generality, we assume that $\cC_1$ is $\ell$-flat. Then the lemma follows from Lemma \ref{pr:intersection}, and Lemma \ref{le:beilinson} with $\beta_i\coloneqq\cl_{\cX,\ell}^{\cZ_i}(\cC_i)$ for $i=1,2$ (together with Remark \ref{re:beilinson}).
\end{proof}

\begin{remark}\label{re:index_smooth}
Suppose that $X$ admits smooth projective reduction over $O_K$. Then the source of \eqref{eq:index} is independent of $\ell$. Moreover, by Proposition \ref{pr:index_smooth}, the map \eqref{eq:index} takes value in $\dC$ and is independent of $\ell$. Thus, in this case, \eqref{eq:index} makes sense for an arbitrary rational prime $\ell$ of which it is independent.
\end{remark}

In the remaining discussion, we only consider the case where $n=2r$ for some integer $r\geq 1$, and $d_1=d_2=r$. We say that a correspondence
\[
t\colon \cX \xleftarrow{p} \cX' \xrightarrow{q} \cX
\]
of $\cX$ is \emph{\'{e}tale} if both $p$ and $q$ are finite \'{e}tale. In what follows, we take a subfield $\dL$ of $\dC$.

\begin{definition}\label{de:tempered_integral}
We say that an $\dL$-\'{e}tale correspondence $t$, that is, an $\dL$-linear combination of \'{e}tale correspondences, of $\cX$ is \emph{$\ell$-tempered} if $t^*$ annihilates $\rH^{2r}(\cX,\dQ_\ell(r))\otimes_\dQ\dL$.
\end{definition}

\begin{proposition}\label{pr:tempered_integral}
Let $t$ be an $\ell$-tempered $\dL$-\'{e}tale correspondence of $\cX$. Then for every pair $(c_1,c_2)\in\rZ^r(X)_\dC\times\rZ^r(X)_\dC$ satisfying $\supp(t^*c_1)\cap\supp(t^*c_2)=\emptyset$, we have $(t^*c_1,t^*c_2)\in\rZ^{r,r}(X)^{\langle\ell\rangle}$ and
\[
\langle t^*c_1,t^*c_2\rangle_{X,K}^\ell=t^*\cC_1.t^*\cC_2,
\]
where $\cC_i\in\rF^r\rK_0^{Y\cup\supp(c_i)}(\cX)_\dC$ is an arbitrary extension of $c_i$ in $\cX$ for $i=1,2$. In particular, we have $\langle t^*c_1,t^*c_2\rangle_{X,K}^\ell\in\dC$.
\end{proposition}

\begin{proof}
For $i=1,2$, put $Z_i\coloneqq\supp(c_i)$, $Z_i^t\coloneqq\supp(t^*c_i)$, and
\[
\beta_i\coloneqq\cl_{\cX,\ell}^{Y\cup Z_i}(\cC_i)\in\rH^{2r}_{Y\cup Z_i}(\cX,\dQ_\ell(r))\otimes_\dQ\dC.
\]
Note that we have the commutative diagram
\[
\xymatrix{
\rH^{2r}_{Y\cup Z_i}(\cX,\dQ_\ell(r))\otimes_\dQ\dL \ar[r] \ar[d]_-{t^*} & \rH^{2r}(\cX,\dQ_\ell(r))\otimes_\dQ\dL \ar[d]^-{t^*} \\
\rH^{2r}_{Y\cup Z_i^t}(\cX,\dQ_\ell(r))\otimes_\dQ\dL \ar[r] & \rH^{2r}(\cX,\dQ_\ell(r))\otimes_\dQ\dL
}
\]
induced by $t$. Since $t$ is $\ell$-tempered, the image of $t^*\beta_i$ in $\rH^{2r}(\cX,\dQ_\ell(r))\otimes_\dQ\dC$ vanishes. Now for $i=1,2$, since $t^*\beta_i=\cl_{\cX,\ell}^{Y\cup Z_i^t}(t^*\cC_i)$, we know that $t^*\cC_i$ is an $\ell$-flat extension of $t^*c_i$. In particular, we have $(t^*c_1,t^*c_2)\in\rZ^{r,r}(X)^{\langle\ell\rangle}$. Finally, the formula for $\langle t^*c_1,t^*c_2\rangle_{X,K}^\ell$ follows from Proposition \ref{pr:index_smooth}.
\end{proof}

\if false

To compute $\langle t^*c_1,t^*c_2\rangle_{X,K}^\ell$, we may assume that $\supp(t^*c_1)\cap\supp(t^*c_2)=\emptyset$. We would like to apply the equivalence between (b) and (c) in \cite{Bei87}*{Lemma-definition~2.1.1}. To do this, we add the following remark for $t^*\beta_2$. Let $\pi\colon\cX\to\Spec O_K$ be the structure morphism, which is proper and flat of pure relative dimension $2r-1$. Thus, we have the trace map $\pi_!\dQ_{\ell,\cX}[2(2r-1)](2r-1)\to\dQ_\ell$, which induces a map
\[
\dQ_{\ell,\cX}\to\pi^!\dQ_\ell[2(1-2r)](1-2r)
\]
of $\dQ_\ell$-sheaves on $\cX$ by adjunction, hence a map
\[
\rH^{2r}_{Y\cup\supp(t^*c_2)}(\cX,\dQ_\ell(r))\to \rH^{2-2r}_{Y\cup\supp(t^*c_2)}(\cX,\pi^!\dQ_\ell(1-r))
\]
on cohomology. We regard $t^*\beta_2$ as a class in $\rH^{2-2r}_{Y\cup\supp(t^*c_2)}(\cX,\pi^!\dQ_\ell(1-r))$, whose restriction to $\rH^{2r}_{\supp(t^*c_2)}(X,\dQ_\ell(r))$ coincides with $\widetilde\cl_{X,\ell}(t^*c_2)$. By \cite{Bei87}*{Lemma-definition~2.1.1}, $\langle t^*c_1,t^*c_2\rangle_{X,K}^\ell$ coincides with the image of $t^*\beta_1\cup t^*\beta_2\in\rH^2_Y(\cX,\pi^!\dQ_\ell(1))$ under the trace map
\[
\rH^2_Y(\cX,\pi^!\dQ_\ell(1))\xrightarrow{\pi_!}\rH^2_{\Spec k}(\Spec O_K,\dQ_\ell(1))=\rH^0(\Spec k,\dQ_\ell)=\dQ_\ell.
\]
From this and Lemma \ref{pr:intersection}, it follows that $\langle t^*c_1,t^*c_2\rangle_{X,K}^\ell$ coincides with the intersection number of $t^*\cC_1$ and $t^*\cC_2$ on $\cX$.

\fi

Now we provide a criterion for an $\dL$-\'{e}tale correspondence to be $\ell$-tempered.

\begin{proposition}\label{pr:tempered}
Put $Y_0\coloneqq Y^{\r{red}}$, the induced reduced subscheme of $Y$. Suppose that we have a finite stratification $Y_0\supset Y_1 \supset \cdots$ of Zariski closed subsets such that $Y_j^\circ\coloneqq Y_j\setminus Y_{j+1}$ is regular and has pure codimension $n_j\geq 1$ in $\cX$ for $j\geq 0$. If $t$ is an $\dL$-\'{e}tale correspondence of $\cX$ stabilizing the stratification $Y_0\supset Y_1 \supset \cdots$ and such that
\begin{enumerate}
  \item $t^*$ annihilates $\rH^{2r}(X,\dQ_\ell(r))\otimes_\dQ\dL$;

  \item $t^*$ annihilates $\rH^i(Y_j^\circ\otimes_k\ol\dF_p,\dQ_\ell)\otimes_\dQ\dL$ for every integer $i\leq 2r-2n_j$ and every $j$,
\end{enumerate}
then some positive power of $t$ annihilates $\rH^{2r}(\cX,\dQ_\ell(r))\otimes_\dQ\dL$.
\end{proposition}

\begin{proof}
It suffices to prove that $(t^m)^*$ annihilates $\rH^{2r}_Y(\cX,\dQ_\ell(r))\otimes_\dQ\dL$ for some integer $m\geq 1$, since then $t^{m+1}$ is $\ell$-tempered.

We prove by decreasing induction on $j$ that $(t^{m_j})^*$ annihilates $\rH^{2r}_{Y_j}(\cX,\dQ_\ell(r))\otimes_\dQ\dL$ for some integer $m_j\geq 1$. We have
\[
\rH^{2r}_{Y_{j+1}}(\cX,\dQ_\ell(r))\otimes_\dQ\dL\to \rH^{2r}_{Y_j}(\cX,\dQ_\ell(r))\otimes_\dQ\dL
\to\rH^{2r}_{Y_j^\circ}(\cX\setminus Y_{j+1},\dQ_\ell(r))\otimes_\dQ\dL.
\]
As $Y_j^\circ$ is a regular closed subscheme of the regular scheme $\cX\setminus Y_{j+1}$, by the absolute purity theorem \cite{Fuj02}, we have
\[
\rH^{2r}_{Y_j^\circ}(\cX\setminus Y_{j+1},\dQ_\ell(r)) \simeq
\rH^{2r-2n_j}(Y_j^\circ,\dQ_\ell(r-n_j)).
\]
By condition (2) and the Hochschild--Serre spectral sequence, we know that $(t^2)^*$ annihilates $\rH^{2r-2n_j}(Y_j^\circ,\dQ_\ell(r-n_j))\otimes_\dQ\dL$. Thus, we may take $m_j=m_{j+1}+2$. In particular, $(t^{m_0})^*$ annihilates $\rH^{2r}_{Y_0}(\cX,\dQ_\ell(r))\otimes_\dQ\dL$, which is same as $\rH^{2r}_Y(\cX,\dQ_\ell(r))\otimes_\dQ\dL$.
\end{proof}

\begin{corollary}\label{co:tempered}
Let $\cX$ and $\dL$ be as above. Let $\dS$ be a ring of \'{e}tale correspondences of $\cX$, and $\fm$ a maximal ideal of $\dS_\dL$.
\begin{enumerate}
  \item If $(\rH^{2d}(\cX,\dQ_\ell(d))\otimes_\dQ\dL)_\fm=0$, then there exists an $\ell$-tempered element in $\dS_\dL\setminus\fm$.

  \item Suppose that we have a finite stratification $Y_0\supset Y_1 \supset \cdots$ of Zariski closed subsets that is stabilized by the action of $\dS$, such that $Y_j^\circ\coloneqq Y_j\setminus Y_{j+1}$ is regular and has pure codimension $n_j\geq 1$ in $\cX$ for $j\geq 0$. If
     \begin{itemize}
       \item $(\rH^{2r}(X,\dQ_\ell(r))\otimes_\dQ\dL)_\fm=0$ and

       \item $(\rH^i(Y_j^\circ\otimes_k\ol\dF_p,\dQ_\ell)\otimes_\dQ\dL)_\fm=0$ for every integer $i\leq 2r-2n_j$ and every $j$,
     \end{itemize}
      then $(\rH^{2d}(\cX,\dQ_\ell(d))\otimes_\dQ\dL)_\fm=0$.
\end{enumerate}
\end{corollary}

\begin{proof}
For (1), since $\rH^{2r}(\cX,\dQ_\ell(r))$ is of finite dimension over $\dQ_\ell$, $\rH^{2r}(\cX,\dQ_\ell(r))\otimes_\dQ\dL$ is a finitely generated $\dS_\dL$-module. Then (1) follows from Definition \ref{de:tempered_integral}.

For (2), since both $\rH^{2r}(X,\dQ_\ell(r))$ and $\bigoplus_{i<2r-1}\bigoplus_j\rH^i(Y_j^\circ\otimes_k\ol\dF_p,\dQ_\ell)$ are of finite dimension over $\dQ_\ell$, both $\rH^{2r}(X,\dQ_\ell(r))\otimes_\dQ\dL$ and $\bigoplus_{i<2r-1}\bigoplus_j\rH^i(Y_j^\circ\otimes_k\ol\dF_p,\dQ_\ell)\otimes_\dQ\dL$ are finitely generated $\dS_\dL$-modules. Then there exists $t\in\dS_\dL\setminus\fm$ satisfying the two conditions in Proposition \ref{pr:tempered}. By the same proposition, some power of $t$ annihilates $\rH^{2d}(\cX,\dQ_\ell(d))\otimes_\dQ\dL$, which implies $(\rH^{2d}(\cX,\dQ_\ell(d))\otimes_\dQ\dL)_\fm=0$. Thus, (2) follows.
\end{proof}


\end{document}